\documentclass[letterpaper,11pt]{amsart}

\usepackage[margin=1in]{geometry}
\usepackage{amscd, amssymb}
\usepackage{amsmath,amscd}
\usepackage[foot]{amsaddr}
\usepackage{mathtools}
\usepackage{url}
\usepackage[driverfallback=hypertex]{hyperref}
\usepackage{graphics}
\usepackage{graphicx}
\usepackage{caption}
\usepackage{subcaption}
\usepackage{color}
\usepackage{tabularx}
\usepackage{booktabs}
\usepackage{bm}
\usepackage{bbm}
\usepackage[shortlabels]{enumitem}
\usepackage[noabbrev,capitalize]{cleveref}
\setlist[enumerate,1]{label={(\alph*)}}
\setlist[enumerate,2]{label={(\roman*)}}

\usepackage{chessboard,skak}

\newif\ifdraft
\ifdraft
\usepackage{draftwatermark}
\SetWatermarkText{DRAFT}
\SetWatermarkScale{2}
\fi
\newtheorem{thm}{Theorem}[section]
\newtheorem{prop}[thm]{Proposition}

\newtheorem{lem}[thm]{Lemma}

\newtheorem{clm}[thm]{Claim}

\Crefname{clm}{Claim}{Claims}

\theoremstyle{definition}
\newtheorem{definition}[thm]{Definition}

\newtheorem*{notation*}{Notation}

\theoremstyle{remark}
\newtheorem{rmk}[thm]{Remark}
\newtheorem{ex}[thm]{Example}
\newtheorem{obs}[thm]{Observation}
\newtheorem{algorithm}[thm]{Algorithm}

\newcommand{\ignore}[1]{}
\newcommand{\R}{\mathbb R}

\newcommand{\N}{\mathbb N}
\newcommand{\Prob}{{\mathbb{P}}}

\newcommand{\Z}{\mathbb Z}

\newcommand{\mB}{{\mathcal{B}}}

\newcommand{\mA}{{\mathcal{A}}}
\newcommand{\mC}{{\mathcal{C}}}
\newcommand{\mD}{{\mathcal{D}}}
\newcommand{\mF}{{\mathcal{F}}}

\newcommand{\mP}{{\mathcal{P}}}
\newcommand{\mQ}{{\mathcal{Q}}}
\newcommand{\mR}{{\mathcal{R}}}
\newcommand{\mT}{{\mathcal{T}}}
\newcommand{\mU}{{\mathcal{U}}}

\newcommand{\E}{{\mathbb{E}}}

\newcommand{\tg}{{\tilde{g}}}
\newcommand{\ta}{{\tilde{a}}}
\newcommand{\tA}{{\tilde{A}}}
\newcommand{\tC}{{\tilde{C}}}

\newcommand{\tQ}{{\tilde{Q}}}
\newcommand{\tR}{{\tilde{R}}}
\newcommand{\tZ}{{\tilde{Z}}}
\newcommand{\tgamma}{{\tilde{\gamma}}}

\newcommand{\oone}{{o \left(1\right)}}
\newcommand{\oneoone}{{\left( 1 \pm \oone \right)}}

\newcommand{\termdefine}[1]{\textbf{#1}}

\newcommand{\norm}[1]{{\lVert {#1} \rVert}}
\newcommand{\distNoParam}{{d_{\diamond}}}
\newcommand{\dist}[2]{{d_{\diamond} \left( {#1},{#2} \right) }}
\newcommand{\distbwNoParam}{{d_{\mathbf{BW}}}}
\newcommand{\distbw}[2]{{\distbwNoParam \left( {#1},{#2} \right) }}
\newcommand{\Queenons}{{\Gamma}}
\newcommand{\interval}{{[-1/2,1/2]}}
\newcommand{\given}{{\mid}}

\newcommand{\codeURL}{https://arxiv.org/format/2107.13460}

\newcommand{\upperBoundConstant}{{1.94}}
\newcommand{\lowerBoundConstant}{{1.9449}}

\newcommand{\distPlus}[1]{{\overline{{#1}}^+}}
\newcommand{\distMinus}[1]{{\overline{{#1}}^-}}
\newcommand{\BWdistPlus}[1]{{\hat{{#1}}^+}}
\newcommand{\BWdistMinus}[1]{{\hat{{#1}}^-}}
\newcommand{\BW}{\mathbf{BW}}
\newcommand{\bw}[1]{{\BW \left( {#1} \right)}}

\begin{document}
\title{The number of $n$-queens configurations}

\author{Michael Simkin}
\address{Harvard University Center of Mathematical Sciences and Applications, Cambridge, MA, USA.}
\email{msimkin@cmsa.fas.harvard.edu}

\begin{abstract}
	The $n$-queens problem is to determine $\mathcal{Q}(n)$, the number of ways to place $n$ mutually non-threatening queens on an $n \times n$ board. We show that there exists a constant $\alpha = 1.942 \pm 3 \times 10^{-3}$ such that $\mathcal{Q}(n) = ((1 \pm o(1))ne^{-\alpha})^n$. The constant $\alpha$ is characterized as the solution to a convex optimization problem in $\mathcal{P}([-1/2,1/2]^2)$, the space of Borel probability measures on the square.
	
	The chief innovation is the introduction of limit objects for $n$-queens configurations, which we call \textit{queenons}. These form a convex set in $\mathcal{P}([-1/2,1/2]^2)$. We define an entropy function that counts the number of $n$-queens configurations that approximate a given queenon. The upper bound uses the entropy method of Radhakrishnan and Linial--Luria. For the lower bound we describe a randomized algorithm that constructs a configuration near a prespecified queenon and whose entropy matches that found in the upper bound. The enumeration of $n$-queens configurations is then obtained by maximizing the (concave) entropy function in the space of queenons.
	
	Along the way we prove a large deviations principle for $n$-queens configurations that can be used to study their typical structure.
\end{abstract}

\maketitle

\pagestyle{plain}

\section{Introduction}

An $n$-queens configuration is a placement of $n$ mutually non-threatening queens on an $n \times n$ chessboard. As queens attack along rows, columns, and diagonals, this is equivalent to an order-$n$ permutation matrix in which the sum of each diagonal is at most $1$. The $n$-queens problem is to determine $\mQ(n)$, the number of such configurations. In this paper we prove the following result on the asymptotics of $\mQ(n)$.

\begin{thm}\label{thm:main theorem first statement}
	There exists a constant $\upperBoundConstant < \alpha < \lowerBoundConstant$ such that
	\[
	\lim_{n \to \infty} \frac{\mQ(n)^{1/n}}{n} = e^{-\alpha}.
	\]
\end{thm}

Previously, the best known bounds were
\[
e^{-1.58} > \limsup_{n\to\infty} \frac{\mQ(n)^{1/n}}{n} \geq \liminf_{n\to\infty} \frac{\mQ(n)^{1/n}}{n} \geq e^{-3},
\]
the upper bound due to Luria \cite{luria2017new} and the lower bound proved independently by Luria and the author \cite{luria2021lower} and Bowtell and Keevash \cite{bowtell2021n}. Before these, the best upper bound was the trivial $\mQ(n) \leq n!$ and the best lower bounds held only for infinite families of natural numbers $n$ (cf.\ \cite{rivin1994n}), whereas the only bound for all $n$ was $\mQ(n) = \Omega(1)$. We note, however, that \cite{zhang2009counting}, which is a physics paper, used Monte Carlo simulations to empirically estimate $\log \left( \frac{1}{n} \mQ(n)^{1/n} \right) \approx - 1.944000$. Previously, Benoit Cloitre \cite[Sequence A000170]{oeis} conjectured that $\log \left( \frac{1}{n} \mQ(n)^{1/n} \right) \approx -1.940$. Theorem \ref{thm:main theorem first statement} justifies these claims. For more history and an extensive list of open problems, we refer the reader to Bell and Stevens's survey \cite{bell2009survey}.

Our methods also allow us to study the typical structure of $n$-queens configurations. To state the main result in this vein we introduce some notation. Let $\mR$ be the collection of subsets of the plane with the form
\[
\left\{ (x,y) \in \interval^2 : a_1 \leq x+y \leq b_1, a_2 \leq y-x \leq b_2 \right\}
\]
for $a_1,a_2,b_1,b_2 \in [-1,1]$. (We use the square $\interval^2$ rather than $[0,1]^2$ because it better respects the natural symmetries of the problem.) Let $\gamma_1,\gamma_2$ be two finite Borel measures on $\interval^2$. We define the distance between $\gamma_1$ and $\gamma_2$ by
\[
\dist{\gamma_1}{\gamma_2} = \sup \left\{ \left| \gamma_1(\alpha) - \gamma_2(\alpha) \right| : \alpha \in \mR  \right\}.
\]
Let $q$ be an $n$-queens configuration, which we view as a subset of $[n]^2$. Define the step function $g_q:\interval^2 \to \R$ by $g_q \equiv n$ on every square $((i-1)/n-1/2,i/n-1/2) \times ((j-1)/n-1/2,j/n-1/2)$ such that $(i,j) \in q$ and $g_q \equiv 0$ elsewhere. Let $\gamma_q$ be the probability measure with density function $g_q$. Our main structural result is the following.

\begin{figure}
	\includegraphics[width=\textwidth*7/8]{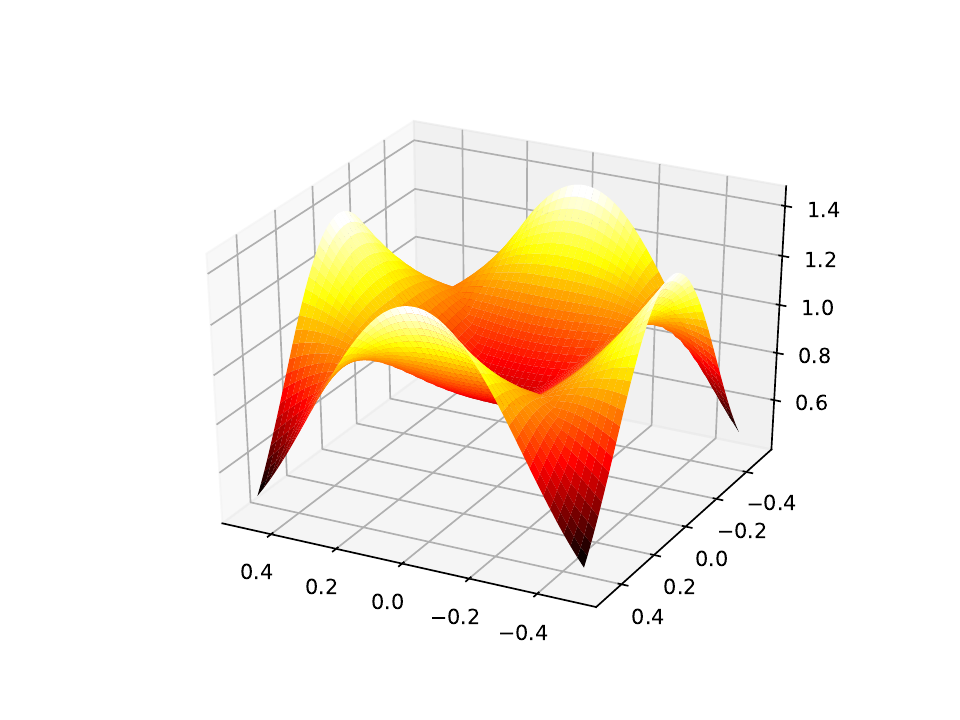}
	\caption{The density function of $\gamma^*$. This is the distribution of queens in a typical $n$-queens configuration.}\label{fig:3d_opt_50}
	\centering
\end{figure}

\begin{thm}\label{thm:main structural}
	There exists a Borel probability measure $\gamma^*$ on $\interval^2$ such that the following holds: Let $\varepsilon > 0$ be fixed and let $q$ be a uniformly random $n$-queens configuration. W.h.p.\footnote{We say that a sequence of events parameterized by $n$ occurs \termdefine{with high probability} (\termdefine{w.h.p.}) if the probability of its occurrence tends to $1$ as $n \to \infty$.} $\dist{\gamma_q}{\gamma^*} < \varepsilon$. Moreover, $\Prob \left[\dist{\gamma_q}{\gamma^*} \geq \varepsilon \right] \leq \exp \left( - \Omega_\varepsilon \left(n\right) \right)$.
\end{thm}

Both the constant $\alpha$ from Theorem \ref{thm:main theorem first statement} and the measure $\gamma^*$ are characterized as the solution to a concave optimization problem defined in Section \ref{sec:queenons}. For a visualization of $\gamma^*$ see Figure \ref{fig:3d_opt_50}.

\subsection{Designs, entropy, and randomized algorithms}

We view $n$-queens configurations as an example of a \textit{combinatorial design}. The last quarter century has seen several breakthroughs related to the construction, enumeration, and analysis of designs. These include the Radhakrishnan \textit{entropy method} \cite{radhakrishnan1997entropy}, which was extended by Linial and Luria \cite{linial2013upper,linial2014upper} to give upper bounds on the number of designs; the \textit{R\"odl nibble} \cite{rodl1985packing} and \textit{random greedy algorithms} \cite{spencer1995asymptotic}, used to construct approximate designs; and, more recently, completion methods such as \textit{randomized algebraic constructions} \cite{keevash2014existence} and \textit{iterative absorption} \cite{glock2016existence}, used to complete approximate designs. We also mention the emerging \textit{limit theory} of combinatorial designs \cite{garbe2020limits,cooper2020quasirandom}, which draws on the theory of graphons \cite{lovasz2012large}, and from which this paper draws inspiration.

These methods are powerful enough to enumerate many classes of designs. In particular, the combination of random greedy algorithms and completion \cite{keevash2018counting,keevash2018existence} often yields lower bounds that match the upper bounds obtained with the entropy method. Nevertheless, the $n$-queens problem has remained challenging for two reasons. The first is the asymmetry of the constraints: Since the diagonals vary in length from $1$ to $n$, some board positions are more ``threatened'' than others. This makes the analysis of nibble-style arguments difficult. Additionally, the constraints are not regular: In a complete configuration, some diagonals contain a queen and some do not. This creates difficulties for the entropy method.

To overcome these challenges we define limit objects for $n$-queens configurations, which we call \textit{queenons}. We give their precise definition in Section \ref{sec:queenons}. For the current discussion it suffices to think of these as Borel probability measures on $\interval^2$. To count $n$-queens configurations we take the following approach. Rather than attempting to estimate $\mQ(n)$ directly, we fix a queenon $\gamma$, a parameter $\varepsilon > 0$ and set ourselves the easier task of estimating $|B_n(\gamma,\varepsilon)|$, where $B_n(\gamma,\varepsilon)$ is the set of $n$-queens configurations $q$ satisfying $\dist{\gamma_q}{\gamma} < \varepsilon$.

For the upper bound we use the entropy method: We choose $q \in B_n(\gamma,\varepsilon)$ uniformly at random and reveal its queens in a random order. The knowledge that $q$ is close to $\gamma$ allows us to obtain tight bounds on the entropy of each step in this process, which in turn gives a tight upper bound on $|B_n(\gamma,\varepsilon)|$ in terms of a ``queenon entropy'' function $H_q$.

For the lower bound we design a randomized algorithm that constructs an element of $B_n(\gamma,\varepsilon)$ by placing one queen at a time on the board. The algorithm has the additional property that the entropy of each step matches the entropy of the corresponding step in the upper bound. Very roughly, in each step of the algorithm we first choose a small area of the board according to the distribution $\gamma$. We then place a queen in a uniformly random position from that area subject to the constraint that it does not conflict with previously placed queens. We show that w.h.p.\ this algorithm places $n-o(n)$ queens on the board and, furthermore, w.h.p.\ the outcome of the algorithm is close to a complete configuration. Since the entropy of this process matches the entropy in the upper bound we obtain a matching lower bound on $|B_n(\gamma,\varepsilon)|$.

Notably, we do not use a simple random greedy algorithm for the lower bound. Instead, we use queenons as a ``bridge'' between the entropy method on the one side and a randomized construction on the other. Thus, the upper and lower bounds are two sides of the same coin: each follows from estimating the entropy of a process in which a configuration is constructed one queen at a time.

After finding tight bounds for $|B_n(\gamma,\varepsilon)|$ we use a compactness argument to reduce estimating $\mQ(n)$ to maximizing the (concave) entropy function $H_q$ over the (convex) space of queenons.

The rest of this paper is organized as follows. At the end of this section we introduce notation. In Section \ref{sec:queenons} we define queenons and their entropy function $H_q$. We state an enumeration theorem (Theorem \ref{thm:precise bounds}) which we use to prove a large deviations principle (Theorem \ref{thm:LDP}). We then use Theorem \ref{thm:LDP} to prove Theorem \ref{thm:main structural}. In Section \ref{sec:useful calculations} we collect useful claims. In Section \ref{sec:upper bound} we prove the upper bound in Theorem \ref{thm:precise bounds} and we prove the lower bound in Section \ref{sec:lower bound}. These two sections can be read independently of each other. In Section \ref{sec:optimizing H_q} we bound the optimal value of $H_q$, which ultimately implies Theorem \ref{thm:main theorem first statement}. We close with a few comments and open problems in Section \ref{sec:conclusion}.

\subsection{Notation}\label{sec:notation}

We introduce here some notation and definitions that we use throughout the paper. For the reader's convenience, many of the symbols that have a ``global'' scope (including some that are defined in later sections) are collected in notation tables in \cref{app:notation}.

For $n \in \N$ we write $[n] = \{1,2,\ldots,n\}$. For $a,b \in \R$ we use $a \pm b$ to denote a quantity in the interval $[a-|b|,a+|b|]$.

Let $n \in \N$. A \termdefine{row} in $[n]^2$ is a set of the form $\{(1,y),(2,y),\ldots,(n,y)\}$ and a \termdefine{column} is a set of the form $\{(x,1),(x,2),\ldots,(x,n)\}$. For $c \in \Z$, \termdefine{plus-diagonal $c$} is the set $\{(x,y) \in [n]^2:x+y=c\}$ and \termdefine{minus-diagonal $c$} is the set $\{(x,y) \in [n]^2 : y-x = c\}$. The term ``diagonal'' refers to a diagonal of either type.

A \termdefine{partial $n$-queens configuration} is a set $Q \subseteq [n]^2$ containing at most one element in each row, column, and diagonal. We say that $(x,y) \in [n]^2$ is \termdefine{available} in $Q$ if it does not share a row, column, or diagonal with any element of $Q$. We denote the set of such positions by $\mA_Q$.

Throughout the paper, unless stated otherwise, all asymptotics are as $n \to \infty$ and other parameters fixed. In general, we will assume that $n$ is sufficiently large for asymptotic inequalities to hold. For example, we may write $n^2 > 10n$ without explicitly requiring $n>10$.

\subsection{Partitions of $\interval^2$, $[n]^2$, and $[-1,1]$}\label{sec:partitions}

Although $n$-queens configurations are discrete objects, in this paper we consider their limits as analytic objects. The following notation is useful when moving from the discrete set $[n]^2$ to the continuous set $\interval^2$. Let $n \in \N$ and let $i,j \in [n]$. Define
\[
\sigma_{i,j}^n \coloneqq ((i-1)/n - 1/2,i/n - 1/2) \times ((j-1)/n - 1/2,j/n - 1/2).
\]

For $N \in \N$ let $I_N$ be the division of $\interval^2$ into squares and half-squares of the form
\[
\{ (x,y) \in \interval^2 : -1 + \frac{i-1}{N} \leq x+y \leq -1 + \frac{i}{N}, -1 + \frac{j-1}{N} \leq y-x \leq -1 + \frac{j}{N} \}
\]
for $i,j \in [2N]$ (see Figure \ref{fig:lattice}). Note that these sets are $\ell_1$-balls of radius $1/(2N)$ (intersected with $\interval^2$). We denote the squares in $I_N$ by $S_N$ and the half-squares by $T_N$. For $\alpha \in I_N$ we write $|\alpha|$ for its area (so that $|\alpha| = 1/(2N^2)$ if $\alpha \in S_N$ and $|\alpha| = 1/(4N^2)$ if $\alpha \in T_N$).

\begin{figure}
	\centering
	\begin{subfigure}{0.45\textwidth}
		\centering
		\includegraphics[width=0.9\textwidth]{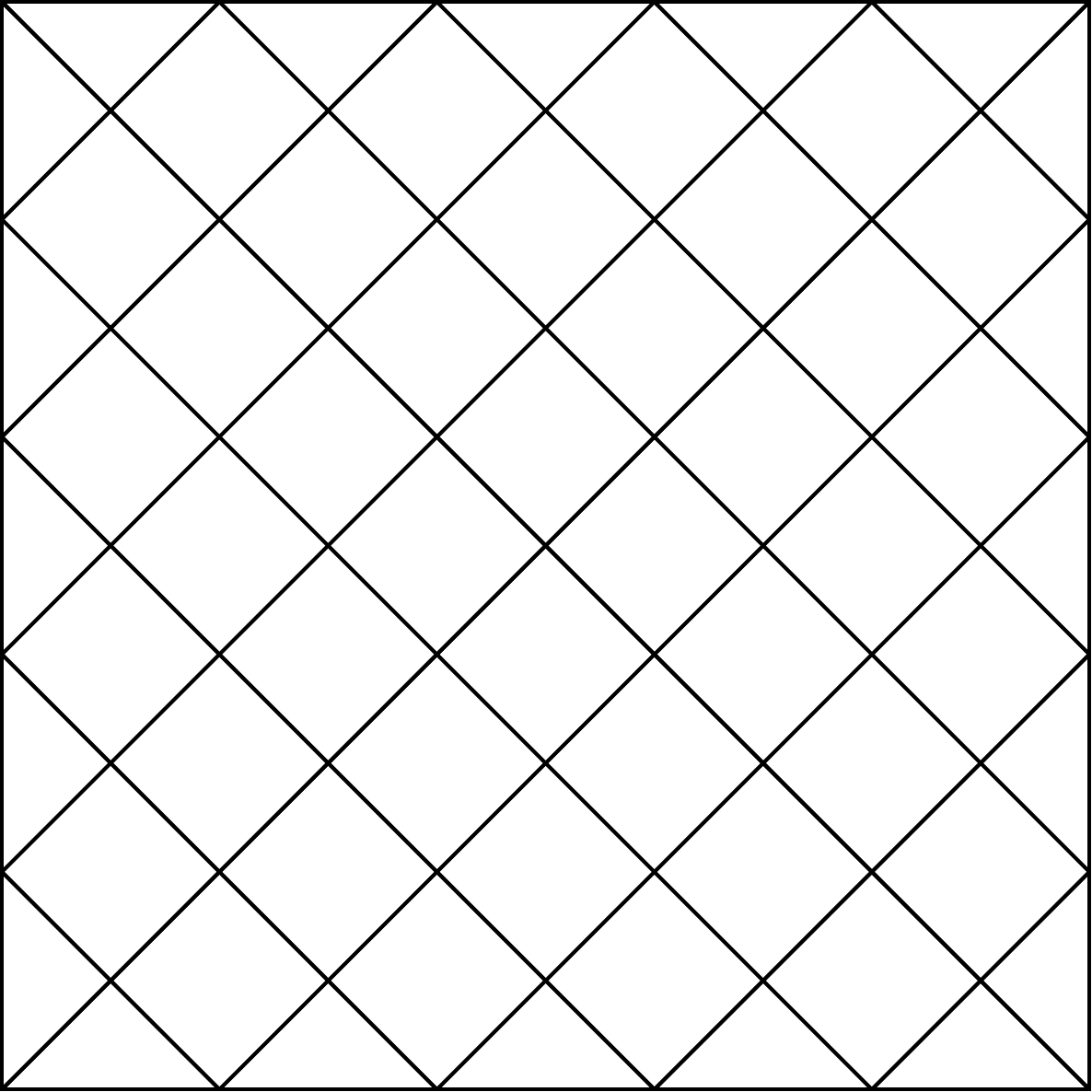}
	\end{subfigure}
	\begin{subfigure}{0.45\textwidth}
		\centering
		\includegraphics[width=0.9\textwidth]{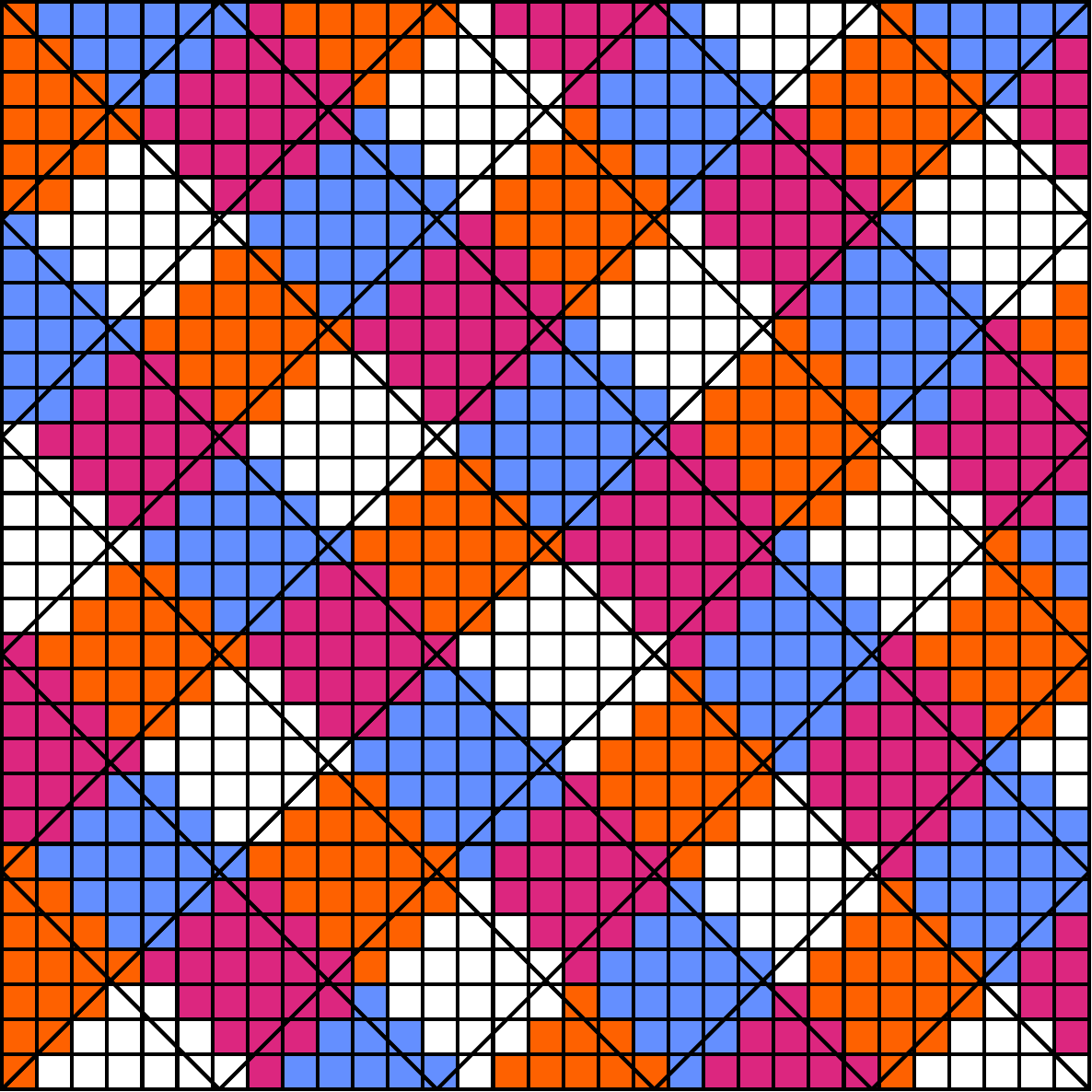}
	\end{subfigure}
	\caption{On the left, the division of $\interval^2$ into $I_N$, for $N=5$. The squares have area $1/(2N^2)$ while the half-squares have area $1/(4N^2)$. On the right, the corresponding partition of $[n]^2$ into $\{\alpha_n\}_{\alpha\in I_N}$, for $n=31$.}\label{fig:lattice}
\end{figure}

Let $n,N \in \N$. We partition $[n]^2$ into sets $\{\alpha_n\}_{\alpha \in I_N}$ as follows: For each $(i,j) \in [n]^2$, assign $(i,j)$ to the set $\alpha_n$ such that $\alpha \cap \sigma_{i,j}^n \neq \emptyset$ and such that the center-point of $\alpha$ is minimal in the lexicographic order. We emphasize that $\alpha$ is a subset of the continuous set $\interval^2$ whereas $\alpha_n$ is a subset of the discrete set $[n]^2$. We observe that $|\alpha_n| = |\alpha|n^2 \pm 8 \lceil n/N \rceil$ for every $\alpha \in I_N$. We write $\alpha^N(i,j)$ for the element $\alpha \in I_N$ such that $(i,j) \in \alpha_n$. Usually, $N$ will be clear from context, in which case we write $\alpha(i,j)$.

Let $(x,y) \in [n]^2$ and $\alpha \in I_N$. We write $L_{y,\alpha}^r$, $L_{x,\alpha}^c$, $L_{x+y,\alpha}^+$, and $L_{y-x,\alpha}^-$ for the number of positions in $\alpha_n$ and, respectively, row $y$, column $x$, plus-diagonal $x+y$, and minus-diagonal $y-x$. Here we abuse notation and do not write the dependence on $n$ explicitly; whenever we use this notation $n$ is clear from context.

Let $J_N$ be the division of $[-1,1]$ into the intervals $\{[-1+(i-1)/N,-1+i/N]\}_{1 \leq i \leq 2N}$.

We remark that neither $I_N$ nor $\{\sigma_{i,j}^n\}_{i,j\in[n]}$ is a partition of $\interval^2$. However, they are partitions up to sets of measure zero under all measures considered in the paper. Similarly, $J_N$ is a partition of $[-1,1]$ up to sets of measure zero under all measures we consider.

\section{Queenons}\label{sec:queenons}

In this section we define queenons --- the limits of $n$-queens configurations. We also define an associated entropy function and prove basic properties of these objects.

The limit theory of combinatorial objects is interesting in its own right (see, for example, \cite{lovasz2012large,hoppen2013limits,beneliezer_et_al:LIPIcs.ITCS.2021.42,garbe2020limits}). Nevertheless, it is beyond our scope to develop a comprehensive theory of queenons. Instead, we restrict ourselves to statements needed for the proofs of Theorems \ref{thm:main theorem first statement} and \ref{thm:main structural}.

\subsection{Definitions and basic properties}

Queens configurations are, in particular, permutation matrices. There is already a well-developed limit theory for permutations, in which the limiting objects are called \textit{permutons} \cite{hoppen2013limits,kral2013quasirandom,glebov2015finitely,kenyon2020permutations}. Let us recall their definition.

\begin{definition}
	A \termdefine{permuton} is a Borel probability measure on $\interval^2$ with uniform marginals:
	\[
	\forall -1/2 \leq a \leq b \leq 1/2, \gamma([a,b] \times \interval) = \gamma(\interval \times [a,b]) = b-a.
	\]
	
	For $N \in \N$, a permuton $\gamma$ is an \termdefine{$N$-step permuton} if for every $i,j \in [N]$, $\gamma$ has constant density (i.e., constant Radon--Nikodym derivative with respect to the Lebesgue measure) on $\sigma_{i,j}^N$. We call $\gamma$ a \termdefine{step permuton} if it is an $N$-step permuton for some $N$.
\end{definition}

\begin{rmk}
	In the definition above we follow \cite{kenyon2020permutations}. There are other, equivalent, definitions.
\end{rmk}

Before defining queenons we recall that since $\interval^2$ is a compact metric space, the space $\mP$ of Borel probability measures on $\interval^2$ with the weak topology is compact and metrizable (cf.\ \cite[Lemma 6.4]{parthasarathy2005probability}).

The characterization of $n$-queens configurations as permutation matrices in which the sum of every diagonal is at most $1$ suggests the following definitions.

\begin{definition}
	Let $\mu \in \mP$. We say that $\mu$ has \termdefine{sub-uniform diagonal marginals} if for every $-1 \leq a \leq b \leq 1$ it holds that
	\begin{align*}
	& \mu \left( \left\{ (x,y) \in \interval^2 : a \leq y-x \leq b \right\} \right) \leq b-a,\\
	& \mu \left( \left\{ (x,y) \in \interval^2 : a \leq x+y \leq b \right\} \right) \leq b-a.
	\end{align*}
\end{definition}
	
\begin{definition}\label{def:queenons}
	Let $\tilde{\Queenons} \subseteq \mP$ be the set of step permutons with sub-uniform diagonal marginals. Let $\Queenons = \overline{\tilde{\Queenons}}$ be its closure in the weak topology. We call the elements of $\Queenons$ \termdefine{queenons} and the elements of $\tilde{\Queenons}$ \termdefine{step queenons}.
\end{definition}

Recall that for an $n$-queens configuration $q \subseteq [n]^2$, we denote by $\gamma_q \in \mP$ the measure that has constant density $n$ on every $\sigma_{i,j}^n$ such that $(i,j) \in q$ and density $0$ elsewhere. The next observation follows from the more general \cref{clm:step queenon sufficient condition}.

\begin{obs}\label{obs:queens are queenons}
	Let $q \subseteq [n]^2$ be an $n$-queens configuration. Then $\gamma_q \in \tilde{\Queenons}$ and, in particular, is a queenon.
\end{obs}

\begin{obs}
	Every queenon has sub-uniform diagonal marginals.
\end{obs}

\begin{proof}
	This follows immediately from the fact that the set of measures in $\mP$ with sub-uniform diagonal marginals is closed in the weak topology.
\end{proof}

\begin{rmk}
	An alternative approach is to define queenons as the set of permutons with sub-uniform diagonal marginals. Denote this set by $\Gamma'$. As far as the goals of this paper are concerned it makes no difference whether one uses $\Gamma$ or $\Gamma'$. In particular, the enumeration theorem (\cref{thm:precise bounds}) and the large deviations principle (\cref{thm:LDP}) hold with $\Gamma$ replaced by $\Gamma'$. However, we do not know whether $\Gamma=\Gamma'$. Specifically, we were not able to prove that if $\gamma \in \Gamma'$ is not absolutely continuous with respect to the Lebesgue measure then $\gamma \in \Gamma$. (A proof that this does in fact hold when $\gamma \in \Gamma'$ \textit{is} absolutely continuous with respect to the Lebesgue measure follows similarly to the proof of \cref{lem:approximation}, below.) We have elected to work with $\Gamma$ since a consequence of the enumeration theorem is that for every $\gamma \in \Queenons$ there is a sequence of queens configurations $\{q_n\}_{n\in\N}$ such that $\gamma_{q_n} \to \gamma$. This justifies the perspective of queenons as limits of $n$-queens configurations.
\end{rmk}

Every queenon carries with it information about the distribution of queens in the diagonals. This is encapsulated by the measures on $[-1,1]$ in the next definition.

\begin{definition}\label{def:interval measures}
	For $\gamma \in \mP$ we define the probability measures $\gamma^+$ and $\gamma^-$ on $[-1,1]$ as the pushforwards of $\gamma$ under, respectively, $(x,y) \mapsto x+y$ and $(x,y) \mapsto y-x$. In other words, for any Borel set $X \subseteq [-1,1]$ we have
	\begin{align*}
	&\gamma^+(X) = \gamma \left( \left\{ (x,y) \in \interval^2 : x+y \in X \right\} \right) \text{ and}\\
	&\gamma^-(X) = \gamma \left( \left\{ (x,y) \in \interval^2 : y-x \in X \right\} \right).
	\end{align*}
	If $\gamma$ has sub-uniform diagonal marginals then for every $-1\leq a \leq b \leq 1$ it holds that $\gamma^+([a,b]),$ $\gamma^-([a,b]) \leq b-a$. Thus, we can define probability measures $\distPlus{\gamma}$ and $\distMinus{\gamma}$ on $[-1,1]$ by
	\begin{align*}
	&\distPlus{\gamma}([a,b]) = b-a-\gamma^+([a,b]),\\
	&\distMinus{\gamma}([a,b]) = b-a-\gamma^-([a,b]).
	\end{align*}
	In other words, momentarily denoting the Lebesgue measure on $[-1,1]$ by $\lambda$, we have $\distPlus{\gamma}(X) = \lambda(X) - \gamma^+(X)$ and $\distMinus{\gamma}(X) = \lambda(X) - \gamma^-(X)$ for every Borel set $X$.
	
	We also define the following notation: Let $\gamma \in \mP$, $N \in \N$, and $\alpha \in I_N$. There exists a unique $\beta \in J_N$ such that $\gamma(\alpha)$ contributes to $\gamma^+(\beta)$. We abuse notation and define $\gamma^+(\alpha) = \gamma^+(\beta)$. Similarly, we write $\gamma^-(\alpha)$ for $\gamma^-(\beta)$, where $\beta$ is the unique element of $J_N$ such that $\gamma(\alpha)$ contributes to $\gamma^-(\beta)$. If $\gamma$ has sub-uniform diagonal marginals we define $\distPlus{\gamma}(\alpha)$ and $\distMinus{\gamma}(\alpha)$ similarly.
\end{definition}

We are ready to define the entropy of a queenon.

Let $\mathcal{U}_\square$ denote the uniform distribution on $\interval^2$ and let $\mU_{[-1,1]}$ denote the uniform distribution on $[-1,1]$. We remind the reader that if $\mu$ is a probability measure on $\interval^2$ with density function $f$ then the Kullback--Leibler (KL) divergence is defined by
\[
D_{KL}(\mu || \mU_\square) \coloneqq \int_{\interval^2} f(x) \log (f(x)) dx.
\]
We remark that KL divergence is always nonnegative and may be infinite. If $\mu$ does not have a density function then we define $D_{KL}(\mu || \mU_\square) = \infty$. The KL divergence of a probability measure $\nu$ on $[-1,1]$ with density function $g$ is denoted and defined by
\[
D_{KL}(\nu || \mU_{[-1,1]})  \coloneqq \int_{[-1,1]} g(x) \log(2g(x)) dx.
\]
When it is clear from context if a measure $\rho$ is defined on $\interval^2$ or on $[-1,1]$ we may write simply $D_{KL}(\rho)$ for the KL divergence of $\rho$ with respect to the appropriate uniform distribution.

\begin{definition}\label{def:Q entropy}
	Let $\gamma \in \Queenons$. We define its \termdefine{Q-entropy} by
	\[
	H_q(\gamma) = -D_{KL} \left( \gamma || \mU_\square \right) - D_{KL}( \distPlus{\gamma} || \mU_{[-1,1]}) - D_{KL}( \distMinus{\gamma} || \mU_{[-1,1]}) + 2\log 2 - 3.
	\]
\end{definition}

We will use the following discrete approximations of $H_q$. For a finite probability distribution $p_1,\ldots,p_n$ we write $D(\{p_i\}_{i=1,\ldots,n}) \coloneqq \sum_{i=1}^n p_i \log \left( np_i \right)$ for its KL divergence with respect to the uniform distribution. Also, for $N \in \N$ and $\gamma \in \Queenons$ we define
\[
D^N(\gamma) = \sum_{\alpha \in I_N} \gamma(\alpha) \log \left( \frac{\gamma(\alpha)}{|\alpha|} \right)
\]
(recall that $|\alpha|$ is the area of $\alpha$). This is the KL divergence with respect to $\mU_\square$ of the measure $\tgamma \in \mP$ that has constant density on each $\alpha \in I_N$ and satisfies $\tgamma(\alpha) = \gamma(\alpha)$ for every $\alpha \in I_N$.

\begin{definition}\label{def:discrete Q entropy}
	Let $N \in \N$ and let $\gamma \in \Queenons$. Then
	\[
	H_q^N(\gamma) \coloneqq - D^N(\gamma) - D \left( \{ \distPlus{\gamma}(\alpha) \}_{\alpha \in J_N} \right) - D \left( \{ \distMinus{\gamma}(\alpha) \}_{\alpha \in J_N} \right) + 2\log 2 - 3.
	\]
\end{definition}

We are now in a position to state our enumeration theorem. We remind the reader that for $n \in \N$, $\gamma \in \Gamma$, and $\varepsilon > 0$, $B_n(\gamma,\varepsilon)$ is the set of $n$-queens configurations $q$ satisfying $\dist{\gamma_q}{\gamma} < \varepsilon$.

\begin{thm}\label{thm:precise bounds}
	Let $\gamma \in \Queenons$. Then:
	\begin{itemize}
		\item \textbf{Upper bound}: For all sufficiently small $\varepsilon > 0$ there exists an integer $N \geq \varepsilon^{-1/3}$ such that
		\[
		\limsup_{n \to \infty} \frac{|B_n(\gamma,\varepsilon)|^{1/n}}{n} \leq \exp \left( H_q^N(\gamma) + \varepsilon^{1/200} \right).
		\]
		
		\item \textbf{Lower bound}: For every $\varepsilon > 0$ there holds
		\[
		\liminf_{n \to \infty} \frac{|B_n(\gamma,\varepsilon)|^{1/n}}{n} \geq \exp \left( H_q(\gamma) - \varepsilon| H_q(\gamma)| \right).
		\]
	\end{itemize}
\end{thm}

We prove the upper bound in Section \ref{sec:upper bound} and the lower bound in Section \ref{sec:lower bound}.

\begin{rmk}
	The asymmetry between the upper bound (which uses $H_q^N$) and the lower bound (which uses $H_q$) is due to the way we approach each proof. In Lemma \ref{lem:step approximation of entropy} we prove that for every $\gamma \in \Queenons$, $\lim_{N\to\infty} H_q^N(\gamma) = H_q(\gamma)$. Together with Theorem \ref{thm:precise bounds} this implies the more symmetric statement
	\[
	e^{H_q(\gamma)} \leq \liminf_{\varepsilon \downarrow 0} \liminf_{n \to \infty} \frac{|B_n(\gamma,\varepsilon)|^{1/n}}{n} \leq \limsup_{\varepsilon \downarrow 0} \limsup_{n \to \infty} \frac{|B_n(\gamma,\varepsilon)|^{1/n}}{n} \leq e^{H_q(\gamma)}
	\]
	which itself implies
	\[
	\lim_{\varepsilon \downarrow 0} \liminf_{n \to \infty} \frac{|B_n(\gamma,\varepsilon)|^{1/n}}{n} = \lim_{\varepsilon \downarrow 0} \limsup_{n \to \infty} \frac{|B_n(\gamma,\varepsilon)|^{1/n}}{n} = e^{H_q(\gamma)}.
	\]
\end{rmk}

\begin{ex}
	Let $\gamma$ be the uniform distribution on $\interval^2$. We will show that $H_q(\gamma) = -2$. This implies that $\mQ(n) \geq  ((1-\oone)ne^{-2})^n$. This already improves on the previous best bound $\mQ(n) \geq ((1-\oone)ne^{-3})^n$ \cite{luria2021lower,bowtell2021n}.
	
	Since $\gamma$ is uniform $D_{KL}(\gamma) = 0$. By symmetry, $D_{KL}( \distPlus{\gamma}) = D_{KL}(\distMinus{\gamma})$. The density function of $\gamma^+$ is $1-|c|$ (where $c$ varies from $-1$ to $1$). Therefore the density function of $\distPlus{\gamma}$ is $|c|$. Therefore:
	\[
	D_{KL}( \distPlus{\gamma} ) = \int_{-1}^1 |c|\log(2|c|)dc = \log(2) -1/2.
	\]
	Consequently
	\[
	H_q(\gamma) = - D_{KL}(\gamma) - D_{KL}(\distPlus{\gamma}) - D_{KL}(\distMinus{\gamma}) + 2\log 2 - 3 = -2.
	\]
\end{ex}

The next claims summarize basic properties of queenons and $\distNoParam$. We will rely on the following covering lemma. Recall the definition of $\mR$ from the introduction. We say the \termdefine{width} of the sets $\{(x,y): a \leq x+y \leq b\}$ and $\{(x,y): a \leq y-x \leq b\}$ is $b-a$.

\begin{lem}\label{lem:covering}
	Let $\alpha \in \mR$ and $N \in \N$. There exists a set $X \subseteq I_N$ such that $\alpha \subseteq \bigcup_{\beta \in X} \beta$ and $\left( \bigcup_{\beta \in X} \beta \right) \setminus \alpha$ is contained in four diagonals, each of width $2/N$.
\end{lem}

\begin{proof}
	By definition of $\mR$ there exist $a_1,a_2,b_1,b_2 \in [-1,1]$ such that
	\[
	\alpha = \{ (x,y) \in \interval^2 : a_1 \leq x+y \leq b_1, a_2 \leq y-x \leq b_2 \}.
	\]
	Let $X = \{ \beta \in I_N : \alpha \cap \beta \neq \emptyset \}$. Then, by definition, $\alpha \subseteq \bigcup_{\beta \in X} \beta$. Now, for $\beta \in X$, if $\beta \nsubseteq \alpha$ then $\beta$ intersects one of the four lines $y=a_1-x,y=b_1-x,y=a_2+x,y=b_2+x$. For each line, the set of elements $\beta \in I_N$ intersecting it forms a diagonal of width $\leq 2/N$, proving the lemma.
\end{proof}

\begin{clm}\label{clm:I_N approximation of metric}
	Let $\gamma_1,\gamma_2 \in \Queenons$, $N \in \N$, and $\varepsilon > 0$. Suppose that for every $\alpha \in I_N$ we have $|\gamma_1(\alpha) - \gamma_2(\alpha)| < \varepsilon$. Then $\dist{\gamma_1}{\gamma_2} < 8/N + 4N^2 \varepsilon$.
\end{clm}

\begin{proof}
	Let $\alpha \in \mR$. Let $X \subseteq I_N$ be a cover of $\alpha$ as guaranteed by Lemma \ref{lem:covering}. Let $U = \bigcup_{\beta \in X} \beta$. Then, for $i=1,2$:
	\[
	\gamma_i(\alpha) = \gamma_i(U) - \gamma_i(U \setminus \alpha).
	\]
	Since $\gamma_i$ has sub-uniform diagonal marginals, by Lemma \ref{lem:covering} we have $\gamma_i(U \setminus \alpha) \leq 8/N$. Additionally, using the fact that $|X| \leq |I_N| \leq 4N^2$:
	\[
	\gamma_1(U) = \sum_{ \beta \in X} \gamma_1(\beta) = \sum_{ \beta \in X} \gamma_2(\beta) \pm |X|\varepsilon = \gamma_2(U) \pm 4N^2\varepsilon.
	\]
	Therefore:
	\[
	|\gamma_1(\alpha) - \gamma_2(\alpha)| \leq |\gamma_1(U) - \gamma_2(U)| + |\gamma_1(U\setminus\alpha) - \gamma_2(U\setminus\alpha)| \leq 4N^2\varepsilon + 8/N.
	\]
	We conclude that $\dist{\gamma_1}{\gamma_2} \leq 4N^2\varepsilon + 8/N$.
\end{proof}

\begin{clm}\label{clm:config to queenon approximation}
	Let $\gamma \in \Queenons$, $N \in \N$, $\varepsilon > 0$, and let $q$ be an $n$-queens configuration satisfying
	\[
	\forall \alpha \in I_N, |\alpha_n \cap q| = (\gamma(\alpha) \pm \varepsilon)n.
	\]
	Then $\dist{\gamma_q}{\gamma} \leq 4N^2 \left( \varepsilon + 8/n \right) + 8/N$.
\end{clm}

\begin{proof}
	By Claim \ref{clm:I_N approximation of metric} it is enough to show that for every $\alpha \in I_N$, $|\gamma_q(\alpha) - \gamma(\alpha)| \leq \varepsilon + 8/n$. Let $\alpha \in I_N$. Let $X$ be the set of queens $(i,j) \in q$ such that $\sigma_{i,j}^n \subseteq \alpha$ and let $Y$ be the set of queens $(i,j) \in q$ such that $\sigma_{i,j}^n \cap \alpha \neq \emptyset$. For every $(i,j) \in Y \setminus X$, $\sigma_{i,j}^n$ intersects one of the four diagonals defining $\alpha$. Since $q$ is a queens configuration, each diagonal line intersects at most $2$ queens. Therefore $|Y \setminus X| \leq 8$. Observe that $X \subseteq \alpha_n \cap q \subseteq Y \implies |X| \leq |\alpha_n \cap q| \leq |Y| \leq |X| + 8$. Similarly:
	\[
	|X|/n \leq \gamma_q(\alpha) \leq |Y|/n \leq (|X|+8)/n.
	\]
	Therefore $\gamma_q(\alpha) = (|\alpha_n \cap q| \pm 8)/n = \gamma(\alpha) \pm (\varepsilon + 8/n)$, as desired.
\end{proof}

\begin{clm}\label{clm:queenons compact convex metric}
	$(\Queenons,\distNoParam)$ is a convex, compact, metric space.
\end{clm}

\begin{proof}
	We first remark that $\distNoParam$ is a metric on $\mP$ (and, in fact, on the space of all finite Borel measures on $\interval^2$). Symmetry and the triangle inequality clearly hold, so we need only prove that for $\gamma_1,\gamma_2 \in \mP$, $\dist{\gamma_1}{\gamma_2} = 0 \implies \gamma_1 = \gamma_2$. Since $\mR$ is closed under finite intersections and generates the Borel $\sigma$-algebra, this follows from \cite[Lemma 1.9.4]{bogachev2007measure}.
	
	To see that $\Queenons$ is convex it is enough to observe that $\tilde{\Queenons}$ is convex.
	
	We have already mentioned that $\mP$, and hence $\Queenons$, is compact and metrizable with respect to the weak topology. Thus it suffices to show that weak sequential convergence in $\Gamma$ implies sequential convergence in $(\Queenons,\distNoParam)$. We remark (but do not prove) that the notions are, in fact, equivalent. However, convergence in $(\mP,\distNoParam)$ is \textit{stronger} than weak convergence. Their equivalence in $(\Queenons,\distNoParam)$ is due to the sub-uniform diagonal marginals property.
	
	Recall that a sequence $\gamma_1,\gamma_2,\ldots \in \mP$ converges to $\gamma \in \mP$ in the weak topology if for every continuous $f:\interval^2\to\R$ there holds $\lim_{n\to\inf}\int fd\gamma_n = \int f d\gamma$. Let $\gamma,\gamma_1,\gamma_2,\ldots \in \Queenons$ and suppose that $\gamma_n \to \gamma$ in the weak topology. We first show that for every $\alpha \in \mR$ it holds that $\gamma_n(\alpha) \to \gamma(\alpha)$. Let $\alpha \in \mR$ and let $\varepsilon > 0$. Let $\beta$ be the $\varepsilon$-neighborhood of $\alpha$ in the $\ell_1$ norm. Then $\beta \setminus \alpha$ is contained in four diagonals of width $\varepsilon$. Hence $\delta(\beta \setminus \alpha) \leq 4\varepsilon$ for every $\delta \in \Queenons$. Let $f:\interval^2 \to [0,1]$ be continuous, equal to $1$ on $\alpha$, and equal to $0$ outside $\beta$. Then, for every $\delta \in \Queenons$:
	\[
	\int f d\delta \geq \delta(\alpha) = \int f d \delta - \int_{\beta \setminus \alpha} f d \delta \geq \int f d \delta - 4\varepsilon.
	\]
	Thus
	\[
	|\gamma_n(\alpha) - \gamma(\alpha)| \leq \left| \int f d\gamma_n - \int f d\gamma \right| + 4\varepsilon \xrightarrow[n \to \infty]{} 4\varepsilon.
	\]
	Since $\varepsilon > 0$ was arbitrary, we conclude that $\gamma_n(\alpha) \to \gamma(\alpha)$.
	
	We now show that $\dist{\gamma_n}{\gamma} \to 0$. Let $\varepsilon > 0$. Let $N = \lfloor \varepsilon^{-1} \rfloor$ and let $n_0$ be large enough that for all $n \geq n_0$ and for every $\alpha \in I_N$ it holds that $|\gamma_n(\alpha) - \gamma(\alpha)| < \varepsilon^3$. Then, by Claim \ref{clm:I_N approximation of metric}, for every $n \geq n_0$ we have $\dist{\gamma_n}{\gamma} < 100\varepsilon$. Hence $\dist{\gamma_n}{\gamma} \to 0$. We conclude that $(\Queenons,\distNoParam)$ is compact.
\end{proof}

We now prove that $H_q^N$ approximates $H_q$.

\begin{lem}\label{lem:step approximation of entropy}
	Let $\gamma \in \Queenons$. Then $\lim_{N \to \infty} H_q^N(\gamma) = H_q(\gamma)$.
\end{lem}

\begin{proof}
	It suffices to show that $\lim_{N\to\infty} D^N(\gamma) = D_{KL}(\gamma)$, $\lim_{N \to \infty} D(\{\distPlus{\gamma}(\alpha)\}_{\alpha\in J_N}) = D_{KL}(\distPlus{\gamma})$, and $\lim_{N \to \infty} D(\{\distMinus{\gamma}(\alpha)\}_{\alpha \in J_N}) = D_{KL}(\distMinus{\gamma})$.
	
	By definition $D^N(\gamma) = \sum_{ \alpha \in I_N } \gamma(\alpha)\log \left( \gamma(\alpha) / |\alpha| \right)$. Therefore, $D^N(\gamma)$ is a Riemann sum for $D_{KL}(\gamma)$. Of course, $\gamma$ may not have a density function, and even if it does it may not be Riemann-integrable. Therefore, it is not immediate that the Riemann sums converge. This can be shown using a standard measure-theoretic argument relying on specific properties of the function $x \log x$. Rather than give the details, we derive our lemma from the following claim used to prove the analogous statement for permutons.
	
	In the next claim the absolute continuity (or lack thereof) of the measure $\mu$ is with respect to the Lebesgue measure. We remind the reader that by the Radon--Nikodym theorem if $\mu$ is absolutely continuous then it is given by a density function (also known as its Radon--Nikodym derivative).
	
	\begin{clm}[{\cite[Proposition 9]{kenyon2020permutations}}]
		Let $\mu$ be a finite Borel measure on $[0,1]^2$. For $m \in \N$ and $i,j \in [m]$, let $\mu_{i,j} = m^2 \mu([(i-1)/m,i/m]\times [(j-1)/m,j/m])$. Define:
		\[
		R_m = \frac{1}{m^2} \sum_{i,j \in [m]^2} \mu_{i,j} \log(\mu_{i,j}).
		\]
		Then:
		\begin{enumerate}
			\item If $\mu$ is absolutely continuous with density $f$ and $f \log f$ is integrable then $\lim_{m\to\infty}R_m = \int_{[0,1]^2} f\log f$.
			
			\item If $\mu$ is absolutely continuous with density $f$ and $f \log f$ is not integrable then $\lim_{m\to\infty}R_m = \infty$.
			
			\item If $\mu$ has a singular component (i.e., $\mu$ is not absolutely continuous) then $\lim_{m\to\infty}R_m = \infty$.
		\end{enumerate}
	\end{clm}
	
	In order to show that $\lim_{N\to\infty} D^N(\gamma) = D_{KL}(\gamma)$ we will define a finite measure $\mu$ on $[0,1]^2$ such that for every $N \in \N$, $R_{2N} = D^N(\gamma) + \log(2) - O(1/N)$. Let $F:\interval^2 \to [0,1]^2$ be the function
	\[
	F(x,y) = (1/2,1/2) + \frac{1}{2}(x+y,x-y).
	\]
	$F$ is a rotation of the plane by $\pi/4$ followed by rescaling and translation. It easily follows that for every $N\in\N$ and $i,j \in [2N]$ it holds that $F^{-1}([(i-1)/(2N),i/(2N)]\times [(j-1)/(2N),j/(2N)])$ is either empty or an element of $I_N$. Define the measure $\mu$ on $[0,1]^2$ by setting, for every Borel $U \subseteq [0,1]^2$, $\mu(U) = \gamma(F^{-1}(U))$. Now, for every $N \in \N$ there holds
	\[
	R_{2N} = \sum_{ \alpha \in I_N } \gamma(\alpha) \log \left( 4N^2 \gamma(\alpha) \right) = \sum_{\alpha \in I_N} \gamma(\alpha) \log \left( \frac{2\gamma(\alpha)}{|\alpha|} \right) - \sum_{\alpha \in T_N} \gamma(\alpha) \log(2).
	\]
	The half-squares in $T_N$ satisfy $\sum_{\alpha \in T_N} \gamma(\alpha) = O(1/N)$, so:
	\begin{align*}
	R_{2N} = \sum_{\alpha \in I_N} \gamma(\alpha) \log \left( \frac{2\gamma(\alpha)}{|\alpha|} \right) - O \left( \frac{1}{N} \right) = D^N(\gamma) + \log(2) - O \left( \frac{1}{N} \right).
	\end{align*}
	Hence
	\[
	\lim_{N \to \infty} D^N(\gamma) = \lim_{N \to \infty} R_{2N} - \log 2.
	\]
	Now $D_{KL}(\gamma) < \infty$ if and only if $\mu$ is absolutely continuous with density $f$ and $f \log f$ is integrable. Thus, if $D_{KL}(\gamma) = \infty$ then $\lim_{N \to \infty} D^N(\gamma) = \lim_{N \to \infty} R_{2N} - \log 2 = \infty$. Otherwise, if $\gamma$ has density function $g$ then $f = 2g\circ F^{-1}$. Hence, by the change of variables formula:
	\[
	\int_{[0,1]^2} f \log f = \int_{[0,1]^2} 2g \circ F^{-1} \log(2g\circ F^{-1}) = \frac{1}{2} \int_{\interval^2} 2g \log(2g) = D_{KL}(\gamma) + \log(2),
	\]
	implying $\lim_{N \to \infty} D^N(\gamma) = D_{KL}(\gamma)$.
	
	We now show that for $* \in \{+,-\}$, $\lim_{N \to \infty} D(\{\overline{\gamma}^*(\alpha)\}_{\alpha \in J_N}) = D_{KL}(\overline{\gamma}^*)$. We define a measure $\nu$ on $[0,1]^2$ as follows: Let $G:[-1,1]\to[0,1]$ be given by $G(x) = (x+1)/2$. Define the measure $\tilde{\nu}$ on $[0,1]$ by $\tilde{\nu}(U) = \overline{\gamma}^*(G^{-1}(U))$. Then, let $\nu$ be the product measure of $\tilde{\nu}$ with the uniform distribution on $[0,1]$. It then holds that
	\[
	D_{KL}(\nu || \mU_{[0,1]^2}) = D_{KL}(\tilde{\nu} || \mU_{[0,1]} ) = D_{KL}(\overline{\gamma}^* || \mU_{[-1,1]}).
	\]
	For every $N$ it holds that $R_N = D(\{\overline{\gamma}^*(\alpha)\}_{\alpha\in J_N})$. Therefore
	\[
	\lim_{N \to \infty} D(\{\overline{\gamma}^*(\alpha)\}_{\alpha\in J_N}) = \lim_{N \to \infty} R_N = D_{KL}(\nu || \mU_{[0,1]^2}) = D_{KL}(\overline{\gamma}^* || \mU_{[-1,1]}),
	\]
	completing the proof.
\end{proof}

\begin{lem}\label{lem:H_q concave}
	$H_q$ is strictly concave and upper semi-continuous.
\end{lem}

Before proving \cref{lem:H_q concave} we make the following observation, which follows from the convexity of KL divergence.

\begin{obs}\label{obs:HqN > Hq}
	For every $\gamma \in \Gamma$ and $N \in \N$ there holds $H_q^N(\gamma) \geq H_q(\gamma)$.
\end{obs}

\begin{proof}[Proof of \cref{lem:H_q concave}]
	Strict concavity of $H_q$ follows from strict convexity of KL divergence and the fact that $\distPlus{\gamma}$ and $\distMinus{\gamma}$ are linear in $\gamma$.
	
	To prove that $H_q$ is upper semi-continuous we note that by \cref{lem:step approximation of entropy,obs:HqN > Hq} $H_q(\gamma) = \inf_{N\in\N}H_q^N(\gamma)$ for every $\gamma \in \Queenons$. Furthermore, each $H_q^N$ is continuous. Hence $H_q$ is the pointwise infimum of a collection of continuous functions and therefore is upper-semicontinuous.
\end{proof}

\begin{lem}\label{lem:unique maximizer}
	There exists a unique maximizer $\gamma^* \in \Queenons$ for $H_q$.
\end{lem}

\begin{proof}
	Uniqueness follows from the strict concavity of $H_q$. It remains to prove that $H_q$ has a maximizer. Since KL divergence is nonnegative, $H_q$ is bounded above by $2\log2-3$. Let $\gamma_1,\gamma_2,\ldots \in \Queenons$ be a sequence such that
	\[
	\lim_{n\to\infty} H_q(\gamma_n) = \sup_{\gamma \in \Queenons} H_q(\gamma).
	\]
	Since $\Queenons$ is compact we may assume that the sequence converges to a queenon $\gamma^*$. We claim that $H_q(\gamma^*) = \sup_{\gamma \in \Queenons} H_q(\gamma)$. This follows from upper semi-continuity of $H_q$.
\end{proof}

In Section \ref{sec:optimizing H_q} we will prove the following bounds on $H_q(\gamma^*)$.

\begin{clm}\label{clm:explicit bounds on H_q}
	The following holds: $-\lowerBoundConstant \leq H_q(\gamma^*) \leq - \upperBoundConstant$.
\end{clm}

\subsection{Large deviations for queenons}

Theorems \ref{thm:main theorem first statement} and \ref{thm:main structural} both follow from the following large deviations principle.

For $\Delta \subseteq \Queenons$ we write $\Delta^{\circ}$ for the interior of $\Delta$ (in $\Gamma$) and $\overline{\Delta}$ for its closure. For $n \in \N$ we write $\Delta_n$ for the set of $n$-queens configurations $q$ such that $\gamma_q \in \Delta$.

\begin{thm}\label{thm:LDP}
	Let $\Delta \subseteq \Queenons$. The following hold:
	\[
	\sup_{\gamma \in \Delta^{\circ}} H_q(\gamma) \leq \liminf_{n \to \infty} \frac{1}{n} \log \left( \frac{|\Delta_n|}{n^n} \right) \leq \limsup_{n \to \infty} \frac{1}{n} \log \left( \frac{|\Delta_n|}{n^n} \right) \leq \sup_{\gamma \in \overline{\Delta}} H_q(\gamma).
	\]
\end{thm}

Taking $\Delta = \Queenons$ and using Lemma \ref{lem:unique maximizer} and Claim \ref{clm:explicit bounds on H_q}, we derive Theorem \ref{thm:main theorem first statement}.

To prove Theorem \ref{thm:main structural}, let $\varepsilon > 0$, and take $\Delta = \left\{ \gamma \in \Queenons : \dist{\gamma}{\gamma^*} \geq \varepsilon \right\}$. Then, by upper semi-continuity and the fact that $\gamma^*$ uniquely maximizes $H_q$, we conclude that $\sup_{\gamma \in \overline{\Delta}} H_q(\gamma) < H_q(\gamma^*)$. Additionally, if an $n$-queens configuration $q$ satisfies $\dist{\gamma_q}{\gamma^*} \geq \varepsilon$ then $q \in \Delta_n$. By Theorem \ref{thm:LDP} there exists some $\delta > 0$ such that for all sufficiently large $n$,
\[
\frac{1}{n} \log\left( \frac{\mQ(n)}{n^n} \right) \geq H_q(\gamma^*) - \delta > H_q(\gamma^*) - 2\delta \geq \frac{1}{n} \log\left( \frac{|\Delta_n|}{n^n} \right).
\]
This implies
\[
\frac{|\Delta_n|}{\mQ(n)} \leq \exp \left( -n\delta \right),
\]
proving Theorem \ref{thm:main structural}.

\begin{proof}[Proof of Theorem \ref{thm:LDP}]
	The proof is modeled on that of \cite[Theorem 1]{kenyon2020permutations}.
	
	We first prove the lower bound, for which we may assume $\sup_{\gamma \in \Delta^{\circ}} H_q(\gamma) > -\infty$. Let $\varepsilon > 0$ and let $\delta \in \Delta^{\circ}$ satisfy $H_q(\delta) > \sup_{\gamma \in \Delta^{\circ}} H_q(\gamma) - \varepsilon$. Let $\rho \in (0,\varepsilon)$ satisfy $B_\rho(\delta) \subseteq \Delta$ (where $B_\rho(\delta)$ is the open ball of radius $\rho$ centered at $\delta$). Then, for every $n \in \N$, $B_n(\delta,\rho) \subseteq \Delta_n$. By the lower bound in Theorem \ref{thm:precise bounds}
	\begin{align*}
	\liminf_{n \to \infty} \frac{1}{n} \log \left( \frac{|\Delta_n|}{n^n} \right) \geq \liminf_{n \to \infty} \frac{1}{n} \log \left( \frac{|B_n(\delta,\rho)|}{n^n} \right)
	\geq H_q(\delta) - \rho |H_q(\delta)| & = (1+\rho) H_q(\delta)\\
	& \geq (1+\varepsilon) (\sup_{\gamma \in \Delta^{\circ}} H_q(\gamma) - \varepsilon).
	\end{align*}
	Since this is true for every $\varepsilon > 0$ the lower bound follows.
	
	For the upper bound we first handle the case that $\beta \coloneqq \sup_{\gamma \in \overline{\Delta}} H_q(\gamma) > -\infty$. Let $\varepsilon > 0$. By Theorem \ref{thm:precise bounds} and Lemma \ref{lem:step approximation of entropy}, for every $\delta \in \overline{\Delta}$ there exists some $n_\delta \in \N$ and some $\varepsilon_\delta > 0$ such that for all $n \geq n_\delta$:
	\[
	\left( \frac{|B_n(\delta,\varepsilon_\delta)|}{n^n} \right)^{1/n} \leq \exp \left( \beta + \varepsilon \right).
	\]
	Since $\overline{\Delta}$ is compact there exists a finite set $X \subseteq \overline{\Delta}$ such that $\overline{\Delta} \subseteq \bigcup_{\delta \in X} B_{\varepsilon_\delta}(\delta)$. Therefore, for every $n\in\N$ we have $\Delta_n \subseteq \bigcup_{\delta \in X} B_n(\delta, \varepsilon_\delta)$. Thus:
	\[
	\limsup_{n \to \infty} \frac{1}{n} \log \left( \frac{|\Delta_n|}{n^n} \right) \leq \sup_{\gamma \in \overline{\Delta}} H_q(\gamma) + \varepsilon.
	\]
	Since this is true for every $\varepsilon > 0$, we obtain the upper bound.
	
	The case that $\beta = -\infty$ is proved similarly. Let $t < 0$. By Theorem \ref{thm:precise bounds} and Lemma \ref{lem:step approximation of entropy}, for every $\delta \in \overline{\Delta}$ there exists some $n_\delta \in \N$ and some $\varepsilon_\delta > 0$ such that for all $n \geq n_\delta$:
	\[
	\left( \frac{|B_n(\delta,\varepsilon_\delta)|}{n^n} \right)^{1/n} \leq e^t.
	\]
	Applying a compactness argument we obtain
	\[
	\limsup_{n \to \infty} \frac{1}{n} \log \left( \frac{|\Delta_n|}{n^n} \right) \leq t.
	\]
	Since this is true for every $t < 0$ the proof is complete.
\end{proof}

\section{Useful claims and calculations}\label{sec:useful calculations}

We now collect several claims that will be useful in the sequel. On a first reading the reader may wish to skip this section and refer to it as each claim is used in the proof.

\begin{clm}\label{clm:H_q^N continuity bound}
	Let $N \in \N$ and $1/(2e) > \varepsilon > 0$. Suppose that $\gamma_1,\gamma_2 \in \Queenons$ satisfy $\dist{\gamma_1}{\gamma_2} < \varepsilon$. Then $|H_q^N(\gamma_1)-H_q^N(\gamma_2)| < 8N^2 \varepsilon \log \left( \frac{N}{2\varepsilon^2} \right)$.
\end{clm}

\begin{proof}
	Let $f:[0,1]\to\R$ be the function $f(x) = x\log(x)$ (with $f(0)=0$). Let $x,y \in [0,1]$ such that $|x-y| \leq \varepsilon$. We claim that $|f(x)-f(y)| \leq -2\varepsilon\log(2\varepsilon)$. Indeed, assume without losing generality that $x \leq y$. If $0 \leq x \leq \varepsilon$ then, since $f$ is convex and decreasing on $[0,x+\varepsilon]$:
	\[
	|f(x)-f(y)| \leq |f(x+|y-x|) - f(0)| \leq |f(2\varepsilon)| = - 2\varepsilon\log(2\varepsilon).
	\]
	Otherwise, $x \in [\varepsilon,1]$. We observe that for every $\zeta \in [\varepsilon,1]$, $|f'(\zeta)| \leq -2\log(2\varepsilon)$. Hence, by the mean value theorem:
	\[
	|f(x)-f(y)| \leq -2\log(2\varepsilon)|y-x| \leq -2\varepsilon\log(2\varepsilon).
	\]
	
	Now, by definition:
	\begin{align*}
	D^N(\gamma_1) - D^N(\gamma_2)	& = \sum_{ \alpha \in I_N} \left(\gamma_1(\alpha) \log \left( \gamma_1(\alpha) \right) - \gamma_2(\alpha)\log(\gamma_2(\alpha)) - (\gamma_1(\alpha)-\gamma_2(\alpha))\log(|\alpha|)\right)\\
	& = \sum_{ \alpha \in I_N } \left( f(\gamma_1(\alpha)) - f(\gamma_2(\alpha)) - (\gamma_1(\alpha)-\gamma_2(\alpha))\log(|\alpha|) \right).
	\end{align*}
	Since both $\gamma_1$ and $\gamma_2$ are probability measures:
	\[
	\left| \sum_{\alpha \in I_N} (\gamma_1(\alpha)-\gamma_2(\alpha)) \log(|\alpha|) \right| \leq 8N^2 \dist{\gamma_1}{\gamma_2} \log(2N) \leq 8N^2 \varepsilon \log(2N).
	\]
	Additionally:
	\[
	\left|\sum_{ \alpha \in I_N } (f(\gamma_1(\alpha)) - f(\gamma_2(\alpha))) \right| \leq -|I_N| 2\varepsilon \log(2\varepsilon) \leq - 8N^2 \varepsilon \log(2\varepsilon).
	\]
	By similar considerations:
	\[
	\left|D \left( \{\distPlus{\gamma_1}(\alpha)\}_{\alpha \in J_N} \right) - D \left( \{\distPlus{\gamma_2}(\alpha)\}_{\alpha \in J_N} \right)\right| \leq -4N\varepsilon\log(2\varepsilon)
	\]
	and
	\[
	\left|D \left( \{\distMinus{\gamma_1}(\alpha)\}_{\alpha \in J_N} \right) - D \left( \{\distMinus{\gamma_2}(\alpha)\}_{\alpha \in J_N} \right)\right| \leq -4N\varepsilon\log(2\varepsilon).
	\]
	Therefore:
	\[
	\left|H_q^N(\gamma_1) - H_q^N(\gamma_2)\right| \leq 8N^2\varepsilon\log \left( \frac{N}{\varepsilon} \right) - 8N\varepsilon\log(2\varepsilon) \leq 8N^2 \varepsilon \log \left( \frac{N}{2\varepsilon^2} \right),
	\]
	as claimed.
\end{proof}

\begin{clm}\label{clm:log integral}
	Let $0 < b \leq 1$ and $n,T \in \N$ satisfy $(1-1/e)n \leq T < n - \sqrt{n}$. Then
	\[
	\sum_{t=0}^{T-1} b \log \left( 1 - bt/n \right) = n \left( - (1-b) \log (1-b) - b \right) \pm 3(n-T) \left|\log(1-T/n)\right|.
	\]
\end{clm}

\begin{proof}
	Let $f(x) = b \log(1-bx)$ and observe that $\frac{b}{n} \sum_{t=0}^{T-1} \log \left( 1 - bt/n \right)$ is a Riemann sum for the integral $\int_0^{T/n} f(x) dx$. Also, for every $x \in [0,T/n]$:
	\[
	|f'(x)| = \frac{b^2}{1-bx} \leq \frac{1}{1-T/n}.
	\]
	Therefore:
	\[
	\int_0^{T/n} f(x) dx = \frac{b}{n} \sum_{t=0}^{T-1} \log \left( 1 - bt/n \right) \pm \frac{1}{n} \max_{x\in[0,T/n]}|f'(x)| = \frac{b}{n} \sum_{t=0}^{T-1} \log \left( 1 - bt/n \right) \pm \frac{1}{n-T}.
	\]
	We can calculate the integral exactly. Let $F(x) = -\left( bx + (1-bx)\log(1-bx) \right)$. Then $F'(x) = f(x)$. Thus:
	\begin{align*}
	\int_0^{T/n} f(x) dx & = F(T/n) - F(0) = F(1) - F(0) + F(T/n)-F(1)\\
	& = - b - (1-b)\log(1-b) + F(T/n)-F(1).
	\end{align*}
	Now, for $g(x) = x\log(x)$:
	\[
	F(T/n) - F(1) = b(1-T/n) + g(1-b)-g(1-bT/n).
	\]
	For every $x,y \in [0,1]$ with $|x-y|\leq 1/e$ there holds $|g(x)-g(y)| \leq |y-x||\log(|y-x|)|$. Therefore $|g(1-b)-{g(1-bT/n)}| \leq b(1-T/n)|\log(b(1-T/n))| \leq (1-T/n)|\log(1-T/n)|$. Finally, since $T \geq (1-1/e)n$ we have $|\log(1-T/n)| \geq 1$. Therefore:
	\[
	|F(T/n) - F(1)| \leq 2(1-T/n)|\log(1-T/n)|.
	\]
	Hence:
	\[
	\sum_{t=0}^{T-1} b \log \left( 1 - bt/n \right) = n \left( - (1-b) \log (1-b) - b \right) \pm \left(2(n-T) \left|\log(1-T/n)\right| + \frac{n}{n-T}\right).
	\]
	Since $(1-1/e)n \leq T < n-\sqrt{n}$ it holds that $n/(n-T) < \sqrt{n} < (n-T) \left|\log(1-T/n)\right|$. Therefore:
	\[
	\sum_{t=0}^{T-1} b \log \left( 1 - bt/n \right) = n \left( - (1-b) \log (1-b) - b \right) \pm 3(n-T) \left|\log(1-T/n)\right|,
	\]
	as claimed.
\end{proof}

\begin{clm}\label{clm:line sums}
	There exists a constant $C>0$ such that the following holds: Let $\gamma$ be an $N$-step queenon. Let $G_M$ be the maximal density of $\gamma$. Let $n \in \N$ satisfy $n \geq N^2$ and let $(x,y) \in [n]^2$. The following hold:
	\begin{enumerate}
		\item\label{itm:row sum} $\sum_{\alpha \in I_N} \frac{\gamma(\alpha) L_{y,\alpha}^r}{|\alpha_n|} = \frac{1}{n} \pm \frac{C NG_M}{n^2}$.
		
		\item\label{itm:column sum} $\sum_{\alpha \in I_N} \frac{\gamma(\alpha) L_{x,\alpha}^c}{|\alpha_n|} = \frac{1}{n} \pm \frac{C NG_M}{n^2}$.
		
		\item\label{itm:plus diagonal sum} $\sum_{ \alpha \in I_N } \frac{\gamma(\alpha) D_{x+y,\alpha}^+ }{|\alpha_n|} = \frac{N \gamma^+(\alpha)}{n} \pm \frac{C G_M}{Nn}$.
		
		\item\label{itm:minus diagonal sum} $\sum_{ \alpha \in I_N } \frac{\gamma(\alpha) D_{y-x,\alpha}^- }{|\alpha_n|} = \frac{N \gamma^-(\alpha)}{n} \pm \frac{C G_M}{Nn}$.
	\end{enumerate}
\end{clm}

\begin{proof}
	Let $G$ be the $N \times N$ matrix such that for every $i,j\in[N]$, the density of $\gamma$ on $\sigma_{i,j}^N$ is $G_{i,j}$.
	
	Let $\delta$ be the probability measure on $\interval^2$ that, for every $\alpha \in I_N$, has constant density on $\alpha$ and satisfies $\delta(\alpha) = \gamma(\alpha)$. We claim that $\delta$ is a permuton, i.e., has uniform marginals. We will show that it has uniform marginals along columns; this suffices because of the symmetry between rows and columns.
	
	Let $f:\interval \to \R$ be the density function of the marginal distribution of $\delta$ along vertical lines (i.e., for every $-1/2\leq a \leq b \leq 1/2$ we have $\delta(\{(x,y) : x \in [a,b]\}) = \int_a^bf(x)dx$). We need to show that $f \equiv 1$. We observe that $f$ is continuous and piece-wise linear with respect to the intervals $\{[-1/2+(i-1)/(2N),-1/2+i/(2N)]\}_{i \in [2N]}$. Thus, it suffices to show that for every integer $0 \leq i \leq 2N$ it holds that $f(-1/2+i/(2N)) = 1$.
	
	We must take a closer look at $\delta$. For this we need some notation. We recommend the reader have Figure \ref{fig:lattice} at hand. For $(a,b) \in \R^2$ and $\varepsilon > 0$ let $B_\varepsilon^1(a,b)$ be the closed $\ell_1$-ball of radius $\varepsilon$ centered at $(a,b)$. Observe that every element of $I_N$ is the intersection of an $\ell_1$-ball of radius $1/(2N)$ with $\interval^2$. For $0 \leq i \leq 2N$ even and $j \in [N]$, let
	\[
	\alpha_{i,j} \coloneqq B_{1/(2N)}^1(-1/2+i/(2N), -1/2 - 1/(2N) + j/N) \cap \interval^2 \in I_N
	\]
	and for $1 \leq i < 2N$ odd and $0 \leq j \leq N$ let
	\[
	\alpha_{i,j} \coloneqq B_{1/(2N)}^1(-1/2+i/(2N), -1/2 + j/N) \cap \interval^2 \in I_N.
	\]
	
	We make the following observations.
	
	\begin{itemize}
		\item If $i=0$ then, for every $j \in [N]$, $\alpha_{i,j} = \alpha_{0,j}$ is a half-square contained (up to a set of measure zero) in $\sigma_{1,j}^N$. Thus, $\delta(\alpha_{0,j}) = G_{1,j}/(4N^2)$. Therefore $f(-1/2) = \sum_{j=1}^NG_{1,j}/N$. Because $G$ is the density matrix of a permuton, the sum along each row is $N$. Therefore $f(-1/2) = 1$.
		
		\item The case $i=2N$ is handled similarly: $f(1/2) = \sum_{j=1}^N G_{N,j}/N = 1$.
		
		\item If $1 < i < 2N$ is even, then every $\alpha_{i,j}$ is a square, the left half of which is contained in $\sigma_{i/2,j}^N$ and the right half of which is contained in $\sigma_{i/2+1,j}^N$. Therefore $\delta(\alpha_{i,j}) = (G_{i/2,j} + G_{i/2+1,j})/(4N^2)$. Hence $f(-1/2+i/(2N)) = \sum_{j=1}^N (G_{i/2,j} + G_{i/2+1,j})/(2N) = 1$.
		
		\item If $1 \leq i < 2N$ is odd then for $1 \leq j < N$, $\alpha_{i,j}$ is a square, the lower half of which is contained in $\sigma_{(i+1)/2,j}^N$ and the upper half of which is contained in $\sigma_{(i+1)/2,j+1}^N$. In this case $\delta(\alpha_{i,j}) = (G_{(i+1)/2,j} + G_{(i+1)/2,j+1})/(4N^2)$. Additionally, $\alpha_{i,0}$ is a half-square contained in $\sigma_{(i+1)/2,1}^N$ and $\alpha_{i,N}$ is a half-square contained in $\sigma_{(i+1)/2,N}^N$. Therefore $\delta(\alpha_{i,1}) = G_{(i+1)/2,1}/(4N^2)$ and $\delta(\alpha_{i,N}) = G_{(i+1)/2,N}/(4N^2)$. Hence
		\begin{align*}
		f(-1/2+i/(2N)) & = \frac{1}{2N} \left(G_{(i+1)/2,1} + G_{(i+1)/2,N} + \sum_{j=1}^{N-1} \left( G_{(i+1)/2,j} + G_{(i+1)/2,j+1} \right)  \right)\\
		& = \frac{1}{N} \sum_{j=1}^N G_{(i+1)/2,j} = 1.
		\end{align*}
	\end{itemize}
	This completes the proof that $\delta$ is a permuton.
	
	In the following, the constants $0 < C_1 < C_2 <\ldots$ are each chosen to be sufficiently large with respect to the previous choices. We emphasize that none of them depend on $\gamma,N$, or $n$.
	
	We now prove \ref{itm:row sum}. By construction: $\sum_{\alpha \in I_N} \frac{\gamma(\alpha) L_{y,\alpha}^r}{|\alpha_n|} = \sum_{\alpha \in I_N} \frac{\delta(\alpha) L_{y,\alpha}^r}{|\alpha_n|}$.	Because $\delta$ is a permuton it has uniform marginals and so:
	\[
	\sum_{ a=1 }^n \delta(\sigma_{a,y}^n) = \delta(\{ (a,b) : -1/2 + (y-1)/n \leq b \leq -1/2+y/n \}) = \frac{1}{n}.
	\]
	Hence it suffices to prove that
	\begin{equation}\label{eq:row column sum goal}
	\left| \sum_{\alpha \in I_N} \frac{\delta(\alpha) L_{y,\alpha}^r}{|\alpha_n|} - \sum_{ a=1 }^n \delta(\sigma_{a,y}^n) \right| \leq \frac{C NG_M}{n^2}
	\end{equation}
	for a suitable constant $C$ (independent of $\gamma,N$, and $n$).
	
	Recall that for every $\alpha \in I_N$, $|\alpha_n| = |\alpha|n^2 \pm 8n/N = |\alpha|n^2(1 \pm C_1 N/n)$. Therefore, $\sum_{\alpha \in I_N} \frac{\delta(\alpha) L_{y,\alpha}^r}{|\alpha_n|} = \left( 1 \pm \frac{C_2 N}{n} \right) \frac{1}{n^2} \sum_{\alpha \in I_N} \frac{\delta(\alpha) L_{y,\alpha}^r}{|\alpha|}$. Now, there are fewer than $2N$ elements $\alpha \in I_N$ such that $L_{y,\alpha}^r>0$. Additionally, for each one, $L_{y,\alpha}^r \leq 2n/N$. Finally, for every $\alpha$, there holds $\delta(\alpha)/|\alpha| \leq G_M$. Hence: $\left( 1 \pm \frac{C_2 N}{n} \right) \frac{1}{n^2} \sum_{\alpha \in I_N} \frac{\delta(\alpha) L_{y,\alpha}^r}{|\alpha|} = \frac{1}{n^2} \sum_{\alpha \in I_N} \frac{\delta(\alpha) L_{y,\alpha}^r}{|\alpha|} \pm \frac{C_3 G_M N}{n^2}$. Now consider $\frac{1}{n^2}\sum_{\alpha \in I_N} \frac{\delta(\alpha) L_{y,\alpha}^r}{|\alpha|}$. This may be rewritten as $\sum_{a=1}^n \frac{\delta(\alpha(a,y))}{|\alpha(a,y)|n^2}$. Let $X \subseteq [n]$ be the set of indices $a$ such that $\sigma_{a,y}^n \subseteq \alpha(a,y)$. For every $a \in X$ there holds $\delta(\sigma_{a,y}^n) = \delta(\alpha(a,y))/(|\alpha(a,y)|n^2)$. Therefore:
	\[
	\left|\sum_{a=1}^n \frac{\delta(\alpha(a,y))}{|\alpha(a,y)|n^2} - \sum_{ a=1 }^n \delta(\sigma_{a,y}^n)\right| \leq \sum_{a \notin X} \left( \frac{\delta(\alpha(a,y))}{|\alpha(a,y)|n^2} + \delta(\sigma_{a,y}^n) \right) \leq \frac{2G_M}{n^2}(n-|X|).
	\]
	Since there are at most $3N$ indices $a$ such that $\sigma_{a,y}^n$ intersects more than one element of $I_N$ we have:
	\[
	\left| \sum_{\alpha \in I_N} \frac{\delta(\alpha) L_{y,\alpha}^r}{|\alpha_n|} - \sum_{ a=1 }^n \delta(\sigma_{a,y}^n) \right| \leq \frac{C_3 G_M N}{n^2} + \frac{6 G_M N}{n^2} \leq \frac{C_4 G_M N}{n^2},
	\]
	proving \eqref{eq:row column sum goal} and hence \ref{itm:row sum}. A proof of \ref{itm:column sum} is obtained by interchanging the roles of rows and columns in the preceding proof.
	
	Next, we use a similar argument to prove \ref{itm:plus diagonal sum}. Consider
	\[
	\sum_{ \alpha \in I_N } \frac{\gamma(\alpha) D_{x+y,\alpha}^+ }{|\alpha_n|} = \sum_{ \alpha \in I_N } \frac{\delta(\alpha) D_{x+y,\alpha}^+ }{|\alpha_n|} = \sum_{ \alpha \in S_N } \frac{\delta(\alpha) D_{x+y,\alpha}^+ }{|\alpha_n|} + \sum_{ \alpha \in T_N } \frac{\delta(\alpha) D_{x+y,\alpha}^+ }{|\alpha_n|}.
	\]
	We first show that the contribution from $T_N$ is negligible. Indeed, there are at most $4$ elements $\alpha \in T_N$ such that $D_{x+y,\alpha}^+ > 0$. For every $\alpha$ it holds that $D_{x+y,\alpha}^+ \leq 3n/N$. Therefore
	\[
	\sum_{ \alpha \in T_N } \frac{\delta(\alpha) D_{x+y,\alpha}^+ }{|\alpha_n|} \leq \frac{C_5 G_M}{nN}.
	\]
	Now, for every $\alpha \in S_N$ such that $D_{x+y,\alpha}^+ > 0$ we have $D_{x+y,\alpha}^+ = n/(2N) \pm 1$ and $|\alpha_n| = n^2/(2N^2) \pm 8n/N$. Therefore:
	\begin{align*}
	\sum_{ \alpha \in S_N } \frac{\delta(\alpha) D_{x+y,\alpha}^+ }{|\alpha_n|} & = \left( 1 \pm \frac{C_6 N}{n} \right) \frac{N}{n} \sum_{ \alpha \in S_N : D_{x+y,\alpha}^+>0 } \delta(\alpha)\\
	& = \left( 1 \pm \frac{C_6 N}{n} \right) \frac{N}{n} \left( \delta^+(\alpha) - \sum_{ \alpha \in T_N : D_{x+y,\alpha}^+>0 } \delta(\alpha) \right).
	\end{align*}
	Again using the fact that there are at most $4$ half-squares $\alpha \in T_N$ such that $D_{x+y,\alpha}^+>0$:
	\begin{align*}
	\sum_{ \alpha \in S_N } \frac{\delta(\alpha) D_{x+y,\alpha}^+ }{|\alpha_n|} & = \left( 1 \pm \frac{C_6 N}{n} \right) \frac{N}{n} \left( \delta^+(\alpha) \pm \frac{G_M}{N^2} \right) = \frac{N\delta^+(\alpha)}{n} \pm \left( \frac{C_6 N}{n^2} + \frac{G_M}{Nn} + \frac{C_6 G_M}{n^2} \right).
	\end{align*}
	By assumption, $n \geq N^2$. Additionally, because $\gamma$ is a permuton, $G_M \geq 1$. Therefore $\frac{C_6 N}{n^2} + \frac{G_M}{Nn} + \frac{C_6 G_M}{n^2} \leq \frac{C_7 G_M}{Nn}$. Finally, we note that by construction $\delta^+(\alpha) = \gamma^+(\alpha)$. Therefore:
	\[
	\sum_{ \alpha \in I_N } \frac{\gamma(\alpha) D_{x+y,\alpha}^+ }{|\alpha_n|} = \frac{N\gamma^+(\alpha)}{n} \pm \frac{C_7 G_M}{nN},
	\]
	as desired. A proof of \ref{itm:minus diagonal sum} can be obtained similarly.
\end{proof}

The next claim establishes a sufficient condition for a step permuton to be a queenon.

\begin{clm}\label{clm:step queenon sufficient condition}
	Let $A$ be a nonnegative $N \times N$ matrix in which the sum of every row and column is equal to $N$ and the sum of every diagonal is at most $N$. Let $\gamma$ be the measure on $\interval^2$ that for every $i,j \in [N]$ has constant density $A_{i,j}$ on the square $\sigma_{i,j}^N$. Then $\gamma$ is a queenon.
\end{clm}

\begin{proof}
	The fact that $A$ has all row and column sums equal to $N$ implies that $\gamma$ is a permuton. It remains to verify the sub-uniform diagonal marginals property. We will do so for $\gamma^+$; the proof for $\gamma^-$ is similar.
	
	Let $g:\interval^2\to\R$ and $g^+:[-1,1] \to \R$ be the density functions of $\gamma$ and $\gamma^+$, respectively. We wish to show that $g^+$ is bounded from above by $1$.
	
	We first observe that $g^+$ is linear on each of the intervals $[-1+(i-1)/N,-1+i/N]$ for $i \in [2N]$. Indeed, denoting the indicator of $\sigma_{i,j}^N$ by $s_{i,j}$, we can write $g = \sum_{i,j \in [N]}A_{i,j} s_{i,j}$. From this it follows that for every $c \in [-1,1]$ and $i,j \in [N]$ the contribution of $A_{i,j} s_{i,j}$ to $g^+(c)$ is proportional to $A_{i,j}$ as well as to the length of the line segment $\ell_{i,j}(c) \coloneqq \{ (x,y) \in \sigma_{i,j}^N : x+y=c \}$.
	
	Let $i,j \in [N]$ and $c \in [-1,1]$. In order to calculate the length of $\ell_{i,j}(c)$, which we denote $\norm{\ell_{i,j}(c)}$, we consider two cases. First, if $c < -1 + (i+j-2)/N$ or $c > -1 + (i+j)/N$ then $\ell_{i,j}(c)$ is empty and $\norm{\ell_{i,j}(c)} = 0$. Otherwise there exists some $t \in [-1/N,1/N]$ such that $c = -1 + (i+j-1)/N + t$. In this case $\norm{\ell_{i,j}(c)} = \sqrt{2}(1/N-|t|)$.
	
	We can now calculate the contribution of $A_{i,j}s_{i,j}$ to $g^+(c)$. As we have mentioned there exists a universal constant $C$ such that this contribution is $CA_{i,j}\norm{\ell_{i,j}(c)}$. To determine the value of $C$ we note that the total contribution, which is $\int_{-1}^1 CA_{i,j} \norm{\ell_{i,j}(c)} dc = \sqrt{2}CA_{i,j} / N^2$ must be equal to $\int_{\interval^2} A_{i,j} s_{i,j} = A_{i,j}/N^2$. It follows that $C = 1/\sqrt{2}$.
	
	By the piece-wise linear nature of $g^+$ it now suffices to show that for every $k=0,1,\ldots,2N$ there holds $g^+(-1+k/N) \leq 1$. By the considerations above for each such $k$ we have
	\[
	g^+(-1+k/N) = \sum_{i,j \in [N]} \frac{A_{i,j}}{\sqrt{2}} \norm{\ell_{i,j}(-1+k/N)} = \sum_{i,j \in [N] : i+j = k+1} \frac{A_{i,j}}{\sqrt{2}} \frac{\sqrt{2}}{N} = \frac{1}{N} \sum_{i,j \in [N] : i+j = k+1} A_{i,j}.
	\]
	On the right we are summing a diagonal of $A$. By assumption this is at most $N$, which implies the desired bound.
\end{proof}

\subsection{An approximation lemma}

To prove the lower bound in \cref{thm:precise bounds} we will use the fact that every queenon can not only be approximated by a step queenon (which is true by definition) but that this can be done without losing too much entropy.

\begin{lem}\label{lem:approximation}
	Let $\gamma \in \Queenons$ and $\varepsilon > 0$. There exists a step queenon $\tgamma \in \tilde{\Queenons}$ such that $\dist{\gamma}{\tgamma} < \varepsilon$. Furthermore, if $H_q(\gamma) > -\infty$ then we may choose $\tgamma$ such that $H_q(\tgamma) > H_q(\gamma) - \varepsilon$. Additionally, we may assume that the densities of $\tgamma^+$ and $\tgamma^-$ are bounded away from $1$ and that the density of $\tgamma$ is bounded away from $0$.
\end{lem}

\cref{lem:approximation} holds by definition when $H_q(\gamma) = -\infty$. Henceforth, we assume that $H_q(\gamma)>-\infty$.

\begin{definition}
	For a permuton $\gamma$ and $N \in \N$, we write $\gamma_N$ for the permuton with constant density on $\sigma^N_{i,j}$ and $\gamma_N(\sigma^N_{i,j}) = \gamma(\sigma^N_{i,j})$ for every $i,j \in [N]$.
\end{definition}

\begin{rmk}
	It is easy to prove the permuton analogue of \cref{lem:approximation}. Indeed, given a permuton $\gamma$ and $N \in \N$, it is always the case that $\gamma_N$ has permuton entropy at least that of $\gamma$ (this follows from the concavity of KL divergence). Since $\gamma_N\to\gamma$, the statement follows. However, even if $\gamma$ is a queenon $\gamma_N$ may not be. Furthermore, even if $\gamma_N$ \textit{is} a queenon, it is possible that $H_q(\gamma_N) < H_q(\gamma)$. This explains the relative complexity of the proof of \cref{lem:approximation}.
\end{rmk}

The proof idea is to take a very large $N \in \N$ and consider $\gamma_N$. As mentioned, $\gamma_N$ is not necessarily a queenon, as it might not have sub-uniform diagonal marginals. Nevertheless, we will show that by ``shifting'' a small amount of probability mass we can modify it to have sub-uniform diagonal marginals.

In the following, $\varepsilon_1,\varepsilon_2,\ldots$ are positive constants that are chosen successively such that each is sufficiently small with respect to all previous choices as well as $\varepsilon$ and $\gamma$.

As a first step we replace $\gamma$ with a queenon with strictly sub-uniform diagonal marginals. Let $\kappa$ be the $3$-step queenon whose density function is given by the matrix
\[
\begin{bmatrix}
	0.9&1.2&0.9\\
	1.2&0.6&1.2\\
	0.9&1.2&0.9
\end{bmatrix}.
\]
We note that \cref{clm:step queenon sufficient condition} implies that $\kappa$ is indeed a queenon. Let $\gamma_1 \coloneqq \varepsilon_1 \kappa + (1-\varepsilon_1)\gamma$.

\begin{obs}\label{obs:strict SUDM}
	The densities of $\gamma_1^+$ and $\gamma_1^-$ are at most $1 - 0.2 \varepsilon_1 < 1$.
\end{obs}

\begin{proof}
	This follows immediately from the fact that the densities of $\kappa^+$ and $\kappa^-$ are less than $0.8$.
\end{proof}

Since $\gamma_1$ has finite queenon entropy it has a density function $g$. Next, we find a continuous approximation of $g$. By Lusin's theorem \cite[Theorem 2.23]{rudin1966real} there exists a continuous $\tg : \interval^2 \to [0,\infty)$ satisfying:
\begin{enumerate}[({\bfseries{G\arabic{enumi}}})]
	\item\label{itm:tg-g small} $\int |\tg - g| < \varepsilon_2$,
	
	\item $\int |g\log g - \tg \log \tg| < \varepsilon_2$,
	
	\item $\mu(\{x : \tg(x) \neq g(x)\}) < \varepsilon_2$, where $\mu$ is the Lebesgue measure on $\interval^2$.
\end{enumerate}

Let $N \in \N$ be sufficiently large with respect to $1 / \varepsilon_2$ and $\norm{\tg}_\infty$. Let $\gamma_2 \coloneqq (\gamma_1)_N$. We emphasize that $\gamma_2$ is a permuton but not necessarily a queenon. Let $A$ be the $N \times N$ matrix corresponding to the density function of $\gamma_2$ in the obvious way (so that every row and column of $A$ sums to $N$). By \cref{clm:step queenon sufficient condition} $\gamma_2$ is a queenon if every diagonal sum in $A$ is at most $N$. We will show that almost all diagonal sums of $A$ are less than $N$, so that $\gamma_2$ is almost a queenon. We will then ``spread'' the over-weighted diagonals around the rest of the permuton to obtain a bona fide queenon that is close to $\gamma_2$ (and hence $\gamma$).

Let $\tgamma_2$ be the measure on $\interval^2$ defined by $\tgamma_2(X) = \int_X \tg$. We now show that because $\tg$ is continuous, $\tgamma_2$ almost has sub-uniform diagonal marginals. Since $\gamma_2-\tgamma_2$ has small total weight, this will allow us to conclude the same for $\gamma_2$.

We need some terminology. First, for $k \in \Z$, we call sets of the form $\bigcup_{i+j=k} \sigma^N_{i,j}$ or $\bigcup_{i-j=k} \sigma^N_{i,j}$ \termdefine{block diagonals}. Observe that each (nonempty) block diagonal is contained in a unique diagonal strip of width $2/N$. Call a block diagonal $D$ \termdefine{bad} if $|\tgamma_2(D) - \gamma_1(D)| > \sqrt{\varepsilon_2} / N$ or $| \tgamma_2(E) - \gamma_1(E)| > \sqrt{\varepsilon_2} / N$ for $E$ the diagonal strip containing $D$ (otherwise the block diagonal is \termdefine{good}). Since $\int |\tg - g| < \varepsilon_2$ there are at most $8\sqrt{\varepsilon_2} N$ bad block diagonals (counting both plus- and minus-diagonals).

\begin{clm}\label{clm:good diagonals are good}
	Let $D$ be a good block diagonal. Then
	\[
	N^2 \tgamma_2(D) = N^2 \gamma_2(D) \pm \sqrt{\varepsilon_2}N = \sum_{i,j : \sigma^N_{i,j} \subseteq D} A_{i,j} \pm \sqrt{\varepsilon_2}N \leq (1 - 0.1\varepsilon_1) N.
	\]
\end{clm}

\begin{rmk}
	\cref{clm:good diagonals are good} implies that if $g$ itself is continuous then $\gamma_2$ is an $N$-step queenon.
\end{rmk}

\begin{proof}
	We first note that $\gamma_2(D)=\gamma_1(D)= \frac{1}{N^2} \sum_{i,j:\sigma^N_{i,j} \subseteq D}A_{i,j}$ by definition. Since $D$ is good, it holds that $\tgamma_2(D) = \gamma_1(D) \pm \sqrt{\varepsilon_2}/N$. This proves the two equalities in the claim. It remains to prove the inequality.
	
	Let $E$ be the width-$2/N$ diagonal strip containing $D$. By \cref{obs:strict SUDM} the density of $\gamma_1^+$ is $\leq 1-0.2\varepsilon_1$. Hence:
	\[
	\gamma_1(E) \leq 2(1-0.2\varepsilon_1) / N.
	\]
	Let $\mD$ be the collection of squares $\sigma_{i,j}^N$ that are contained in $D$. For every $\sigma\in\mD$, let $x_\sigma \in \interval^2$ be its centerpoint. We may assume that $N$ is large enough that for every $x,y \in \interval^2$, if $\norm{x-y}_2 < 5 / N$ then $|\tg(x) - \tg(y)| < \varepsilon_2^2$. It then holds that
	\begin{equation}\label{eq:continuous diagonal weight}
		\tgamma_2(E) = \frac{2}{N^2} \sum_{\sigma \in \mD}\left( \tg(x_\sigma) \pm \varepsilon_2^2 \right) \pm \frac{\norm{\tg}_\infty}{N^2} = \frac{2}{N^2} \sum_{\sigma \in \mD}\tg(x_\sigma) \pm \frac{3\varepsilon_2^2}{N}.
	\end{equation}
	Similarly, since $D$ is good, it holds that
	\begin{equation}\label{eq:block diagonal continuous weight}
		\tgamma_2(D) = \frac{1}{N^2} \sum_{\sigma \in \mD} \left( \tg(x_\sigma) \pm \varepsilon_2^2 \right) \stackrel{\eqref{eq:continuous diagonal weight}}{=} \frac{1}{2} \tgamma_2(E) \pm \frac{3 \varepsilon_2^2}{N} = \frac{1}{2} \gamma_1(E) \pm \left(\frac{3\varepsilon_2^2}{N} + \frac{\sqrt{\varepsilon_2}}{N}\right) \leq \frac{1 - 0.2\varepsilon_1 + 2\sqrt{\varepsilon_2}}{N}.
	\end{equation}
	Hence
	\[
	N^2\gamma_2(D) \pm \sqrt{\varepsilon_2}N \leq N^2\tgamma_2(D) + 2\sqrt{\varepsilon_2}N \stackrel{\eqref{eq:block diagonal continuous weight}}{\leq} (1 - 0.2\varepsilon_1 + 4\sqrt{\varepsilon_2})N \leq (1 - 0.1\varepsilon_1)N,
	\]
	completing the proof.
\end{proof}

We will make small changes to $\gamma_2$ to obtain a queenon. The basic building blocks are matrices that shift weight between diagonals while preserving the uniform marginals. Specifically, for $a,b,c,d \in [N]$ let $A(a,b,c,d)$ be the $N \times N$ matrix in which $(A(a,b,c,d))_{a,b} = (A(a,b,c,d))_{c,d} = 1$, $(A(a,b,c,d))_{a,d} = (A(a,b,c,d))_{c,b} = -1$, and all other entries are $0$. Observe that for every $N \times N$ matrix $B$ and every $t \in \R$, adding $tA(a,b,c,d)$ to $B$ leaves its row and column sums unchanged.

For every $a,b \in [N]$, let $\mathfrak{A}_{a,b}$ be the set of matrices $A(a,b,c,d)$ such that all block diagonals incident to $(a,d),(c,b)$, and $(c,d)$ are good and $a \neq c, b \neq d$. We will show presently that this set is not empty. Let
\[
B(a,b) \coloneqq \frac{1}{|\mathfrak{A}_{a,b}|} \sum_{X \in \mathfrak{A}_{a,b}} X.
\]

\begin{clm}\label{clm:shifter properties}
	Let $a,b \in [N]$. The following hold:
	\begin{enumerate}
		\item\label{itm:shifter self 1} $(B(a,b))_{a,b} = 1$;
		
		\item\label{itm:far square bound} for every $c \neq a$ and $d \neq b$, $0 \leq (B(a,b))_{c,d} \leq 2 / N^2$;
		
		\item\label{itm:near square bound} for every $c \neq a$ and $d \neq b$, $0 \geq (B(a,b))_{c,b}, (B(a,b))_{a,d} \geq - 2 / N$;
		
		\item\label{itm:shifter row columns sums} all row and column sums in $B(a,b)$ are zero; and
		
		\item\label{itm:B absolute sumn} $\sum_{i,j \in [N]} |B(a,b)_{i,j}| \leq 4$.
	\end{enumerate}
\end{clm}

\begin{proof}
	We first prove that
	\begin{equation}\label{eq:absorber size lower bound}
		|\mathfrak{A}_{a,b}| \geq N^2\! / 2.
	\end{equation}
	Indeed, there are fewer than $8\sqrt{\varepsilon_2}N$ bad diagonals and each intersects column $b$ in at most one place. Thus, there are at least $(1-9\sqrt{\varepsilon_2})N$ indices $c \neq a$ such that both diagonals incident to $(c,b)$ are good. By similar reasoning, there are at least $(1-9\sqrt{\varepsilon_2})N$ indices $d \neq b$ such that both diagonals incident to $(a,d)$ are good. Each such choice results in a unique $(c,d)$. Since there are at most $8\sqrt{\varepsilon_2}N$ bad diagonals and each contains at most $N$ squares, there are at most $8\sqrt{\varepsilon_2}N^2$ choices of $c,d$ such that $(c,d)$ is contained in a bad diagonal. Therefore:
	\[
	|\mathfrak{A}_{a,b}| \geq \left((1-9\sqrt{\varepsilon_2})N\right)^2 - 8\sqrt{\varepsilon_2}N^2 \geq (1-30\sqrt{\varepsilon_2})N^2,
	\]
	implying \eqref{eq:absorber size lower bound}.
	
	We now prove the assertions in order.
	
	\begin{enumerate}
		\item By definition $X_{a,b}=1$ for every $X \in \mathfrak{A}_{a,b}$. Since $B(a,b)$ is a convex combination of $\mathfrak{A}_{a,b}$, the claim follows immediately.
		
		\item Fix $c \neq a$ and $d \neq b$. By definition
		\[
		(B(a,b))_{c,d} \leq \frac{1}{|\mathfrak{A}_{a,b}|} A(a,b,c,d) \stackrel{\eqref{eq:absorber size lower bound}}{\leq} \frac{2}{N^2},
		\]
		as desired.
		
		\item We will prove the claim for $(B(a,b))_{a,d}$; the proof for $(B(a,b))_{c,b}$ is similar. By definition:
		\begin{align*}
			(B(a,b))_{a,d} = \frac{1}{|\mathfrak{A}_{a,b}|} \sum_{X \in \mathfrak{A}_{a,b}} X_{a,d}  = - \frac{|\{ c \in  [N] : A(a,b,c,d) \in \mathfrak{A}_{a,b} \}|}{|\mathfrak{A}_{a,b}|} \stackrel{\eqref{eq:absorber size lower bound}}{\geq} - \frac{N}{N^2\!/2} = - \frac{2}{N},
		\end{align*}
		proving the claim.
		
		\item This follows from the fact that for every $a,b,c,d$, all row and column sums in $A(a,b,c,d)$ are zero.
		
		\item This follows from the fact that for every $X \in \mathfrak{A}_{a,b}$ the sum of the absolute values of the entries in $X$ is $4$.\qedhere
	\end{enumerate}
\end{proof}

Let $S \subseteq [N]^2$ be the set of index pairs $(a,b)$ such that at least one of the two block diagonals containing $\sigma^N_{a,b}$ is bad. Define
\[
C \coloneqq A - \sum_{(a,b) \in S}A_{a,b} B(a,b).
\]
Let $\gamma^\dagger$ be the measure on $\interval^2$ whose density matrix is given by $C$ and let $\tgamma \coloneqq (1-\varepsilon_3)\gamma^\dagger + \varepsilon_3 \kappa$ (with $\varepsilon_3>0$ chosen sufficiently small). Let $g^\dagger$ denote the density function of $\gamma^\dagger$. The next claim implies \cref{lem:approximation}.

\newcommand{\entropyApproximation}{{\varepsilon}}

\begin{clm}\label{clm:tgamma props}
	The following hold.
	\begin{enumerate}
		\item\label{itm:tgamma is step-queenon} $\tgamma$ is a strictly positive step-queenon and the densities of $\tgamma^+$ and $\tgamma^-$ are bounded away from $1$;
		
		\item\label{itm:tgamma close} $\dist{\tgamma}{\gamma} < \varepsilon$; and
		
		\item\label{itm:entropy approximation} $H_q(\tgamma)> H_q(\gamma) - \entropyApproximation$.
	\end{enumerate}
\end{clm}

Before proving this claim we recall that there are at most $8 \sqrt{\varepsilon_2} N$ bad diagonals. Additionally, since $\gamma_1$ has sub-uniform diagonals marginals, each block diagonal has measure at most $2/N$. Hence $\frac{1}{N^2} \sum_{(a,b) \in S}A_{a,b} \leq 8 \sqrt{\varepsilon_2} N \times (2/N)$, implying
\begin{equation}\label{eq:sum of bad cells}
\sum_{(a,b) \in S}A_{a,b} \leq 16 \sqrt{\varepsilon_2} N^2.
\end{equation}

\begin{proof}[Proof of \cref{clm:tgamma props}\ref{itm:tgamma is step-queenon}]
	We first prove that $\gamma^\dagger$ is a step queenon. For this we must show that $C$ is nonnegative, that every row and column sum in $C$ is equal to $N$, and that every (plus- or minus-) diagonal sum in $C$ is at most $N$.
	
	That the row and column sums of $C$ are equal to $N$ follows from the fact that the same holds for $A$ and \cref{clm:shifter properties} \ref{itm:shifter row columns sums}.
	
	Next, we show that $C$ is nonnegative. Let $c,d \in [N]$. By construction, $\gamma_1 \geq \varepsilon_1 \kappa$. Therefore $A_{c,d} \geq 0.6 \varepsilon_1$. If $(c,d) \in S$ then, by definition, for every $a,b \in [N]$ and $X \in \mathfrak{A}_{a,b}$ we have $X_{c,d}=0$ unless $c=a$ and $b=d$. Therefore, in this case
	\[
	C_{c,d} = A_{c,d} - A_{c,d}B(c,d) \stackrel{\text{\cref{clm:shifter properties} \ref{itm:shifter self 1}}}{=} 0.
	\]
	Otherwise we have:
	\[
	C_{c,d} = A_{c,d} - \sum_{(a,b) \in S}A_{a,b} (B(a,b))_{c,d}.
	\]
	Observe that $-(B(a,b))_{c,d} \leq 0$ only if $c\neq a$ and $b\neq d$ (since $(c,d) \notin S$ we do not need to consider the case $a=c,b=d$). In this case by \cref{clm:shifter properties}\ref{itm:far square bound} we have $(B(a,b))_{c,d} \leq 2 / N^2$, and hence:
	\begin{align*}
		C_{c,d} \geq 0.6 \varepsilon_1 - \frac{2}{N^2} \sum_{(a,b) \in S}A_{a,b} \stackrel{\text{\eqref{eq:sum of bad cells}}}{\geq} 0.6\varepsilon_1 - \frac{2}{N^2} 16 \sqrt{\varepsilon_2} N^2 > 0.6\varepsilon_1 - 32 \sqrt{\varepsilon_2} > 0.
	\end{align*}
	This completes the proof that $C$ is nonnegative.
	
	It remains to show that every diagonal sum in $C$ is at most $N$. Let $D \subseteq [N]^2$ be (the index set of) a diagonal. If $D$ corresponds to a bad block diagonal then, by construction, $D \subseteq S$ and $\sum_{(c,d) \in D} C_{c,d} = 0$. Now suppose $D$ is good. We first bound $C_{c,d}$ in terms of $A_{c,d}$ for fixed $(c,d) \in D$. Since $(c,d) \notin S$,
	\begin{align*}
		C_{c,d} & = A_{c,d} - \sum_{(a,b) \in S}A_{a,b} (B(a,b))_{c,d}\\
		& \leq A_{c,d} - \sum_{a \in [N] : (a,d) \in S} A_{a,d}(B(a,d))_{c,d} - \sum_{b \in [N] : (c,b) \in S} A_{c,b} (B(c,b))_{c,d}\\
		& \stackrel{\text{\cref{clm:shifter properties} \ref{itm:near square bound}}}{\leq} A_{c,d} + \frac{2}{N} \left( \sum_{a \in [N] : (a,d) \in S} A_{a,d} + \sum_{b \in [N] : (c,b) \in S} A_{c,b} \right).
	\end{align*}
	Using this, we bound the measure of the entire diagonal. We have
	\begin{align*}
		\sum_{(c,d) \in D} C_{c,d} & \leq \sum_{(c,d) \in D} A_{c,d} + \frac{2}{N} \sum_{(c,d) \in D} \left( \sum_{a \in [N] : (a,d) \in S} A_{a,d} + \sum_{b \in [N] : (c,b) \in S} A_{c,b} \right)\\
		& \leq \sum_{(c,d) \in D} A_{c,d} + \frac{2}{N} \sum_{(a,b) \in S} 2A_{a,b} \stackrel{\text{\eqref{eq:sum of bad cells}}}{\leq} \sum_{(c,d) \in D} A_{c,d} + \frac{4}{N} \times 16 \sqrt{\varepsilon_2}N^2\\
		& \leq \sum_{(c,d) \in D} A_{c,d} + 64 \sqrt{\varepsilon_2}N \stackrel{\text{\cref{clm:good diagonals are good}}}{\leq} \left( 1 - \frac{1}{10}\varepsilon_1 + 65 \sqrt{\varepsilon_2} \right)N < \left( 1 - \frac{1}{20}\varepsilon_1 \right)N.
	\end{align*}
	This completes the proof that $\gamma^\dagger$ is a step-queenon. Since $\tgamma$ is the convex combination of step queenons, it too is a step queenon. Furthermore, $\tgamma \geq \varepsilon_3\kappa$ which is strictly positive. Similarly, since the densities of both $\kappa^+$ and $\kappa^-$ are bounded away from $1$ so are the densities of $\tgamma^+$ and $\tgamma^-$.
\end{proof}

\begin{rmk}\label{rmk:tgamma diag density upper bound}
	Note that we have actually proved that $1-\varepsilon_1 / 20$ is an upper bound on the densities of $(\gamma^\dagger)^+$ and $(\gamma^\dagger)^-$.
\end{rmk}

\begin{proof}[Proof of \cref{clm:tgamma props}\ref{itm:tgamma close}]
	We first remark that by construction $\dist{\gamma_1}{\gamma} \leq 2\varepsilon_1$ and $\dist{\gamma^\dagger}{\tgamma} < 2\varepsilon_3$. Thus, since we assume $\varepsilon_1$ is sufficiently small, it suffices to show that $\dist{\gamma^\dagger}{\gamma_1} < \varepsilon_1$.
	
	Recall that $g^\dagger$ is the density function of $\gamma^\dagger$, that $g$ is the density function of $\gamma_1$, and that $\tg$ is a continuous approximation of $g$. It then holds that $\dist{\gamma^\dagger}{\gamma_1} \leq \int |g^\dagger - g|$. Let $g_2$ be the density function of $\gamma_2$. We then have
	\[
	\int |g^\dagger - g| \leq \int |g^\dagger - g_2| + \int |g_2 - \tg| + \int |\tg - g|.
	\]
	By \ref{itm:tg-g small}, $\int|\tg-g| < \varepsilon_2$. We turn to the second summand. There holds
	\[
	\int |g_2 - \tg| = \sum_{i,j \in [N]} \int_{\sigma^N_{i,j}} |g_2 - \tg|.
	\]
	For $i,j \in [N]$ let $x_{i,j}$ be the center point of $\sigma^N_{i,j}$. Let $\Delta_{i,j} = |g_2(x_{i,j}) - \tg(x_{i,j})|$. We may assume that $N$ is sufficiently large that for every $x \in \sigma^N_{i,j}$ we have $|\tg(x)-\tg(x_{i,j})| < \varepsilon_2$. Therefore, since $g_2$ is constant on $\sigma_{i,j}^N$, we have $\int_{\sigma_{i,j}} |g_2 - \tg| \leq \left(\Delta_{i,j} + \varepsilon_2\right)/N^2$. We conclude that
	\[
	\int |g_2 - \tg| \leq \frac{1}{N^2} \sum_{i,j \in [N]}\Delta_{i,j} + \varepsilon_2.
	\]
	
	Recall that by definition of $\gamma_2$ there holds $g_2(x_{i,j})=N^2\int_{\sigma_{i,j}^N} g$ for every $i,j \in [N]$. Hence:
	\begin{align*}
		\sum_{i,j \in [N]} \Delta_{i,j} = \sum_{i,j \in [N]} \left| \left(N^2 \int_{\sigma_{i,j}^N} g \right) - \tg(x_{i,j}) \right| & \leq \sum_{i,j \in [N]} \left| N^2 \left( \int_{\sigma_{i,j}^N} (g -\tg) + \int_{\sigma_{i,j}^N} (\tg - \tg(x_{i,j})) \right) \right|\\
		& \leq N^2 \left( \int|g-\tg| + \frac{\varepsilon_2}{N^2} \right) \leq 2N^2\varepsilon_2.
	\end{align*}
	Therefore
	\[
	\int|g_2-\tg| \leq 3\varepsilon_2.
	\]
	It remains to bound $\int|g^\dagger-g_2|$. By definition of $g^\dagger$ and $g_2$ there holds
	\begin{align*}
		\int|g^\dagger-g_2| & = \frac{1}{N^2} \sum_{i,j \in [N]} |(C-A)_{i,j}| = \frac{1}{N^2} \sum_{i,j \in [N]} \left| \sum_{(a,b) \in S}A_{a,b}(B(a,b))_{i,j} \right|\\
		& \leq \frac{1}{N^2} \sum_{(a,b) \in S} A_{a,b} \sum_{i,j \in [N]} |(B(a,b))_{i,j}| \stackrel{\text{\cref{clm:shifter properties}\ref{itm:B absolute sumn}}}{\leq} \frac{4}{N^2} \sum_{(a,b) \in S} A_{a,b} \stackrel{\text{\eqref{eq:sum of bad cells}}}{\leq} 64 \sqrt{\varepsilon_2}.
	\end{align*}
	As a consequence,
	\[
	\int|g^\dagger-g| \leq 64 \sqrt{\varepsilon_2} + 3\varepsilon_2 + \varepsilon_2 < \varepsilon_1,
	\]
	as desired.
\end{proof}

\begin{rmk}\label{rmk:small ell1 difference}
	For future reference we record that we have proved $\int |g^\dagger-g| < \varepsilon_1$.
\end{rmk}

It remains to prove \cref{clm:tgamma props}\ref{itm:entropy approximation}. Recall that by definition, for every queenon $\delta$
\[
H_q(\delta) = -D_{KL} \left( \delta || \mU_\square \right) - D_{KL}( \distPlus{\delta} || \mU_{[-1,1]}) - D_{KL}( \distMinus{\delta} || \mU_{[-1,1]}) + 2\log 2 - 3.
\]
Additionally, if $\delta$ is a step queenon whose density function is given by the $N\times N$ matrix $D$ then
\[
D_{KL} \left( \delta || \mU_\square \right) = \frac{1}{N^2}\sum_{i,j \in [N]^2} D_{i,j} \log \left( D_{i,j} \right).
\]

We will use the following auxiliary claims, which bound the contributions of the various components of $H_q(\gamma) - H_q(\tgamma)$.

\begin{clm}\label{clm:close square entropy}
	$D_{KL} \left( \gamma^\dagger || \mU_\square \right) \leq D_{KL} \left( \gamma_1 || \mU_\square \right) + \varepsilon_1$.
\end{clm}

\begin{proof}
	We first observe that $D_{KL} \left( \gamma_2 || \mU_\square \right) \leq D_{KL} \left( \gamma_1 || \mU_\square \right)$. This follows from the fact that $D_{KL} \left( \gamma_1 || \mU_\square \right) = \int_{\interval^2} g \log g$ (where $g$ is the density function of $\gamma_1$) together with the convexity of the function $x \log x$. Thus, it suffices to prove $D_{KL} \left( \gamma^\dagger || \mU_\square \right) \leq D_{KL} \left( \gamma_2 || \mU_\square \right) + \varepsilon_1$.
	
	Let $(a_1,b_1),\ldots,(a_{|S|},b_{|S|})$ be an arbitrary ordering of $S$. For $0 \leq k \leq |S|$, let
	\[
	C^k \coloneqq A - \sum_{i=1}^k A_{a_i,b_i} B(a_i,b_i),
	\]
	and let $\gamma^\dagger_k$ be the measure on $\interval^2$ with density function given by $C^k$ (so that $\gamma^\dagger_0 = \gamma_2$ and $\gamma^\dagger_{|S|}=\gamma^\dagger$).
	We will show that for every $k \in [|S|]$ there holds
	\begin{equation}\label{eq:one step DKL increase}
		D_{KL} \left( \gamma^\dagger_k || \mU_\square \right) \leq D_{KL} \left( \gamma^\dagger_{k-1} || \mU_\square \right) +\frac{1}{N^2} \left( 1 + 2A_{a_k,b_k} \log(600/\varepsilon_1) \right).
	\end{equation}
	
	We first observe that for every $k \in [|S|]$ and $a,b \in [N]$ it is always the case that either $C_{a,b}^k = 0$ or $C_{a,b}^k \geq \varepsilon_1 / 2$. Furthermore, if $(a,b) \notin S$ then $C_{a,b}^k \geq \varepsilon_1 / 2$ always. Indeed, recall that $\gamma^\dagger_0=\gamma_2$ which has density function greater than $0.6\varepsilon_1$. Therefore $C_{a,b}^0 = A_{a,b} \geq 0.6\varepsilon_1$ for every $a,b$. Now, if $(a,b) \in S$, then by definition, $B(a_i,b_i)_{a,b}=0$ unless $(a,b)=(a_i,b_i)$, in which case $B(a_i,b_i)_{a,b}=1$. This implies the claim for $(a,b) \in S$. If $(a,b) \notin S$, we note that for every $i \in [|S|]$ it holds that $B(a_i,b_i)_{a,b} \geq 0$ only if $a_i\neq a$ and $b_i\neq b$, in which case $B(a_i,b_i)_{a,b} \leq 2/N^2$. Hence, for every $k$, it holds that $C^k_{a,b} \geq A_{a,b} - \frac{2}{N^2} \sum_{i=1}^{|S|}A_{a_i,b_i}$, which by \eqref{eq:sum of bad cells} is at least $0.6\varepsilon_1 - 32\sqrt{\varepsilon_2}>\varepsilon_1/2$.
	
	For $x>0$ let $f(x) = x \log(x)$ and let $f(0)=0$. Let $k \in [|S|]$. To avoid excessive subscripting, let $a=a_k$ and $b=b_k$. For $i,j \in [N]$ let $\Delta_{i,j} \coloneqq C^k_{i,j}-C^{k-1}_{i,j}$. We then have
	\begin{align*}
		D_{KL} \left( \tgamma_k || \mU_\square \right) - D_{KL} \left( \tgamma_{k-1} || \mU_\square \right) & = \frac{1}{N^2} \sum_{i,j \in [N]} \left( f(C^{k}_{i,j}) - f(C^{k-1}_{i,j}) \right)\\
		& = \frac{1}{N^2} \sum_{i,j \in [N]} \left( f(C^{k-1}_{i,j} + \Delta_{i,j}) - f(C^{k-1}_{i,j}) \right).
	\end{align*}
	We first note that by definition $C^k_{a,b}=0$ and $C^{k-1}_{a,b}=A_{a,b}$. Thus $f(C^k_{a,b})-f(C^{k-1}_{a,b})=-A_{a,b}\log(A_{a,b}) \leq -\min f = 1/e \leq 1$.
	
	Next, for every $i \neq a$ and $j \neq b$ observe that by \cref{clm:shifter properties}\ref{itm:far square bound} $0 \geq \Delta_{i,j} \geq - 2A_{a,b} / N^2 \geq -2/N \geq -\varepsilon_1/4$. If $(i,j) \in S$ then $\Delta_{i,j}=0$. Otherwise $C_{i,j}^{k-1} \geq \varepsilon_1/2$. Since $f$ is convex $f(C^{k-1}_{i,j} + \Delta_{i,j}) - f(C^{k-1}_{i,j}) \leq \Delta_{i,j} f'(C_{i,j}^{k-1}+\Delta_{i,j}) \leq \Delta_{i,j} f'(\varepsilon_1/2) = \Delta_{i,j} (1+\log(\varepsilon_1/2))$. Thus:
	\[
	\frac{1}{N^2} \sum_{i,j \in [N]:i\neq a, j \neq b} \left( f(C^{k-1}_{i,j} + \Delta_{i,j}) - f(C^{k-1}_{i,j}) \right) \leq \frac{2A_{a,b}}{N^2} \log(4/\varepsilon_1).
	\]
	Finally, we will bound
	\[
	\frac{1}{N^2} \sum_{i \in [N]:i\neq a} \left( f(C^{k-1}_{i,b} + \Delta_{i,b}) - f(C^{k-1}_{i,b}) \right) + \frac{1}{N^2} \sum_{j \in [N]:j\neq b} \left( f(C^{k-1}_{a,j} + \Delta_{a,j}) - f(C^{k-1}_{a,j}) \right).
	\]
	We will bound the first sum; the second sum can be similarly bounded. We note that by \cref{clm:shifter properties}\ref{itm:near square bound} for every $i\neq a$ there holds $0 \leq \Delta_{i,b} \leq 2A_{a,b}/N \leq 2$. We further note that if $(i,b) \in S$ then $\Delta_{i,b}=0$. Otherwise $C^{k-1}_{i,b} \geq \varepsilon_1 / 2$. Hence, since $f$ is convex, we have 
	\begin{align*}
		f(C^{k-1}_{i,b} + \Delta_{i,b}) - f(C^{k-1}_{i,b}) \leq \Delta_{i,b} f'(C_{i,b}^{k-1} + \Delta_{i,b}) & \leq \frac{2A_{a,b}}{N} f'(C^{k-1}_{i,b}+2)\\
		& \leq \frac{2A_{a,b}}{N} \left(1 + \log(C_{i,b}^{k-1} + 2)\right).
	\end{align*}
	Consequently:
	\begin{align*}
		\frac{1}{N^2} & \sum_{i \in [N]:i\neq a} \left( f(C^{k-1}_{i,b} + \Delta_{i,b}) - f(C^{k-1}_{i,b}) \right)
		\leq \frac{2A_{a,b}}{N^3} \sum_{i \in [N]:i\neq a} \left(1 + \log(C_{i,b}^{k-1} + 2)\right)\\
		& \leq \frac{2A_{a,b}}{N^2} \left(1 + \log \left( \frac{2}{N}\sum_{i \in [N]:i \neq a}C_{i,b}^{k-1} + 2 \right)\right) \leq \frac{2A_{a,b}}{N^2} \left(1 + \log(4)\right) = \frac{2A_{a,b}}{N^2} \log(4e).
	\end{align*}
	Putting these together we obtain
	\[
	D_{KL} \left( \gamma^\dagger_k || \mU_\square \right) - D_{KL} \left( \gamma^\dagger_{k-1} || \mU_\square \right) \leq \frac{1}{N^2} \left( 1 + 2A_{a,b} (\log(4/\varepsilon_1) + 2\log(4e)) \right),
	\]
	proving \eqref{eq:one step DKL increase}.
	
	It follows that
	\begin{align*}
		D_{KL} \left( \gamma^\dagger || \mU_\square \right) - D_{KL} \left( \gamma_2 || \mU_\square \right) \leq \sum_{(a,b) \in S} \frac{1}{N^2} \left( 1 + 2A_{a,b} \log(600/\varepsilon_1) \right).
	\end{align*}
	Recall that $\sum_{(a,b) \in S}A_{a,b} \leq 16 \sqrt{\varepsilon_2}N^2$ and that $|S| \leq 8\sqrt{\varepsilon_2}N^2$. This implies
	\[
	D_{KL} \left( \gamma^\dagger || \mU_\square \right) - D_{KL} \left( \gamma_2 || \mU_\square \right) \leq 8 \sqrt{\varepsilon_2} (1 + 4\log(600/\varepsilon_1)).
	\]
	Provided that $\varepsilon_2$ is sufficiently small with respect to $\varepsilon_1$, this proves the claim.
\end{proof}

\begin{clm}\label{clm:close diagonal entropy}
	$D_{KL}(\distPlus{\gamma^\dagger}||\mU_{[-1,1]}) + D_{KL}(\distMinus{\gamma^\dagger}||\mU_{[-1,1]}) \leq D_{KL}(\distPlus{\gamma_1}||\mU_{[-1,1]}) + D_{KL}(\distMinus{\gamma_1}||\mU_{[-1,1]}) + \varepsilon / 2$.
\end{clm}

\begin{proof}
	\newcommand{\tf}{{\tilde{f}}}
	
	We will show that $D_{KL}(\distPlus{\gamma^\dagger}||\mU_{[-1,1]}) \leq D_{KL}(\distPlus{\gamma_1}||\mU_{[-1,1]}) + \varepsilon / 4$. A similar argument can be used to show that $D_{KL}(\distMinus{\gamma^\dagger}||\mU_{[-1,1]}) \leq D_{KL}(\distMinus{\gamma_1}||\mU_{[-1,1]}) + \varepsilon / 4$.
	
	Let $f_1:[-1,1] \to [0,1]$ be the density function of $\distPlus{\gamma_1}$ and let $f^\dagger:[-1,1]\to[0,1]$ be the density function of $\distPlus{\gamma^\dagger}$. For $x >0$ let $h(x) = x\log(2x)$ and set $h(0)=0$. By definition:
	\[
	D_{KL}(\distPlus{\gamma^\dagger}||\mU_{[-1,1]}) - D_{KL}(\distPlus{\gamma_1}||\mU_{[-1,1]}) = \int_{-1}^{1} h(f^\dagger(x)) - h(f_1(x)) dx.
	\]
	By the mean value theorem, for every $x \in [-1,1]$ there exists some $\zeta_x$ between $f_1(x)$ and $f^\dagger(x)$ such that $h(f^\dagger(x)) - h(f_1(x)) = (f^\dagger(x) - f_1(x)) h'(\zeta_x)$. We claim that $|h'(\zeta_x)|$ is uniformly bounded by $\log(10 / \varepsilon_1)$. Indeed, since $h'(x) = 1/2 + \log(2x)$ this will follow if we can show that $\zeta_x \in [\varepsilon_1/20, 1]$ for every $x$. In turn, this will follow if we can show that $f^\dagger(x),f_1(x) \in [\varepsilon_1/20, 1]$ for every $x$. The upper bounds on $f^\dagger(x)$ and $f_1(x)$ are immediate from the definitions of $\distPlus{\gamma^\dagger}$ and $\distPlus{\gamma_1}$. For the lower bound, first note that (as observed earlier) $\gamma_1^+$ has density $\leq 1 - 0.2\varepsilon_1$. This implies that $f_1(x) \geq 1 - (1-0.2\varepsilon_1) = \varepsilon_1 / 5$. Similarly, by \cref{rmk:tgamma diag density upper bound}, $f^\dagger(x) \geq 1 - (1-\varepsilon_1/20) = \varepsilon_1/20$. We conclude:
	\begin{equation}\label{eq:diagonal entropy increase}
		\begin{split}
		\left| D_{KL}(\distPlus{\gamma^\dagger}||\mU_{[-1,1]}) - D_{KL}(\distPlus{\gamma_1}||\mU_{[-1,1]}) \right| & \leq \int_{-1}^1 |f^\dagger(x) - f_1(x)| |h'(\zeta_x)| dx\\
		& \leq \log(10/\varepsilon_1) \int_{-1}^1 |f^\dagger(x)-f_1(x)| dx.
		\end{split}
	\end{equation}
	We will now bound the integral $I \coloneqq \int_{-1}^1 |f^\dagger(x)-f_1(x)| dx$.
	
	We begin with the following observation: Suppose that  $d:\interval^2\to\R$ is the density function of a queenon $\delta$ and that $d:[-1,1]\to\R$ is the density function of $\distPlus{\delta}$. Then, for every $x \in [0,1]$ there holds $d^+(x) = 1 - \int_{-1/2+x}^{1/2} d(s,x-s) ds$. Similarly, for $x \in [-1,0]$, there holds $d^+(x) = 1 - \int_{-1/2}^{1/2-x} d(s,x-s) ds $. This allows us to express $I$ as an integral over the unit square.
	
	Recall that $g,g^\dagger:\interval^2\to\R$ are, respectively, the density functions of $\gamma_1$ and $\gamma^\dagger$. Thus, it holds that
	\begin{align*}
	I & \leq \int_{-1}^0 \int_{-1/2}^{1/2-x} |g^\dagger(s,x-s) - g(s,x-s)| ds dx + \int_{0}^1 \int_{-1/2+x}^{1/2} |g^\dagger(s,x-s) - g(s,x-s)| ds dx\\
	& = \int_{\interval^2} |g^\dagger - g| \stackrel{\text{\cref{rmk:small ell1 difference}}}{\leq} \varepsilon_1.
	\end{align*}
	Thus, \eqref{eq:diagonal entropy increase} implies
	\[
	D_{KL}(\distPlus{\gamma^\dagger}||\mU_{[-1,1]}) \leq D_{KL}(\distPlus{\gamma_1}||\mU_{[-1,1]}) + I \log(10/\varepsilon_1) \leq D_{KL}(\distPlus{\gamma_1}||\mU_{[-1,1]}) + \varepsilon_1 \log(10/\varepsilon_1).
	\]
	Provided $\varepsilon_1$ is sufficiently small with respect to $\varepsilon$, this completes the proof.
\end{proof}

We are ready to prove \cref{clm:tgamma props}\ref{itm:entropy approximation}. This will complete the proof of \cref{lem:approximation}.

\begin{proof}[Proof of \cref{clm:tgamma props}\ref{itm:entropy approximation}]
	We have:
	\[
	H_q(\tgamma) - H_q(\gamma) = \left(H_q(\tgamma) - H_q(\gamma^\dagger)\right) + \left(H_q(\gamma^\dagger) - H_q(\gamma_1)\right) + \left(H_q(\gamma_1) - H_q(\gamma)\right).
	\]
	Recalling that $\gamma_1 = \varepsilon_1 \kappa + (1-\varepsilon_1)\gamma$, the concavity of $H_q$ implies $H_q(\gamma_1) \geq \varepsilon_1 H_q(\kappa) + (1-\varepsilon_1)H_q(\gamma)$. By choosing $\varepsilon_1$ sufficiently small we may assume that $H_q(\gamma_1) - H_q(\gamma) > -\varepsilon/5$. Similarly, since $\tgamma = (1-\varepsilon_3)\gamma^\dagger + \varepsilon_3\kappa$ we may assume that $H_q(\tgamma) - H_q(\gamma^\dagger) > -\varepsilon/5$. Hence, it suffices to show that $H_q(\gamma^\dagger) - H_q(\gamma_1) > -3\varepsilon/5$.
	
	By definition:
	\begin{align*}
		H_q(\gamma^\dagger) - H_q(\gamma_1) = -D_{KL}(\gamma^\dagger||\mU_\square) + D_{KL}(\gamma_1||\mU_\square) & - D_{KL}(\distPlus{\gamma^\dagger}||\mU_{[-1,1]}) - D_{KL}(\distMinus{\gamma^\dagger}||\mU_{[-1,1]})\\
		& + D_{KL}(\distPlus{\gamma_1}||\mU_{[-1,1]}) + D_{KL}(\distMinus{\gamma_1}||\mU_{[-1,1]}).
	\end{align*}
	By \cref{clm:close square entropy,clm:close diagonal entropy} this is at least $-\varepsilon_1 - \varepsilon/2 > -3\varepsilon/5$, as desired.
\end{proof}

\section{Upper bound}\label{sec:upper bound}

\subsection{Entropy preliminaries}

In this section we prove the upper bound in Theorem \ref{thm:precise bounds}. The main tool is the entropy method. We briefly recall the definitions and properties we will use.

If $X$ is a random variable taking values in a finite set $S$ then its entropy is defined as
\[
H(X) = - \sum_{s \in S} \Prob \left[ X=s \right] \log \left( \Prob \left[ X=s \right] \right).
\]
The entropy function is strictly concave and so $H(X) \leq \log(|S|)$ with equality holding if and only if $X$ is uniform.

If $X,Y$ are two random variables taking values in a set $S \times T$ we write $(X\given Y=t)$ for the marginal distribution of $X$ given that $Y=t\in T$. The conditional entropy of $X$ given $Y$ is defined as:
\[
H(X \given Y) = \sum_{t \in T} \Prob [Y=t] H(X \given Y=t) = \E H(X \given Y=t).
\]
We will also use the chain rule. If $X_1,\ldots,X_n$ is a sequence of random variables then
\[
H(X_1,\ldots,X_n) = \sum_{i=1}^n H(X_i \given X_1,\ldots,X_{i-1}).
\]

\subsection{Proof overview}

In this subsection we outline the proof and give some intuition. We emphasize that the discussion is informal, and we do not rely on it for the proof.

Let $\gamma \in \Queenons$ and let $\varepsilon > 0$ be sufficiently small. Consider the following random process: Choose $q \in B_n(\gamma,\varepsilon)$ uniformly at random, and let $X_1,X_2,\ldots,X_n$ be a uniformly random ordering of the queens in $q$. Then
\begin{equation}\label{eq:q entropy vs sequential}
H(X_1,\ldots,X_n) = H(q) + \log(n!) = \log|B_n(\gamma,\varepsilon)| + \log(n!).
\end{equation}

We will bound $H(X_1,\ldots,X_n)$ using the chain rule. Specifically, we will bound $H(X_t \given X_1,\ldots, X_{t-1})$ for every $1 \leq t \leq n$. We do this by introducing additional random variables: Let $N$ be a large, fixed constant. For every $1 \leq t \leq n$, let $Y_t \in I_N$ be the $\alpha$ such that $X_t \in \alpha_n$. By the chain rule:
\begin{equation}\label{eq:Xt chain rule}
	H(X_t \given X_1,\ldots,X_{t-1}) = H(Y_t \given X_1,\ldots,X_{t-1}) + H(X_t \given X_1,\ldots,X_{t-1},Y_t).
\end{equation}
Since $q \in B_n(\gamma,\varepsilon)$, for every $\alpha \in I_N$ it holds that $|q \cap \alpha_n| \approx n \gamma(\alpha)$. Therefore:
\begin{equation}\label{eq:HYt}
	H(Y_t) \approx - \sum_{ \alpha \in I_N } \gamma(\alpha) \log(\gamma(\alpha)).
\end{equation}
It is not difficult to show that this holds even when conditioning on $X_1,\ldots,X_{t-1}$, provided $t$ is not too close to $n$. This is because the placements of these queens typically reflect the distribution $\gamma$.

We now wish to bound $H(X_t \given X_1,\ldots,X_{t-1},Y_t)$. Let $Q(t)$ be the partial $n$-queens configuration $\{X_1,\ldots,X_t\}$. Recall that a position is available in $Q(t)$ if it does not share a row, column, or diagonal with an element of $Q(t)$. For $\alpha \in I_N$ let $A_\alpha(t) = |\alpha_n \cap \mA_{Q(t)}|$. Conditioning on $X_1,\ldots,X_{t-1},Y_t$ it follows that $X_t$ given $X_1,\ldots,X_{t-1},Y_t$ is an element of $(Y_t)_n \cap \mA_{Q(t-1)}$. Therefore:
\[
H(X_t \given X_1,\ldots,X_{t-1},Y_t) \leq \E \left[\log \left( A_{Y_t}(t-1) \right)\right] \approx \sum_{\alpha \in I_N} \gamma(\alpha) \E \left[\log \left( A_\alpha (t-1) \right) \given Y_t=\alpha \right].
\]
Applying Jensen's inequality:
\begin{equation}\label{eq:HXt leq log E}
	H(X_t \given X_1,\ldots,X_{t-1},Y_t) \leq \sum_{\alpha\in I_N} \gamma(\alpha) \log \left( \E \left[ A_{\alpha}(t-1) \given Y_t=\alpha \right] \right).
\end{equation}

In order to bound $A_{\alpha}(t-1)$ we make the following observations: Every position in $\alpha_n$ shares its row and column with queens from $q$. Additionally, each position can share between $0$ and $2$ of its diagonals with queens from $q$. For an arbitrary $n$-queens configuration it would be challenging to proceed further. Fortunately, we know that $q \in B_n(\gamma,\varepsilon)$, and we use this to our advantage. Indeed, the number of plus-diagonals passing through $\alpha_n$ that are occupied by elements of $q$ is $\approx \gamma^+(\alpha)n$ and the number of occupied minus-diagonals is $\approx \gamma^-(\alpha_n)n$. The total number of each kind of diagonal passing through $\alpha_n$ is $\approx n/N$. If we assume that the occupied diagonals in each direction are approximately independent (over the choice of $q$), then there are  $\approx B_2(\alpha) \coloneqq |\alpha_n| N^2 \gamma^+(\alpha)\gamma^-(\alpha)$ positions threatened along both diagonals, $\approx B_1(\alpha) \coloneqq |\alpha_n| N\left(\gamma^+(\alpha) (1-N\gamma^-(\alpha)) + \gamma^-(\alpha)(1-N\gamma^+(\alpha)) \right)$ positions threatened by exactly one diagonal, and $\approx B_0(\alpha) \coloneqq |\alpha_n| (1-N\gamma^+(\alpha))(1-N\gamma^-(\alpha))$ positions unthreatened by diagonals. Now, if a position is threatened by $i$ diagonals then the probability that it is available at time $t$ is $\approx (1-t/n)^{2+i}$ (the ``$2$'' in the exponent accounts for the fact that every position shares its row and column with a queen). This is because it is available only if the $2+i$ queens threatening it are not in $Q(t-1)$ and these events are approximately independent. Moreover, these estimates hold even when conditioning on the outcome of $Y_t$. Therefore, for every $\alpha \in I_N$:
\begin{equation}\label{eq:heuristic expectation A}
	\E \left[ A_\alpha(t-1) \given Y_t=\alpha \right] \approx \E \left[ A_\alpha(t-1) \right] \approx \sum_{i=0}^2 B_i(\alpha) \left( 1 - \frac{t}{n} \right)^{2+i}.
\end{equation}

In light of \eqref{eq:q entropy vs sequential}, \eqref{eq:Xt chain rule}, \eqref{eq:HYt}, \eqref{eq:HXt leq log E}, and \eqref{eq:heuristic expectation A} we have
\[
\sum_{t=1}^n H(X_t \given X_1,\ldots,X_{t-1}) \leq \sum_{t=1}^n \sum_{ \alpha \in I_N } \gamma(\alpha) \left( - \log(\gamma(\alpha))  + \log \left( \sum_{i=0}^2 B_i(\alpha) \left( 1 - \frac{t}{n} \right)^{2+i} \right) \right).
\]
Hence, to obtain the desired upper bound on $|B_n(\gamma,\varepsilon)|$ it suffices to verify that
\[
\sum_{t=1}^n \sum_{ \alpha \in I_N } \gamma(\alpha) \left( - \log(\gamma(\alpha))  + \log \left( \sum_{i=0}^2 B_i(\alpha) \left( 1 - \frac{t}{n} \right)^{2+i} \right) \right) \approx n\log(n) + n H_q(\gamma) + \log(n!).
\]

Most of the assertions above can be justified routinely. There is one heuristic, however, that needs more work. This is the statement that the occupied plus-diagonals and the occupied minus-diagonals passing through $\alpha_n$ are distributed independently. At first glance it might not be clear why this is important. Indeed, one might make the mistake of thinking that every plus diagonal passing through $\alpha_n$ intersects every minus-diagonal passing through $\alpha_n$ in exactly one position. However, as every chess player will immediately point out, this is not the case --- diagonals on a chess board intersect only if they are both black or both white. Intuitively, our heuristic is justified by the idea that the entropy is maximized when the configuration is ``color-blind'', and black and white diagonals are equally likely to be occupied. In order to prove this we introduce a new limit object that includes information regarding the distribution of queens on board positions of each color.

\subsection{BW decompositions}\label{sec:BW decompositions}

\newcommand{\gammabw}{{(\gamma_b,\gamma_w)}}

Given a queenon, we will consider the various ways of decomposing it into a distribution of queens on black and white spaces. Given such a decomposition we will bound, from above, the number of $n$-queens configurations close to it. We will show that the bound is maximized when the partition is equitable.

\begin{definition}
	Let $\gamma \in \Queenons$. A \termdefine{BW-decomposition} of $\gamma$ is a pair $\gammabw$ of Borel measures on $\interval^2$ satisfying:
	\begin{enumerate}[({\bfseries{BW\arabic{enumi}}})]
		\item $\gamma_b + \gamma_w = \gamma$.
		
		\item\label{itm:BW subuniform} For every $-1 \leq a < b \leq 1$ both of the sets
		\[
		\left\{ (x,y) : a \leq x+y \leq b \right\}, \quad \left\{ (x,y) : a \leq y-x \leq b \right\}
		\]
		have measure at most $(b-a)/2$ under both $\gamma_b$ and $\gamma_w$.
	\end{enumerate}
	Let $\bw{\gamma}$ be the set of BW-decompositions of $\gamma$. Let $\BW = \bigcup_{\gamma \in \Queenons} \bw{\gamma}$. We endow it with the metric
	\[
	\distbw{\gammabw}{(\delta_b,\delta_w)} = \max \left\{ \dist{\gamma_b}{\delta_b}, \dist{\gamma_w}{\delta_w} \right\}.
	\]
	
	Let $\lambda$ denote the Lebesgue measure on $[-1,1]$. Given $\gammabw \in \BW$, for $i\in\{b,w\}$ we define the measures $\gamma_i^+,\gamma_i^-,\BWdistPlus{\gamma_i},\BWdistMinus{\gamma_i}$ on the interval $[-1,1]$ by setting
	\begin{align*}
	&\gamma_i^+(X) = \gamma_i \left( \left\{ (x,y) : x+y \in X \right\} \right),\\
	&\gamma_i^-(X) = \gamma_i \left( \left\{ (x,y) : x-y \in X \right\} \right),\\
	&\BWdistPlus{\gamma_i}(X) = \lambda(X)/2-\gamma_i^+(X),\\
	&\BWdistMinus{\gamma_i}(X) = \lambda(X)/2-\gamma_i^-(X)
	\end{align*}
	for every Borel set $X \subseteq [-1,1]$.
\end{definition}

We remark that the fact that $\BWdistPlus{\gamma_i}$ and $\BWdistMinus{\gamma_i}$ are positive probability measures follows from \ref{itm:BW subuniform}.

Let $\gammabw \in \bw{\gamma}$ for $\gamma \in \Queenons$. Observe that $\gammabw$ can be viewed as a probability measure on two copies of the unit square. Similarly, $(\BWdistPlus{\gamma_b},\BWdistPlus{\gamma_w})$ and $(\BWdistMinus{\gamma_b},\BWdistMinus{\gamma_w})$ can each be viewed as probability measures on two copies of the interval $[-1,1]$. With this in mind we define, for $N \in \N$, the discrete approximation of the KL divergence of $\gammabw$ with respect to the uniform distribution:
\begin{align*}
D^N \gammabw = \sum_{i\in \{b,w\}} \sum_{\alpha \in I_N} \gamma_i(\alpha) \log \left( \frac{2\gamma_i(\alpha)}{|\alpha|} \right).
\end{align*}
We define the function
\begin{align*}
G^N \gammabw \coloneqq - D^N\gammabw - D \left( \{ \BWdistPlus{\gamma_i}(\alpha) \}_{\alpha \in J_N, i \in \{b,w\}} \right) - D \left( \{ \BWdistMinus{\gamma_i}(\alpha) \}_{\alpha \in J_N, i \in \{b,w\}} \right) + 2\log 2 - 3.
\end{align*}
The reader should think of $G^N$ as a modification of the discrete Q-entropy function $H_q^N$ that is suitable for BW-decompositions.

Let $q$ be an $n$-queens configuration. Then $q$ can be partitioned into $q_b,q_w$, where $q_b$ consists of the queens occupying black positions (i.e., positions $(x,y)$ such that $x+y$ is even) and $q_w$ is the set of queens on white positions. We define a BW-decomposition $(\gamma_{q,b},\gamma_{q,w})$ as follows: For $i \in \{b,w\}$, let $\gamma_{q,i}$ be the measure that has constant density $n$ on every square $(-1/2+(x-1)/n,-1/2+x/n) \times (-1/2+(y-1)/n,-1/2+y/n)$ for $(x,y) \in q_i$ and density $0$ elsewhere. For $\gammabw \in \BW$ and $\varepsilon > 0$ let $B_n(\gammabw,\varepsilon)$ be the set of $n$-queens configurations $q$ such that $\distbw{\gammabw}{(\gamma_{q,b},\gamma_{q,w})} < \varepsilon$.

The main result of this section is an upper bound on $|B_n(\gammabw,\varepsilon)|$.

\begin{lem}\label{lem:bw upper bound}
	For all $\varepsilon > 0$ sufficiently small the following holds. Let $\gammabw \in \BW$. Set $N = \lceil \varepsilon^{-1/3} \rceil$. Then
	\[
	\limsup_{n \to \infty} \frac{|B_n(\gammabw,\varepsilon)|^{1/n}}{n} \leq \exp \left( G^N \gammabw + \varepsilon^{1/100} \right).
	\]
\end{lem}

Before proving Lemma \ref{lem:bw upper bound} we make the following observations.

\begin{obs}\label{obs:G properties}
	Let $\gamma \in \Queenons$ and $N \in \N$. The following hold.
	\begin{enumerate}
		\item $G^N$ is concave.
		
		\item $G^N \left( \frac{1}{2}(\gamma,\gamma) \right) = H_q^N(\gamma)$.
		
		\item $G^N$ is maximized on $\bw{\gamma}$ by $\frac{1}{2}(\gamma,\gamma)$.
		
		\item $\BW$ with the topology induced by $\distbwNoParam$ is compact.
	\end{enumerate}
\end{obs}

\begin{proof}
	$G^N$ is concave because the function $-x \log(x)$ is concave and $(\BWdistPlus{\gamma_b},\BWdistPlus{\gamma_w})$, $(\BWdistMinus{\gamma_b},\BWdistMinus{\gamma_w})$ are linear functions of $\gammabw$.
	
	The fact that $G^N \left( \frac{1}{2}(\gamma,\gamma) \right) = H_q^N(\gamma)$ is seen by unpacking the definitions.
	
	Let $\gammabw \in \bw{\gamma}$. Then $(\gamma_w,\gamma_b) \in \bw{\gamma}$ as well. $G^N$ is symmetric in $\gamma_b$ and $\gamma_w$. Thus $G^N(\gamma_w,\gamma_b) = G^N \gammabw$. By concavity:
	\begin{align*}
	G^N\gammabw = \frac{1}{2}\left(G^N\gammabw + G^N(\gamma_w,\gamma_b)\right) \leq G^N \left( \frac{1}{2}\left( \gammabw + (\gamma_w,\gamma_b) \right) \right) = G^N \left( \frac{1}{2}(\gamma,\gamma) \right).
	\end{align*}
	
	Compactness follows in much the same way as the analogous statement for queenons (Claim \ref{clm:queenons compact convex metric}): Every element of $\BW$ is, in particular, a Borel probability measure on two copies of $\interval^2$. Thus, $\BW$ is compact with respect to the weak topology. One then argues similarly to the proof of Claim \ref{clm:queenons compact convex metric} that $\distbwNoParam$ induces the weak topology on $\BW$.
\end{proof}

\subsection{Proof of Lemma \ref{lem:bw upper bound}}

\newcommand{\Tdef}{{\lfloor (1 - \varepsilon^{1/13})n \rfloor}}

We prove Lemma \ref{lem:bw upper bound} using the entropy method. Fix (a sufficiently small) $\varepsilon > 0$ and (a sufficiently large) $n \in \N$. We may assume $B_n(\gammabw, \varepsilon) \neq \emptyset$. Define the constant
\[
T \coloneqq \Tdef
\]
and recall that $N = \lceil \varepsilon^{-1/3} \rceil$ was defined in the lemma's statement.

Consider the following random process: Choose $q \in B_n(\gammabw,\varepsilon)$ uniformly at random and let $X_1,X_2,\ldots,X_n$ be a uniformly random ordering of the elements of $q$. Then
\begin{equation}\label{eq:sequential vs total entropy}
H(X_1,\ldots,X_n) = H(q) + \log(n!) = \log|B_n(\gammabw,\varepsilon)| + \log(n!).
\end{equation}

By the chain rule:
\[
H(X_1,\ldots,X_n) = \sum_{t=1}^n H(X_t \given X_1,\ldots,X_{t-1}).
\]

For every $\alpha \in I_N$ and $i \in \{b,w\}$, it holds that
\begin{equation}\label{eq:q_i close to gamma_i}
|\alpha_n \cap q_i| = (\gamma_i(\alpha) \pm 2\varepsilon )n.
\end{equation}

Now define the sequences $Y_1,Y_2,\ldots,Y_n$ and $Z_1,Z_2,\ldots,Z_n$, where $Y_t$ is equal to the $\alpha \in I_N$ such that $X_t \in \alpha_n$ and $Z_t = b$ if $X_t$ is on a black square and $Z_t = w$ otherwise.

\begin{clm}\label{clm:Y entropy}
	For every $1 \leq t \leq T$ it holds that
	\[
	H(Y_t,Z_t \given X_1,\ldots,X_{t-1}) = - D^N\gammabw + 2\log \left( 2N \right) \pm \varepsilon^{5/39}.
	\]
\end{clm}

To prove Claim \ref{clm:Y entropy} we introduce, for every $\alpha \in I_N$, $i \in \{b,w\}$, and $0 \leq t < T$ the random variable $W_{\alpha,i}(t)$, equal to the number of indices $1 \leq s \leq t$ such that $(Y_s,Z_s) = (\alpha, i)$. Observe that $\E W_{\alpha,i}(t) = |q_i \cap \alpha_n|t/n = (\gamma_i(\alpha) \pm 2\varepsilon) t$. Let $\mB(t)$ be the event that for some $\alpha \in I_N$ and $i \in \{b,w\}$, it holds that $| W_{\alpha,i}(t) - \gamma_i(\alpha)t | \geq 3 \varepsilon n$.

\begin{clm}\label{clm:W concentration upper bound}
	For every $0 \leq t < T$ there holds:
	\[
	\Prob \left[ \mB(t) \right] \leq \exp \left( - \Omega \left( n \right) \right).
	\]
\end{clm}

The proof relies on the following concentration inequality for random permutations which follows from arguments of McDiarmid (see, for example, \cite[Theorem 3.7]{mcdiarmid1998concentration} and the examples in the section that follows). The precise statement given here appears as \cite[Lemma 2.7]{keevash2018counting}.

\begin{thm}
	Let $S_n$ be the symmetric group of degree $n$, let $b > 0$, and let $f:S_n \to \R$ be a function satisfying: for every $\sigma \in S_n$ and every transposition $\tau$, $|f(\tau \circ \sigma) - f(\sigma)| < b$. Let $X$ be a uniformly random element of $S_n$. Then, for every $\lambda > 0$:
	\[
	\Prob \left[ |f(X) - \E [f(X)]| > \lambda \right] \leq 2 \exp \left( - \frac{\lambda^2}{2nb^2} \right).
	\]
\end{thm}

\begin{proof}[Proof of Claim \ref{clm:W concentration upper bound}]
	Let $\alpha \in I_N$ and let $i \in \{b,w\}$. Conditioning on $q$, $W_{\alpha,i}(t)$ is a function of the uniformly random permutation that determines the order $X_1,\ldots,X_n$. Furthermore, changing this order by a single transposition affects $W_{\alpha,i}(t)$ by at most $1$. We have already noted that $\E W_{\alpha,i}(t) = (\gamma_i(\alpha) \pm 2\varepsilon) t$. Therefore
	\[
	\Prob \left[ |W_{\alpha,i}(t) - \gamma_i(\alpha)t| > 3 \varepsilon n \right] \leq 2 \exp \left( - \frac{(\varepsilon n)^2}{2n} \right) = \exp \left( - \Omega (n) \right).
	\]
	The claim follows by applying a union bound to the $2|I_N|T$ choices for $i,\alpha,t$.
\end{proof}

\begin{proof}[Proof of Claim \ref{clm:Y entropy}]
	Observe that for every $\alpha \in I_N, i \in \{b,w\}$ and for any $X_1,\ldots,X_{t-1}$ such that $\mB^c(t-1)$ occurs:
	\begin{equation}\label{eq:typical distribution probability}
	\Prob \left[ (Y_t,Z_t) = (\alpha,i) \given X_1,\ldots,X_{t-1} \right] = \frac{(\gamma_i(\alpha) \pm 5\varepsilon)n - \gamma_i(\alpha)t}{n-t}
	= \gamma_i(\alpha) \pm \frac{5\varepsilon n}{n-T} = \gamma_i(\alpha) \pm \varepsilon^{11/13}.
	\end{equation}
	
	Recall that $(Y_t,Z_t)$ takes values in a set of size $2|I_N| = O(1)$. Therefore its entropy is bounded from above by $O(1)$ as well (regardless of any conditioning). Hence, by the law of total probability,
	\begin{align*}
	H(Y_t,Z_t \given X_1,\ldots,X_{t-1}) = & H(Y_t,Z_t \given X_1,\ldots,X_{t-1}, \mB(t-1)) \Prob[\mB(t-1)]\\
	& + H(Y_t,Z_t \given X_1,\ldots,X_{t-1}, \mB^c(t-1))(1 - \Prob[\mB(t-1)])\\
	\stackrel{\text{\cref{clm:W concentration upper bound}}}{=} & H(Y_t,Z_t \given X_1,\ldots,X_{t-1}, \mB^c(t-1)) \pm \exp \left( - \Omega \left( n \right) \right).
	\end{align*}
	
	By \eqref{eq:typical distribution probability}:
	\begin{equation}\label{eq:Y entropy}
	H (Y_t,Z_t \given X_1,\ldots,X_{t-1}, \mB^c(t-1))
	= - \sum_{\alpha \in I_N, i \in \{b,w\}} \left( \gamma_i(\alpha) \pm \varepsilon^{11/13} \right) \log \left( \gamma_i(\alpha) \pm \varepsilon^{11/13} \right).
	\end{equation}
	Let $B \subseteq I_N \times \{b,w\}$ be the set of indices $(\alpha,i)$ such that $\gamma_i(\alpha) \leq 2\varepsilon^{11/13}$. Observe that $|B| \leq 2|I_N| = O(\varepsilon^{-2/3})$. Continuing \eqref{eq:Y entropy}:
	\begin{align*}
	H & (Y_t,Z_t \given X_1,\ldots,X_{t-1}, \mB^c(t-1))\\
	= & - \sum_{(\alpha,i) \in B} (\gamma_i(\alpha) \pm \varepsilon^{11/13}) \log \left( \gamma_i(\alpha) \pm \varepsilon^{11/13} \right)
	\quad - \sum_{(\alpha,i) \notin B} (\gamma_i(\alpha) \pm \varepsilon^{11/13}) \log \left( \gamma_i(\alpha) \pm \varepsilon^{11/13} \right)\\
	= & - \sum_{\alpha \in I_N, i \in \{b,w\}} \gamma_i(\alpha) \log(\gamma_i(\alpha)) \pm \varepsilon^{6/39}\\
	= & - \sum_{\alpha \in I_N, i \in \{b,w\}} \gamma_i(\alpha) \log \left( \frac{2\gamma_i(\alpha)}{|\alpha|} \right) - \sum_{\alpha \in I_N, i \in \{b,w\}} \gamma_i(\alpha) \log \left( \frac{|\alpha|}{2} \right) \pm \varepsilon^{6/39}\\
	= & - D^N\gammabw - \sum_{\alpha \in I_N, i \in \{b,w\}} \gamma_i(\alpha) \log \left( \frac{|\alpha|}{2} \right) \pm \varepsilon^{6/39}.
	\end{align*}
	We turn our attention to the sum $\sum_{\alpha \in I_N, i \in \{b,w\}} \gamma_i(\alpha) \log(|\alpha|/2)$. Recall that $S_N \cup T_N$ is the partition of $I_N$ into squares and half-squares, respectively. For every square $\alpha \in S_N$ we have $|\alpha| = 1/(2N^2)$ and for every half-square $\alpha \in T_N$ we have $|\alpha| = 1/(4N^2)$. Thus:
	\[
	\sum_{\alpha \in I_N, i \in \{b,w\}} \gamma_i(\alpha) \log \left( \frac{|\alpha|}{2} \right) = -\sum_{\alpha \in I_N, i \in \{b,w\}} \gamma_i(\alpha) \log(4N^2) - \sum_{ \alpha \in T_N, i \in \{b,w\} } \gamma_i(\alpha) \log(2).
	\]
	The half-squares in $T_N$ are contained in four axis-parallel lines of width $1/(2N)$ each. Thus, since $\gamma_b+\gamma_w$ has uniform marginals, $\sum_{ \alpha \in T_N, i \in \{b,w\} } \gamma_i(\alpha) \log(2) \leq 2\log(2)/N$. Therefore:
	\[
	\sum_{\alpha \in I_N, i \in \{b,w\}} \gamma_i(\alpha) \log \left(\frac{|\alpha|}{2} \right) = -\log(4N^2) \pm \frac{2\log2}{N} = -2\log(2N) \pm \varepsilon^{6/39}.
	\]
	Hence:
	\[
	H (Y_t,Z_t \given X_1,\ldots,X_{t-1}) = - D^N\gammabw + 2 \log(2N) \pm \varepsilon^{5/39},
	\]
	as claimed.
\end{proof}

We will now estimate $H(X_t \given X_1,\ldots,X_{t-1},Y_t,Z_t)$.

For $(\alpha,i) \in I_N \times \{b,w\}$ and $0 \leq t < T$ let $\mA_{\alpha,i}(t)$ denote the set of available positions of color $i$ in $\alpha_n$ at time $t$. Let $A_{\alpha,i}(t) = |\mA_{\alpha,i}(t)|$. We note that given $X_1,\ldots,X_{t-1},Y_t,Z_t$, the queen $X_t$ is chosen from $\mA_{Y_t,Z_t}(t-1)$. Thus:
\[
H(X_t \given X_1,\ldots,X_{t-1},Y_t,Z_t) \leq \E [\log \left( A_{Y_t,Z_t}(t-1) \right)].
\]
For notational conciseness we define
\[
E_{\alpha,i}(t-1) = \E [A_{\alpha,i}(t-1) \given (Y_t,Z_t) = (\alpha,i) ],\qquad
p_{\alpha,i}(t) = \Prob \left[(Y_t,Z_t) = (\alpha,i)\right].
\]
By concavity of the logarithm:
\begin{equation}\label{eq:H(X_t) total prob}
\begin{split}
H(X_t \given X_1,\ldots,X_{t-1},Y_t,Z_t) & \leq \sum_{\alpha \in I_N, i \in \{b,w\}} p_{\alpha,i}(t) \log \left( E_{\alpha,i}(t-1) \right)\\
& \stackrel{\text{\eqref{eq:q_i close to gamma_i}}}{\leq} \sum_{\alpha \in I_N, i \in \{b,w\}} \gamma_i(\alpha) \log \left( E_{\alpha,i}(t-1) \right) + 2|I_N| \varepsilon \log (n^2)\\
& \leq \sum_{\alpha \in I_N, i \in \{b,w\}} \gamma_i(\alpha) \log \left( E_{\alpha,i}(t-1) \right) + 16N^2 \varepsilon \log(n).
\end{split}
\end{equation}

In order to estimate $E_{\alpha,i}(t-1)$ we look carefully at the positions in $\alpha_n$. Given $q$, each position of color $i$ in $\alpha_n$ falls into exactly one of the following categories:
\begin{itemize}
	\item It is a queen in $q$.
	
	\item It is not a queen, and the diagonals incident to it are unoccupied in $q$.
	
	\item It is not a queen and exactly one of the diagonals incident to it is occupied in $q$.
	
	\item It is not a queen and both diagonals incident to it are occupied in $q$.
\end{itemize}
Denote the number of positions in each category by, respectively, $C(q,\alpha,i)$, $D_0(q,\alpha,i)$, $D_1(q,\alpha,i)$, $D_2(q,\alpha,i)$ (for the $D$s, the subscript denotes the number of \textit{diagonal threats} for each position). Although these are random variables, the fact that $q \in B_n(\gammabw,\varepsilon)$ means they cannot vary too much. For $\alpha,i \in I_N \times \{b,w\}$ define:
\begin{align*}
&D_0(\alpha,i) = \BWdistPlus{\gamma_i}(\alpha) \BWdistMinus{\gamma_i}(\alpha),\\
&D_1(\alpha,i) = \BWdistPlus{\gamma_i}(\alpha) \gamma_i^-(\alpha) + \BWdistMinus{\gamma_i}(\alpha) \gamma_i^+(\alpha),\\
&D_2(\alpha,i) = \gamma_i^+(\alpha) \gamma_i^-(\alpha).
\end{align*}

The following observation can be proved by expanding the definitions.

\begin{obs}\label{obs:D_j weighted sum}
	For every $(\alpha,i) \in I_N \times \{b,w\}$ and every $0 \leq t < T$:
	\[
	\sum_{j=0}^2 D_j(\alpha,i) \left( 1 - \frac{t}{n} \right)^{j+2} = \frac{1}{4N^2}\left( 1 - \frac{t}{n} \right)^2 \left( 1 - 2N \gamma_i^+(\alpha) \frac{t}{n} \right) \left( 1 - 2N \gamma_i^-(\alpha) \frac{t}{n} \right).
	\]
\end{obs}

\begin{clm}\label{clm:partition by threat number}
	The following hold for every $q \in B_n(\gammabw,\varepsilon)$ and every $(\alpha,i) \in S_N \times \{b,w\}$.
	\begin{align*}
	& C(q,\alpha,i) \leq n,\\
	& D_0(q,\alpha,i) = \left( D_0(\alpha,i) \pm O (\varepsilon) \right)n^2,\\
	& D_1(q, \alpha,i) = \left( D_1(\alpha,i) \pm O(\varepsilon) \right)n^2,\\
	& D_2(q, \alpha,i) = \left( D_2(\alpha,i) \pm O(\varepsilon) \right)n^2.
	\end{align*}
\end{clm}

\begin{proof}
	Let $q \in B_n(\gammabw,\varepsilon)$ and $(\alpha,i) \in S_N \times \{b,w\}$. $C(q,\alpha,i)$ is the number of color $i$ queens in $\alpha_n$. Since $q$ contains $n$ queens, $C(q,\alpha,i) \leq n$.
	
	Let $P^+,P^- \subseteq I_N$ be those elements sharing, respectively, their plus-diagonal and minus-diagonal with $\alpha$. Then, by \eqref{eq:q_i close to gamma_i}, for $* \in \{+,-\}$:
	\[
	\sum_{ \beta \in P^*} |q_i \cap \beta_n| = \sum_{\beta \in P^*} (\gamma_i(\beta) \pm 2\varepsilon)n = \left(\gamma_i^*(\alpha) \pm 2N\varepsilon \right)n.
	\]
	Now, every plus-diagonal containing an element of $\bigcup_{\beta \in P^+} (q_i \cap \beta_n)$ intersects every minus-diagonal containing an element of $\bigcup_{\beta \in P^-} (q_i \cap \beta_n)$ in exactly one color-$i$ square in $\alpha_n$. Furthermore, every color-$i$ square in two occupied diagonals (including the color-$i$ queens in $\alpha_n$) is obtained in this way. Hence:
	\begin{align*}
	D_2(q,\alpha,i) & = \left(\gamma_i^+(\alpha) \pm 2N\varepsilon \right) \left(\gamma_i^-(\alpha) \pm 2N\varepsilon \right)n^2 \pm C(q,\alpha,i)\\
	& = \left(D_2(\alpha,i) \pm 2N \varepsilon (\gamma_i^+(\alpha) + \gamma_i^-(\alpha) + 2N\varepsilon) \right) n^2 \pm n.
	\end{align*}
	Since $\gamma$ has sub-uniform diagonal marginals, $\gamma_i^+(\alpha) + \gamma_i^-(\alpha) \leq 2/N$. Additionally, by definition, $(N\varepsilon)^2 < \varepsilon$. Hence:
	\[
	D_2(q,\alpha,i) = \left(D_2(\alpha,i) \pm 8\varepsilon\right)n^2,
	\]
	as desired.
	
	The bounds on $D_1(\alpha,i)$ and $D_0(\alpha,i)$ are proved similarly, after noting that for $* \in \{+,-\}$, the number of unoccupied $*$-diagonals of color $i$ passing through $\alpha_n$ is $(1/(2N) - \gamma_i^*(\alpha) \pm 4N\varepsilon)n = (\overline{\gamma_i}^*(\alpha) \pm 4N\varepsilon)n$.
\end{proof}

Claim \ref{clm:partition by threat number} allows us to estimate $\E [A_{\alpha,i}(t-1)]$. We remark that Claim \ref{clm:partition by threat number} only holds for $\alpha \in S_N$. The half-squares in $T_N$ constitute only a small part of the measure of $\gamma$ and so for them the weak bound $A_{\alpha,i}(t-1) \leq n^2$ is all we need.

\begin{clm}\label{clm:upper bound expected available}
	For every $1 \leq t \leq T$ and every $\alpha,i \in S_N \times \{b,w\}$ it holds that
	\[
	E_{\alpha,i}(t-1) = \left( 1 \pm O \left( N^2\varepsilon \right) \right) \left( 1 - \frac{t}{n} \right)^2 \left( 1 - 2N \gamma_i^+(\alpha) \frac{t}{n} \right) \left( 1 - 2N \gamma_i^-(\alpha) \frac{t}{n} \right) \frac{n^2}{4N^2}.
	\]
\end{clm}

\begin{proof}
	Fix $q$, $\alpha \in S_N, i \in \{b,w\}$ and $1 \leq t \leq T$. For $j = 0,1,2$ there are $D_j(q,\alpha,i)$ positions in $\alpha_n$ that are not queens and exactly $j$ of the diagonals incident to them are occupied. Hence, for each position counted by $D_j(q,\alpha,i)$, the probability that it is available at time $t-1$ is $\oneoone(1-t/n)^{j+2}$. Therefore:
	\begin{align*}
	&E_{\alpha,i}(t-1) = \oneoone \sum_{j=0}^2 \left( 1 - \frac{t}{n} \right)^{j+2} D_j(q,\alpha,i) \pm C(q,\alpha,i)\\
	&\stackrel{\text{Claim \ref{clm:partition by threat number}}}{=} \oneoone \sum_{j=0}^2 \left( 1 - \frac{t}{n} \right)^{j+2} \left(D_j(\alpha,i) \pm O(\varepsilon)\right)n^2 \pm n\\
	&\stackrel{\text{Observation \ref{obs:D_j weighted sum}}}{=} \left( 1 - \frac{t}{n} \right)^2 \left( 1 - 2N \gamma_i^+(\alpha) \frac{t}{n} \right) \left( 1 - 2N \gamma_i^-(\alpha) \frac{t}{n} \right) \frac{n^2}{4N^2} \pm O\left( \varepsilon n^2 \right).
	\end{align*}
	Finally, since $t/n \leq T/n = 1-\Omega(\varepsilon^{1/13})$ and $\gamma_i^+(\alpha) \leq 1/(2N)$, each of $1-t/n,1-2N\gamma_i^+(\alpha)t/n$ and $1-2N\gamma_i^-(\alpha)t/n$ is $\Omega \left( \varepsilon^{1/13} \right)$. Hence:
	\[
	E_{\alpha,i}(t-1) = \left( 1 \pm O \left( N^2 \varepsilon^{9/13} \right) \right) \left( 1 - \frac{t}{n} \right)^2 \left( 1 - 2N \gamma_i^+(\alpha) \frac{t}{n} \right) \left( 1 - 2N \gamma_i^-(\alpha) \frac{t}{n} \right) \frac{n^2}{4N^2},
	\]
	as claimed.
\end{proof}

Continuing from \eqref{eq:H(X_t) total prob} and using the fact that $16N^2 \varepsilon \leq 17/N$:
\begin{align*}
H & (X_t \given X_1,\ldots,X_{t-1},Y_t,Z_t) \leq \sum_{\alpha \in I_N, i \in \{b,w\}} \gamma_i(\alpha) \log \left( E_{\alpha,i}(t-1) \right) + 16 N^2 \varepsilon \log(n)\\
& \leq \sum_{\alpha \in S_N, i \in \{b,w\}} \gamma_i(\alpha) \log (E_{\alpha,i}(t-1)) + \sum_{\alpha \in T_N, i \in \{b,w\}} \gamma_i(\alpha) \log (E_{\alpha,i}(t-1)) + \frac{17}{N} \log(n).
\end{align*}
As mentioned above we bound the second sum using the trivial bound $E_{\alpha,i}(t-1) \leq n^2$ and the fact that the half-squares in $T_N$ are contained in four axis parallel rectangles of width $\leq 1/(2N)$. We now use Claim \ref{clm:upper bound expected available} to bound the contribution of the squares in $S_N$.
\begin{align*}
H & (X_t \given X_1,\ldots,X_{t-1},Y_t,Z_t)\\
\leq & \sum_{\alpha \in S_N, i \in \{b,w\}} \gamma_i(\alpha) \log (E_{\alpha,i}(t-1)) + \frac{4}{N} \log (n) + \frac{17}{N} \log(n)\\
\leq & \sum_{\alpha \in S_N, i \in \{b,w\}} \gamma_i(\alpha) \log \left( 1 \pm O(N^2\varepsilon^{9/13}) \right) + 2\sum_{\alpha \in S_N, i \in \{b,w\}} \gamma_i(\alpha) \log \left( 1 - \frac{t}{n} \right)\\
& + \sum_{\alpha \in S_N, i \in \{b,w\}} \gamma_i(\alpha) \log \left( 1 - 2N \gamma_i^+(\alpha) \frac{t}{n} \right) + \sum_{\alpha \in S_N, i \in \{b,w\}} \gamma_i(\alpha) \log \left( 1 - 2N \gamma_i^-(\alpha) \frac{t}{n} \right)\\
& + \sum_{\alpha \in S_N, i \in \{b,w\}} \gamma_i(\alpha) \log \left( \frac{n^2}{4N^2} \right) + \frac{25}{N}\log(n).
\end{align*}

We will now bound the contribution of each of the terms as we sum over $t$. To begin:
\[
\sum_{t=0}^{T-1} \sum_{\alpha \in S_N, i \in \{b,w\}} \gamma_i(\alpha) \log \left( 1 \pm O(N^2\varepsilon^{9/13}) \right) = O\left( N^2\varepsilon^{9/13} T \right) = O \left( \varepsilon^{1/39} n \right).
\]
Next:
\[
\sum_{t=0}^{T-1} \sum_{\alpha \in S_N, i \in \{b,w\}} \gamma_i(\alpha) \log \left( \frac{n^2}{4N^2} \right) \leq n \log \left( \frac{n^2}{4N^2} \right).
\]

Now:
\begin{align*}
\sum_{t=0}^{T-1} \sum_{\alpha \in S_N, i \in \{b,w\}} & \gamma_i(\alpha) \log \left( 1 - \frac{t}{n} \right)
\leq \left( 1 - O \left( \frac{1}{N} \right) \right) \sum_{t=1}^T \log \left( 1 - \frac{t}{n} \right)\\
& \stackrel{\text{Claim \ref{clm:log integral}}}{\leq} \left( 1 - O \left( \frac{1}{N} \right) \right) \left( -1 \pm \varepsilon^{2/39} \right)n = -n \pm O \left( \varepsilon^{2/39} n \right).
\end{align*}

We now note that
\begin{align*}
\sum_{\alpha \in S_N, i \in \{b,w\}} & \gamma_i(\alpha) \log \left( 1 - 2N \gamma_i^+(\alpha) \frac{t}{n} \right)\\
& \leq \sum_{\alpha \in I_N, i \in \{b,w\}} \gamma_i(\alpha) \log \left( 1 - 2N \gamma_i^+(\alpha) \frac{t}{n} \right) + \frac{4}{N} \left| \log \left( 1 - \frac{T}{n} \right) \right|\\
& = \sum_{ \alpha \in J_N, i \in \{b,w\}} \gamma_i^+(\alpha) \log \left( 1 - 2N\gamma_i^+(\alpha) \frac{t}{n} \right) + \frac{4}{N} \left| \log \left( 1 - \frac{T}{n} \right) \right|.
\end{align*}
Applying \cref{clm:log integral}:
\begin{align*}
\sum_{t=1}^T & \sum_{\alpha \in S_N, i \in \{b,w\}} \gamma_i(\alpha) \log \left( 1 - 2N \gamma_i^+(\alpha) \frac{t}{n} \right)\\
& \leq - \sum_{\alpha \in J_N, i \in\{b,w\} } \frac{n}{2N} \left( (1 - 2N\gamma_i^+(\alpha))\log(1 - 2N\gamma_i^+(\alpha)) + 2N\gamma_i^+(\alpha) \right)\\
& \qquad + \frac{2|J_N|}{2N} 3(n-T)|\log(1-T/n)|\\
& = -n \sum_{ \alpha \in J_N, i \in \{b,w\} } \BWdistPlus{\gamma_i}(\alpha) \log \left( 2N \BWdistPlus{\gamma_i}(\alpha) \right) - n + 12(1-T/n) \left|\log \left( 1 - T/n \right)\right| n\\
& = - n D \left(\{ \BWdistPlus{\gamma_i}(\alpha) \}_{\alpha \in I_N, i \in \{b,w\}} \right) + n \log(2) - n + 12(1-T/n) \left|\log \left( 1 - T/n \right)\right| n.
\end{align*}
Similarly:
\begin{align*}
\sum_{t=1}^T \sum_{\alpha \in S_N, i \in \{b,w\}} & \gamma_i(\alpha) \log \left( 1 - 2N \gamma_i^-(\alpha) \frac{t}{n} \right)\\
& \leq - n D \left(\{ \BWdistMinus{\gamma_i}(\alpha) \}_{\alpha \in I_N, i \in \{b,w\}} \right) + n \log(2) - n + 12(1-T/n) \left|\log \left( 1 - T/n \right)\right| n.
\end{align*}

As a consequence we obtain:
\begin{equation}\label{eq:H X_t bound}
\begin{split}
\sum_{t=1}^T & H(X_t \given X_1,\ldots,X_{t-1},Y_t,Z_t)\\
\leq & - n D \left(\{ \BWdistPlus{\gamma_i}(\alpha) \}_{\alpha \in I_N, i \in \{b,w\}} \right) - n D \left(\{ \BWdistMinus{\gamma_i}(\alpha) \}_{\alpha \in I_N, i \in \{b,w\}} \right)\\
& + 2n \log(2) - 4n + n \log \left( \frac{n^2}{4N^2} \right) + 24(1-T/n) \left|\log \left( 1 - T/n \right)\right| n + O \left( \varepsilon^{1/39} n \right)\\
\leq & - n D \left(\{ \BWdistPlus{\gamma_i}(\alpha) \}_{\alpha \in I_N, i \in \{b,w\}} \right) - n D \left(\{ \BWdistMinus{\gamma_i}(\alpha) \}_{\alpha \in I_N, i \in \{b,w\}} \right)\\
& - 2n \log(2N) - 4n + 2n\log(2) + 2n \log \left( n \right) + O \left( \varepsilon^{1/39} n \right).
\end{split}
\end{equation}

We are ready to prove Lemma \ref{lem:bw upper bound}.

\begin{proof}[Proof of Lemma \ref{lem:bw upper bound}]
	By the chain rule:
	\[
	H(X_1,\ldots,X_n) = \sum_{t=1}^n H(Y_t,Z_t \given X_1,\ldots,X_{t-1}) + \sum_{t=1}^n H(X_t \given X_1,\ldots,X_{t-1},Y_t,Z_t).
	\]
	By Claim \ref{clm:Y entropy} and using the fact that for every $t$, $H(Y_t,Z_t \given X_1,\ldots,X_{t-1}) \leq \log(2|I_N|)$:
	\begin{align*}
	\sum_{t=1}^n H(Y_t,Z_t \given X_1,\ldots,X_{t-1}) \leq - n D^N\gammabw + 2n\log(2N) + O \left( \varepsilon^{1/39} n \right).
	\end{align*}
	Together with \eqref{eq:H X_t bound} this implies
	\begin{align*}
	H  (X_1,\ldots,X_n) \leq & \\
	& - n D^N\gammabw - n D \left(\{ \BWdistPlus{\gamma_i}(\alpha) \}_{\alpha \in I_N, i \in \{b,w\}} \right) - n D \left(\{ \BWdistMinus{\gamma_i}(\alpha) \}_{\alpha \in I_N, i \in \{b,w\}} \right)\\
	& - 4n + 2n \log \left( n \right) + 2n\log(2) + O \left( \varepsilon^{1/39} n \right)\\
	= &\ n G^N\gammabw + 2n\log(n) - n + O \left( \varepsilon^{1/39} n \right)\\
	\leq &\ nG^N\gammabw + n \log(n) + \log(n!) + \varepsilon^{1/100} n.
	\end{align*}
	Therefore, by \eqref{eq:sequential vs total entropy}:
	\[
	\log \left( |B_n(\gammabw,\varepsilon)| \right) = H(X_1,\ldots,X_n) - \log(n!) \leq nG^N\gammabw + n\log(n) + \varepsilon^{1/100} n .
	\]
	Hence
	\[
	\frac{|B_n(\gammabw,\varepsilon)|^{1/n}}{n} \leq \exp \left( G^N\gammabw + \varepsilon^{1/100} \right),
	\]
	proving the lemma.
\end{proof}

We will now use Lemma \ref{lem:bw upper bound} to prove the upper bound in Theorem \ref{thm:precise bounds}.

\begin{proof}[Proof of Theorem \ref{thm:precise bounds} upper bound]
	Define
	\[
	X \coloneqq \left\{ \gammabw \in \BW : \dist{\gamma_b + \gamma_w}{\gamma} \leq \varepsilon \right\}.
	\]
	Note that $X \subseteq \BW$ is closed and therefore compact. Let $(\gamma_b^1,\gamma_w^1),\ldots,(\gamma_b^M,\gamma_w^M) \in X$ be such that $X \subseteq \bigcup_{i=1}^M B_\varepsilon (\gamma_b^i,\gamma_w^i)$. Observe that
	\[
	B_n(\gamma,\varepsilon) \subseteq \bigcup_{i=1}^M B_n((\gamma_b^i,\gamma_w^i), \varepsilon).
	\]

	Now, by Lemma \ref{lem:bw upper bound}, for each $i$:
	\[
	|B_n((\gamma_b^i,\gamma_w^i),\varepsilon)| \leq n^n \exp \left( n \left(G^N(\gamma_b^i,\gamma_w^i) + \varepsilon^{1/100} + \oone \right) \right)
	\]
	where $N = \lfloor \varepsilon^{-1/3} \rfloor$.
	By Observation \ref{obs:G properties} $G^N(\gamma_b^i,\gamma_w^i) \leq H_q^N(\gamma_{b}^i + \gamma_{w}^i)$. Since $(\gamma_b^i,\gamma_w^i) \in X$ by definition $\dist{\gamma_b^i + \gamma_w^i}{\gamma} \leq \varepsilon$. Therefore by Claim \ref{clm:H_q^N continuity bound} ${H_q^N(\gamma_{b}^i + \gamma_{w}^i)} \leq H_q^N(\gamma) + \varepsilon^{1/4}$. Therefore, for every $1 \leq i \leq M$:
	\begin{align*}
	\frac{|B_n((\gamma_b^i,\gamma_w^i),\varepsilon)|}{n^n} \leq \exp \left( n \left( H_q^N(\gamma) + \varepsilon^{1/4} + \varepsilon^{1/100} \right) \right)
	\leq \exp \left( n \left( H_q^N(\gamma) + 2\varepsilon^{1/100} \right) \right).
	\end{align*}
	Therefore:
	\[
	\frac{|B_n(\gamma,\varepsilon)|}{n^n} \leq M \exp \left( n \left( H_q^N(\gamma) + 2\varepsilon^{1/100} \right) \right) \leq \exp \left( n \left( H_q^N(\gamma) + \varepsilon^{1/200} \right) \right),
	\]
	proving the theorem.
\end{proof}

\section{Lower bound}\label{sec:lower bound}

In this section we prove the lower bound in Theorem \ref{thm:precise bounds}.

Let $\gamma$ be a queenon, let $\varepsilon > 0$, and let $n \in \N$. We may assume that $H_q(\gamma) > -\infty$ and that $\varepsilon$ is sufficiently small in terms of $\gamma$. We will describe a randomized algorithm that constructs an element of $B_n(\gamma,\varepsilon)$. We will derive the lower bound by counting the number of possible outcomes.

The algorithm has two phases: a random phase, in which most of the queens are placed on the board, and a correction phase, in which a small number of modifications are made to obtain a complete configuration.

It is helpful to replace $\gamma$ with a queenon that is close to $\gamma$ and has some additional desirable properties. By \cref{lem:approximation} there exists an $N$-step queenon $\delta$ with the following properties:

\begin{enumerate}[({\bfseries{D\arabic{enumi}}})]
	\item\label{itm:delta close to gamma} $\dist{\gamma}{\delta} < \varepsilon^2$.
	
	\item\label{itm:delta approximates entropy} $H_q(\delta) > H_q(\gamma) - \varepsilon |H_q(\gamma)|$.
	
	\item\label{itm:everywhere positive density} $\delta$ has density $\Theta(1)$ everywhere.
	
	\item\label{itm:diagonals not full} There exists a constant $\eta < 1$ such that $\delta^+$ and $\delta^-$ have density $\leq \eta$ everywhere.
\end{enumerate}
Furthermore, since every $N$-step queenon is also an $M$-step queenon for every $M$ that is a multiple of $N$, we may (and do) assume that $N \geq \varepsilon^{-2}$.

The fact that $\delta$ has positive density everywhere will make it easier to find the absorbers required for the correction phase of the algorithm. Additionally, \ref{itm:diagonals not full} ensures that every diagonal has probability bounded away from $1$ of being occupied in the random phase of the algorithm.

Choose a sufficiently large $K \in \N$ (that may depend on $\varepsilon$ and $\gamma$ but not $n$) and define:
\[
M \coloneqq \lfloor n^{0.1} \rfloor N, \quad T \coloneqq n - \lfloor n^{1-1/K^2} \rfloor,
\]
Observe that $N$ divides $M$, so $\delta$ is an $M$-step queenon. We now describe the first phase of the algorithm. Recall from \cref{sec:notation} that for $\alpha \in I_M$, the set $\alpha_n \subseteq [n]^2$ is (roughly) the set of board positions that correspond to $\alpha$.

\begin{algorithm}\label{alg:random}
	\hfill
	\begin{itemize}
		\item Let $Y_1,Y_2,\ldots,Y_T \in I_M$ be i.i.d.\ random variables, where for every $\alpha \in I_M$, $\Prob [Y_1 = \alpha] = \delta(\alpha)$.
		
		\item Set $Q(0) = \emptyset$.
		
		\item For every $0 < t \leq T$:
		
		\begin{itemize}
			\item Let $\mA_{Y_t}(t-1)$ be the set of available positions in $(Y_t)_n$. If $\mA_{Y_t}(t-1) = \emptyset$ abort, define $X_t=X_{t+1}=\ldots=X_T = *$, and set $Q(t), Q(t+1), \ldots, Q(T)$ all equal to $Q(t-1)$.
			
			\item Otherwise, choose $X_t \in \mA_{Y_t}(t-1)$ uniformly at random and set $Q(t) = Q(t-1) \cup \{ X_t \}$.
		\end{itemize}
	\end{itemize}
\end{algorithm}

We will show that w.h.p.\ Algorithm \ref{alg:random} does not abort. We will also calculate the entropy $H(X_t \given X_1,\ldots,X_{t-1})$ which will allow us to estimate the number of possible outcomes. By design, for each $\alpha \in I_M$ the number of queens placed in $\alpha_n$ is $\approx \gamma(\alpha)n$. Hence, we expect that any queen configuration close to $Q(T)$ (i.e., the outcome of Algorithm \ref{alg:random}) is an element of $B_n(\gamma,\varepsilon)$.

In the second phase of the algorithm we seek to make a small number of modifications to $Q(T)$ in order to obtain an $n$-queens configuration. The key is the idea of \textit{absorption}, which we now illustrate: Suppose $Q$ is a partial $n$-queens configuration that does not cover row $r$ and column $c$. We wish to obtain a partial $n$-queens configuration $Q'$ that covers all rows and columns covered by $Q$ and also covers row $r$ and column $c$. We might try adding $(c,r)$ to $Q$, but if either of the diagonals incident to $(c,r)$ is occupied this will not work. Instead, we look for a queen $(x,y) \in Q$ satisfying:
\begin{enumerate}
	\item\label{itm:distinct diags} $(c,y)$ and $(x,r)$ do not share a diagonal (equivalently, $(c,r)$ and $(x,y)$ do not share a diagonal) and
	
	\item\label{itm:free diags} none of the (four) diagonals containing $(c,y)$ or $(x,r)$ are occupied.
\end{enumerate}
Supposing such a queen exists, we observe that $Q' \coloneqq \left( Q \setminus \{(x,y)\} \right) \cup \{(c,y),(x,r)\}$ is a partial $n$-queens configuration satisfying the conditions above. In this way, we have \textit{absorbed} row $r$ and column $c$ into our configuration. We call such a queen an \termdefine{absorber for $(c,r)$ in $Q$} (see Figure \ref{fig:absorber}). We denote the set of absorbers for $(c,r)$ in $Q$ by $\mB_Q(c,r)$.

\begin{figure}
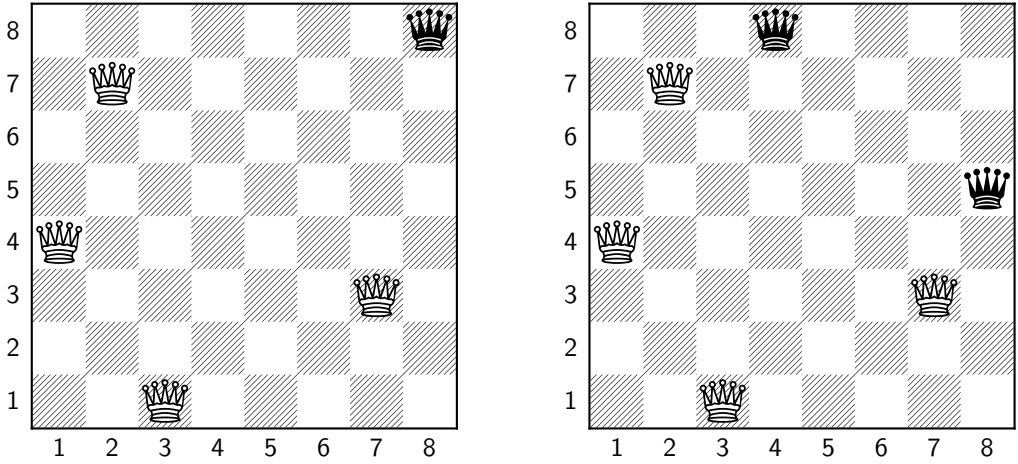

	\centering
	{	\def\whitepieces{qa4, qb7, qc1, qg3}
		\def\blackpieces{qh8}
		\chessboard[ showmover=false,labelbottomformat=\arabic{filelabel}, setwhite=\whitepieces,addblack=\blackpieces]
		\def\whitepieces{qa4, qb7, qc1, qg3}
		\def\blackpieces{qd8,qh5}
		\chessboard[ showmover=false,labelbottomformat=\arabic{filelabel}, setwhite=\whitepieces,addblack=\blackpieces]
	}
	\caption{In the partial $8$-queens configuration on the left the black queen at $(8,8)$ is an absorber for the square $(4,5)$. In the configuration on the right the queen at $(8,8)$ has been removed, while the black queens at $(4,8)$ and $(8,5)$ have been added, thus absorbing row $5$ and column $4$ into the configuration.}\label{fig:absorber}
\end{figure}

The following algorithm attempts to use absorbers to complete $Q(T)$. We remark that the number of uncovered rows in $Q(T)$ is always equal to the number of uncovered columns, and that in the (typical) case that \cref{alg:random} did not abort these are both equal to $n-T = \lfloor n^{1-1/K^2} \rfloor$.

\begin{algorithm}\label{alg:absorption}\hfill
	\begin{itemize}
		\item Let $L_R$ and $L_C$ be, respectively, the sets of rows and columns not covered by $Q(T)$. Set $k \coloneqq |L_R|$.
		
		\item Let $(c_1,r_1),(c_2,r_2),\ldots,(c_{k}, r_{k})$ be an arbitrary matching of $L_C$ to $L_R$.
		
		\item Set $R(0) \coloneqq Q(T)$.
		
		\item For $i=1,2,\ldots,k$:
		
		\begin{itemize}
			\item If $\mB_i \coloneqq \mB_{R(i-1)}(c_i,r_i) = \emptyset$ abort. 
			\item Otherwise, choose some $(x_i,y_i) \in \mB_i$ and set
			\[
			R(i) \coloneqq \left( R(i-1) \setminus \{(x_i,y_i)\} \right) \cup \{(x_i,r_i),(c_i,y_i)\}.
			\]
		\end{itemize}
	\end{itemize}
\end{algorithm}

Clearly, if Algorithm \ref{alg:absorption} does not abort then $R(k)$ is an $n$-queens configuration. In Section \ref{sec:absorbers} we show that w.h.p.\ $Q(T)$ satisfies a combinatorial condition that guarantees the success of Algorithm \ref{alg:absorption}.

\begin{rmk}
	The absorption procedure described above was introduced by Luria and the author in \cite{luria2021lower}. There, it was used in combination with a simple random greedy algorithm to show that $\mQ(n) \geq ((1-\oone))n e^{-3})^n$. While the analysis of Algorithm \ref{alg:absorption} shares some details with \cite{luria2021lower}, there are additional difficulties due to the fact that $\gamma$ may be far from uniform.
\end{rmk}

We analyze Algorithm \ref{alg:random} in Section \ref{sec:random analysis}. We analyze Algorithm \ref{alg:absorption} in Section \ref{sec:absorbers}. Then, in Section \ref{sec:lower bound proof} we put everything together and prove the lower bound in Theorem \ref{thm:precise bounds}.

\subsection{Analysis of Algorithm \ref{alg:random}}\label{sec:random analysis}

The analysis of Algorithm \ref{alg:random} is somewhat technical and calls for motivation. The overarching intuition is that because each $Y_t$ is distributed according to $\{\delta(\alpha)\}_{\alpha \in I_M}$, after $t$ steps of the process approximately $\delta(\alpha)t$ queens have been placed in $\alpha_n$. Thus, $Q(t)$ ``looks like'' a random size-$t$ subset of a random element of $B_n(\delta,\varepsilon)$. Indeed, an outside observer may not know if the process $\{Q(t)\}_{t=0}^T$ is governed by Algorithm \ref{alg:random} or by choosing $q \in B_n(\gamma,\varepsilon)$ uniformly at random and revealing its queens in a random order (though this is not literally true in an information-theoretic sense).

In order to analyze Algorithm \ref{alg:random} we need to track the distribution of available positions on the board. For example, we will need to know the number of available positions in each row. Our general strategy is to track random variables by showing they are close to smooth trajectory functions. However, this means we cannot track available positions directly: Whenever a queen is added to a row the number of available positions it contains jumps down to zero. Thus, we cannot expect this random variable to follow a smooth trajectory. To overcome this we define a related notion.

\begin{definition}
	Let $Q$ be a partial $n$-queens configuration. A position $(x,y) \in [n]^2$ is \termdefine{row-safe} in $Q$ if the column and both diagonals incident to it are unoccupied. It is \termdefine{column-safe} if the row and both diagonals incident to it are unoccupied in and it is \termdefine{plus (minus)-safe} if the row, column, and minus (plus)-diagonal incident to it are unoccupied.
\end{definition}

Observe that a position is available if and only if it is row-safe and its row is unoccupied. Analogous statements hold for column, plus, and minus-safe positions.

We will track the number of safe positions located in small strips of the board.

Let $\alpha \in I_M$ and let $(x,y) \in [n]^2$. Let $\mR_{y,\alpha}(t)$ be the set of row safe positions in $Q(t)$ that are in row $y$ and in $\alpha_n$. Let $\mC_{x,\alpha}(t)$ be the set of column-safe positions in $Q(t)$ in column $x$ and in $\alpha_n$. Let $\mD^+_{x+y,\alpha}(t)$ ($\mD^-_{y-x,\alpha}(t)$) be the set of plus- (minus-)safe positions in $Q(t)$ in plus- (minus-)diagonal $x+y$ ($y-x$) and in $\alpha_n$. Finally, let $\mA_\alpha(t) \coloneqq \alpha_n \cap \mA_{Q(t)}$ be the set of available positions in $\alpha_n$ at time $t$.

For each of these (random) sets, which are denoted using stylized Latin letters, we use the capital Latin letter equivalent for its cardinality. For example, $R_{y,\alpha}(t) = |\mR_{y,\alpha}(t)|$. In order to streamline the analysis it is useful to define the normalized random variable $\tA_\alpha(t) \coloneqq A_\alpha(t) / ((n-t)/M)$.

We now define the expected trajectories of the random variables. Let $(x,y) \in [n]^2$ and $\alpha \in I_M$. Recall the definitions of $L_{y,\alpha}^r,L_{x,\alpha}^c,L_{x+y,\alpha}^+,$ and $L_{x-y,\alpha}^-$ from \cref{sec:partitions}. For $t \in [0,T]$ define:
\begin{align*}
&r_{y,\alpha}(t) = L^r_{y,\alpha} \left( 1 - \frac{t}{n} \right) \left( 1 - M\delta^+(\alpha) \frac{t}{n} \right) \left( 1 - M\delta^-(\alpha) \frac{t}{n} \right),\\
&c_{x,\alpha}(t) = L^c_{x,\alpha} \left( 1 - \frac{t}{n} \right) \left( 1 - M\delta^+(\alpha) \frac{t}{n} \right) \left( 1 - M\delta^-(\alpha) \frac{t}{n} \right),\\
&d_{x+y,\alpha}^+(t) = L_{x+y,\alpha}^+ \left( 1 - \frac{t}{n} \right)^2 \left( 1 - M\delta^-(\alpha) \frac{t}{n} \right),\\
&d_{x-y,\alpha}^-(t) = L_{x-y,\alpha}^- \left( 1 - \frac{t}{n} \right)^2 \left( 1 - M\delta^+(\alpha) \frac{t}{n} \right),\\
&\ta_\alpha(t) = \frac{|\alpha_n|}{(n-t)/M} \left( 1 - \frac{t}{n} \right)^2 \left( 1 - M\delta^+(\alpha) \frac{t}{n} \right) \left( 1 - M\delta^-(\alpha) \frac{t}{n} \right).
\end{align*}
We also define the error function:
\[
E(t) = \frac{n}{M^{5/4}(1-t/n)^{K}}
\]
(where $K$ is the (large) constant used to define $T$).

We will use a differential equation method \cite{wormald1999differential} style martingale analysis to show that w.h.p.\ the random variables closely follow their trajectories. Informally, the method states that if a sequence of random variables $F(0),F(1),\ldots,F(T)$ and a smooth function $f:[0,T] \to \R$ satisfy:
\begin{itemize}
	\item \textit{Initial condition}: $F(0) \approx f(0)$;
	
	\item \textit{Trajectory condition}: For every $t<T$, $\E [F(t+1) - F(t) \given F(t)] \approx f'(t)$; and
	
	\item \textit{Boundedness condition}: There exists a constant $C$ such that $\norm{f'}_\infty \leq C$ and $|F(t+1)-F(t)| \leq C$;
\end{itemize}
then w.h.p.\ $F(t) \approx f(t)$.

In general, it may not be the case that the expected one-step changes in the random variables we have defined are close to the derivatives of their respective trajectories. However, we will show that this is the case for as long as they remain close to their trajectories. This motivates the next definition.

\begin{definition}\label{def:tau}
	Let the stopping time $\tau > 0$ be the smallest $t$ such that one of the random variables deviates by more than $E(t)$ from its expected trajectory. That is, $\tau$ is the smallest $t$ such that there exists some $(x,y) \in [n]^2$ and $\alpha \in I_M$ such that at least one of
	\begin{align*}
	& |R_{y,\alpha}(t) - r_{y,\alpha}(t)|, |C_{x,\alpha}(t) - c_{x,\alpha}(t)|, |D^+_{x+y,\alpha}(t) - d^+_{x+y,\alpha}(t)|,\\
	& |D^-_{x-y,\alpha}(t) - d^-_{x-y}(t)|, |\tA_\alpha(t) - \ta_\alpha(t)|
	\end{align*}
	is larger than $E(t)$. If there is no such $t$ set $\tau = \infty$.
\end{definition}

Most of this section is devoted to proving the next proposition, which implies that w.h.p.\ Algorithm \ref{alg:random} does not abort.

\begin{prop}\label{prop:tau inf whp}
	It holds that $\Prob \left[ \tau < \infty \right] = \exp \left( - \Omega \left( n^{0.75} \right) \right)$.
\end{prop}

We begin by recording an estimate on $A_\alpha(t)$ that holds under the assumption $\tau>t$.

\begin{obs}\label{obs:A_alpha asymptotic}
	Suppose that $\tau > t$. Then, for every $\alpha \in I_M$ we have
	\[
	A_\alpha(t) = \left( 1 \pm \frac{10 M E(t)}{(1-\eta)^2 (n-t) } \right) |\alpha_n| \left( 1 - \frac{t}{n} \right)^2 \left( 1 - M\delta^+(\alpha) \frac{t}{n} \right) \left( 1 - M\delta^-(\alpha) \frac{t}{n} \right)
	\]
	and
	\[
	A_\alpha(t) \geq \frac{(1-\eta)^2 (n-t)^2}{5M^2}.
	\]
\end{obs}

\begin{proof}
	The proof follows straightforwardly from unpacking the definitions. Indeed, since $\tau > t$, we have
	\[
	A_\alpha(t) = \tA_\alpha(t) \frac{n-t}{M} = \ta_\alpha(t)\frac{n-t}{M} \pm E(t)\frac{n-t}{M} = \left( 1 \pm \frac{E(t)}{\ta_\alpha(t)} \right) \ta_\alpha(t) \frac{n-t}{M}.
	\]
	Using the facts that $|\alpha_n| \geq n^2\!/(5M^2)$ and that $M\delta^+(\alpha), M\delta^-(\alpha) \leq \eta$ we conclude that for every $t \leq n$ there holds $\ta_\alpha(t) \geq (1-\eta)^2 (n-t) / (10M)$. Together with the definition of $\ta_\alpha(t)$ this implies the equality in the observation.
	
	The lower bound on $A_\alpha(t)$ follows from the equality together with the fact that $|\alpha_n| \geq n^2 / (4.5M^2)$ and that $10ME(t) / ((1-\eta)^2(n-t)) = o(1)$.
\end{proof}

The next claim shows that for as long as $\tau > t$, every unoccupied row and column is approximately equally likely to be occupied at step $t+1$.

\begin{clm}\label{clm:row-col prob}
	Suppose that $\tau > t$. Then, for every unoccupied row or column in $Q(t)$, the probability that it is occupied in $Q(t+1)$ is
	\[
	\frac{1}{n-t} \pm O \left( \frac{ME(t)}{(n-t)^2}\right).
	\]
\end{clm}

\begin{proof}
	By symmetry it suffices to prove only the statement for unoccupied rows.
	
	Let $y \in [n]$. For $0 \leq t \leq T$ let $B(t)$ be the event that row $y$ is occupied in $Q(t)$. Given $Q(t)$ such that row $y$ is unoccupied (i.e., $B(t)$ does not hold), we have
	\[
	\Prob[B(t+1) \given Q(t)] = \sum_{x=1}^n \Prob \left[ X_t = (x,y) \given Q(t) \right] = \sum_{\alpha \in I_M} \frac{\delta(\alpha) R_{y,\alpha}(t)}{A_\alpha(t)}.
	\]
	If $\tau > t$ then for every $\alpha$ we have $R_{y,\alpha}(t) = r_{y,\alpha}(t) \pm E(t)$ and (by \cref{obs:A_alpha asymptotic}) for $C = 10 / (1-\eta)^2$ and $\alpha \in I_M$:
	\[
	A_{\alpha}(t) = \left( 1 \pm \frac{C M E(t)}{n-t} \right) |\alpha_n| \left( 1 - \frac{t}{n} \right)^2 \left( 1 - M\delta^+(\alpha) \frac{t}{n} \right) \left( 1 - M \delta^-(\alpha) \frac{t}{n} \right).
	\]
	Additionally, for $D = (1-\eta)^2/5$ and every $\alpha \in I_M$, there holds $A_\alpha(t) \geq D (n-t)^2\!/M^2$.	Therefore:
	\begin{align*}
	\Prob & \left[ B(t+1) \given B^c(t) \land \tau > t \right]\\
	= & \left( 1 \pm \frac{C M E(t)}{n-t} \right) \sum_{\alpha \in I_M : L_{y,\alpha}^r>0} \frac{\delta(\alpha) \left( L^r_{y,\alpha} \left( 1 - \frac{t}{n} \right) \left( 1 - M\delta^+(\alpha) \frac{t}{n} \right) \left( 1 - M\delta^-(\alpha) \frac{t}{n} \right) \pm E(t) \right)}{|\alpha_n| \left( 1 - t/n \right)^2 \left( 1 - M\delta^+(\alpha) \frac{t}{n} \right) \left( 1 - M \delta^-(\alpha) \frac{t}{n} \right)}\\
	= & \left( 1 \pm \frac{C M E(t)}{n-t} \right) \sum_{\alpha \in I_M} \frac{\delta(\alpha) L^r_{y,\alpha} }{|\alpha_n| \left( 1 - t/n \right)} \pm \frac{ M^2 E(t)}{D (n-t)^2} \sum_{\alpha \in I_M : L_{y,\alpha}^r > 0} \delta(\alpha).
	\end{align*}
	Observe that the set of $\alpha \in I_M$ such that $L_{y,\alpha} > 0$ all intersect row $y$. Therefore these sets are all contained in an axis parallel rectangle of height $\leq 4/M$. Since $\delta$ has uniform marginals, we have $\sum_{\alpha \in I_M : L_{y,\alpha} > 0} \delta(\alpha) \leq 4/M$. Therefore:
	\[
	\frac{ M^2 E(t)}{D (n-t)^2} \sum_{\alpha \in I_M : L_{y,\alpha}^r > 0} \delta(\alpha) \leq \frac{ 4M E(t)}{D (n-t)^2}.
	\]
	We turn our attention to the first sum. By Claim \ref{clm:line sums} \ref{itm:row sum}:
	\[
	\sum_{\alpha \in I_M} \frac{\delta(\alpha)L_{y,\alpha}^r}{|\alpha_n|} = \frac{1}{n} \pm O \left( \frac{1}{nM} \right).
	\]
	Hence
	\begin{align*}
	\left( 1 \pm \frac{C M E(t)}{n-t} \right) & \sum_{\alpha \in I_M} \frac{\delta(\alpha) L^r_{y,\alpha} }{|\alpha_n| \left( 1 - t/n \right)}
	= \left( 1 \pm \frac{C M E(t)}{n-t} \right) \left( 1 \pm O \left( \frac{1}{M} \right) \right) \frac{1}{n-t}\\
	& = \left( 1 \pm \frac{2C M E(t)}{n-t} \right)  \frac{1}{n-t} = \frac{1}{n-t} \pm \frac{2CME(t)}{(n-t)^2}.
	\end{align*}
	Therefore:
	\begin{align*}
	\Prob \left[ B(t+1) \given B^c(t) \land \tau > t \right] = \frac{1}{n-t} \pm \left( \frac{2CME(t)}{(n-t)^2} + \frac{4ME(t)}{D(n-t)^2} \right) = \frac{1}{n-t} \pm O \left( \frac{ME(t)}{(n-t)^2} \right),
	\end{align*}
	as desired.
\end{proof}

\begin{clm}\label{clm:diagonal occupation probability}
	Let $* \in \{+,-\}$. Suppose that $\tau > t$ and that $*$-diagonal $c$, which intersects $\alpha_n$, is unoccupied in $Q(t)$. Then, the probability that $*$-diagonal $c$ is occupied in $Q(t+1)$ is
	\[
	\frac{M \delta^*(\alpha)}{n \left( 1 - M\delta^*(\alpha) \frac{t}{n} \right)} \pm O \left( \frac{ME(t)}{(n-t)^2} \right).
	\]
\end{clm}

The term $M\delta^*(\alpha) / (n - M\delta^*(\alpha) t)$ can be interpreted as follows: In a configuration $q \in B_n(\delta,\varepsilon)$, approximately $n\delta^*(\alpha)$ of the $*$-diagonals passing through $\alpha_n$ are occupied. Thus, after $t$ steps of the process, the probability that the next queen placed should occupy a particular one of these diagonals is proportional to $\delta^*(\alpha)$ (i.e., the fraction of the occupied $*$-diagonals that pass through $\alpha_n$) and also $1/(n/M-\delta^*(\alpha)t)$ (i.e., the inverse of the number of remaining unoccupied diagonals that pass through $\alpha_n$).

\begin{proof}
	The proof is similar to that of Claim \ref{clm:row-col prob}. We prove only the case $*=+$. For $0 \leq t \leq T$ let $B(t)$ be the event that plus-diagonal $c$ is occupied in $Q(t)$. Given $Q(t)$ such that plus-diagonal $c$ is unoccupied, we have
	\[
	\Prob \left[ B(t+1) \given Q(t) \right] = \sum_{\beta \in I_M} \frac{\delta(\beta) D^+_{c,\beta}(t)}{A_\beta(t)}.
	\]
	If $\tau > t$ we have, for $C = 10/(1-\eta)^2$:
	\begin{align*}
	\Prob & \left[ B(t+1) \given B^c(t) \land \tau > t \right]\\
	= & \left( 1 \pm \frac{CM E(t)}{n-t} \right) \sum_{\beta \in I_M} \frac{\delta(\beta) \left( L^+_{c,\beta} \left( 1 - \frac{t}{n} \right)^2 \left( 1 - M\delta^-(\beta) \frac{t}{n} \right) \pm E(t) \right)}{|\beta_n| \left( 1 - t/n \right)^2 \left( 1 - M\delta^+(\beta) \frac{t}{n} \right) \left( 1 - M \delta^-(\beta) \frac{t}{n} \right)}\\
	= & \left( 1 \pm \frac{CM E(t)}{n-t} \right) \sum_{\beta \in I_M} \frac{\delta(\beta) L^+_{c,\beta}}{|\beta_n| \left( 1 - M\delta^+(\beta) \frac{t}{n} \right)} \pm O \left( \frac{M E(t)}{(n-t)^2} \right).
	\end{align*}
	We note that for every $\beta \in I_M$ such that $L_{c,\beta}^+>0$ it holds that $\delta^+(\beta) = \delta^+(\alpha)$. Therefore, by Claim \ref{clm:line sums} \ref{itm:plus diagonal sum}:
	\begin{align*}
	\sum_{\beta \in I_M} \frac{\delta(\beta) L^+_{c,\beta}}{|\beta_n| \left( 1 - M\delta^+(\beta) \frac{t}{n} \right)} = \frac{M \delta^+(\alpha)}{n - M \delta^+(\alpha)t } \pm O \left( \frac{1}{Mn} \right).
	\end{align*}
	Hence:
	\begin{align*}
	\Prob \left[ B(t+1) \given B^c(t) \land \tau > t \right]
	& = \left( 1 \pm \frac{CM E(t)}{n-t} \right) \frac{M \delta^+(\alpha)}{n - M \delta^+(\alpha) t} \pm O \left( \frac{ME(t)}{(n-t)^2} \right)\\
	& = \frac{M \delta^+(\alpha)}{n - M \delta^+(\alpha) t } \pm O \left( \frac{ME(t)}{(n-t)^2} \right),
	\end{align*}
	as desired.
\end{proof}

We now transform the random variables so that we can apply a martingale analysis. Let $X(t)$ be one of the random variables in
\[
\{R_{y,\alpha}(t), C_{x,\alpha}(t), D^+_{x+y,\alpha}(t), D^-_{x-y,\alpha}(t), \tA_{\alpha}(t) : \alpha \in I_M, (x,y) \in [n]^2, 0 \leq t \leq T\}.
\]
We write the corresponding trajectory function as $x(t)$ (for example, if $X(t) = R_{y,\alpha}(t)$ then $x(t) = r_{y,\alpha}(t)$). Define the following random variables:
\begin{align*}
X^+(t) = \begin{cases}
X(t) - x(t) - \frac{1}{2}E(t) & t \leq \tau\\
X^+(t-1) & t > \tau.
\end{cases},\\
X^-(t) = \begin{cases}
x(t) - X(t) - \frac{1}{2}E(t) & t \leq \tau\\
X^-(t-1) & t > \tau
\end{cases}.
\end{align*}
We will show that these sequences are supermartingales with respect to the filtration induced by $Q(0),Q(1),\ldots,Q(T)$. We will then apply the Azuma--Hoeffding inequality to show that they closely follow their trajectories. This will imply Proposition \ref{prop:tau inf whp}.

As shown in Claims \ref{clm:row-col prob} and \ref{clm:diagonal occupation probability}, conditioning on $\tau > t$ implies that a certain regularity holds at time $t$. This makes it easy to calculate expected one-step changes. This motivates freezing the random variables at the stopping time $\tau$.

We will use the following version of the Azuma--Hoeffding inequality.

\begin{thm}[{\cite[Lemma 1]{wormald1995differential}}]\label{thm:azuma}
	Let $X_0,X_1,\ldots$ be a supermartingale with respect to a filtration $\mF_0,\mF_1,\ldots$. Let $C>0$ satisfy $C \geq |X_i - X_{i-1}|$ for every $i$. Then, for every $\lambda>0$ and $t \geq 0$, it holds that
	\[
	\Prob \left[ X_t \geq X_0 + \lambda \right] \leq \exp \left( - \frac{\lambda^2}{2 t C^2} \right).
	\]
\end{thm}

The next lemma establishes the boundedness condition required by Theorem \ref{thm:azuma}.

\begin{lem}\label{lem:one step bound}
	Let $\{X(t)\}_{t=0}^T$ be one of the sequences $\{ R_{y,\alpha}(t) \}$, $\{C_{x,\alpha}(t) \}$, $\{D^+_{x+y,\alpha}(t) \}$, $\{D^-_{x-y,\alpha}(t)\}$, $\{\tA_\alpha(t)\}$, for $(x,y) \in [n]^2$ and $\alpha \in I_M$. Then, for every $0 \leq t < T$:
	\[
	|X^+(t+1) - X^+(t)|, |X^-(t+1) - X^-(t)| = O(1/(1-t/n)).
	\]
\end{lem}

\begin{proof}
	Let $(x,y) \in [n]^2$ and $\alpha \in I_M$.
	
	We first note that the derivatives of the functions $r_{y,\alpha}$, $c_{x,\alpha}$, $d^+_{x+y,\alpha}$, $d^-_{y-x,\alpha}$, $\ta_\alpha$, and $E$ are bounded in absolute value by $1$.
	
	Next, we observe that whenever a queen is added to a partial $n$-queens configuration, exactly one row, one column, one plus-diagonal, and one minus-diagonal are occupied. Thus, in every time step, each of $\{ R_{y,\alpha} \}$, $\{C_{x,\alpha} \}$, $\{D^+_{x+y,\alpha} \}$, and $\{D^-_{x-y,\alpha} \}$ changes by at most $3$. Similarly, $A_\alpha$ changes by at most $7n/M$, so $\tA_{\alpha}$ changes by at most $7/(1-t/n)$.
	
	Together, these observations imply that $X^+$ and $X^-$ can change by at most $9/(1-t/n) = O(1/(1-t/n))$ at time $t$.
\end{proof}

The next step is to show that the random variables are supermartingales.

\begin{lem}\label{lem:supermartingale condition}
	Let $\{X(t)\}_{t=0}^T$ be one of the sequences $\{ R_{y,\alpha}(t) \}$, $\{C_{x,\alpha}(t) \}$, $\{D^+_{x+y,\alpha}(t) \}$, $\{D^-_{x-y,\alpha}(t)\}$, $\{\tA_\alpha(t)\}$, for $(x,y) \in [n]^2$ and $\alpha \in I_M$. Then, for every $0 \leq t < T$:
	\[
	\E \left[ X^+(t+1) - X^+(t) \given Q(1), \ldots, Q(t) \right] \leq 0
	\]
	and
	\[
	\E \left[ X^-(t+1) - X^-(t) \given Q(1), \ldots, Q(t) \right] \leq 0.
	\]
\end{lem}

Before proving Lemma \ref{lem:supermartingale condition} we calculate the expected one step changes of our random variables.

\begin{clm}
	Let $(x,y) \in [n]^2$, $\alpha \in I_M$, and $0 \leq t < T$. The following hold for every $Q \subseteq [n]^2$ such that $\Prob[\tau>t \cap Q(t) = Q] > 0$:
	
	\begin{enumerate}
		\item\label{itm:R expected change}
		$\E \left[ R_{y,\alpha}(t+1) - R_{y,\alpha}(t) \given Q(t) = Q, \tau > t \right] = r'_{y,\alpha}(t) \pm O \left( \frac{E(t)}{n-t} \right)$.
		
		\item\label{itm:C expected change}
		$\E \left[ C_{x,\alpha}(t+1) - C_{x,\alpha}(t) \given Q(t) = Q, \tau > t \right] = c'_{x,\alpha}(t) \pm O \left( \frac{E(t)}{n-t} \right)$.
		
		\item\label{itm:D^+ expected change}
		$\E \left[ D_{x+y,\alpha}^+(t+1) - D_{x+y,\alpha}^+(t) \given Q(t) = Q, \tau > t \right] = {d^+_{x+y,\alpha}}'(t) \pm O \left( \frac{E(t)}{n-t} \right)$.
		
		\item\label{itm:D^- expected change}
		$\E \left[ D_{x+y,\alpha}^-(t+1) - D_{x+y,\alpha}^-(t) \given Q(t) = Q, \tau > t \right] = {d_{x-y,\alpha}^-}'(t) \pm O \left( \frac{E(t)}{n-t} \right)$.
		
		\item\label{itm:Z expected change}
		$\E \left[ \tA_\alpha(t+1) - \tA_\alpha(t) \given Q(t) = Q, \tau > t \right] = \ta_{\alpha}'(t) \pm O \left( \frac{E(t)}{n-t} \right)$.
	\end{enumerate}
\end{clm}

\begin{proof}
	All five assertion follow from Claims \ref{clm:row-col prob} and \ref{clm:diagonal occupation probability}. We first prove \ref{itm:R expected change}. For notational conciseness we write $\E'$ and $\Prob'$, respectively, for expectations and probabilities conditioned on $Q(t)=Q \cap \tau > t$.
	
	Conditioning on a particular $Q(t)$ we have, by definition:
	\[
	\E \left[ R_{y,\alpha}(t+1) - R_{y,\alpha}(t) \given Q(t) \right] = - \sum_{(c,r) \in \mR_{y,\alpha}(t)} \Prob \left[ (c,r) \notin \mR_{y,\alpha}(t+1) \given Q(t) \right].
	\]
	Let $(c,r) \in \mR_{y,\alpha}(t)$. By definition, $(c,r) \notin \mR_{y,\alpha}(t+1)$ if and only if column $c$, plus-diagonal $r+c$, or minus-diagonal $r-c$ is occupied at time $t+1$. By Claims \ref{clm:row-col prob} and \ref{clm:diagonal occupation probability} if $\tau > t$ then the respective probabilities of these events are
	\[
	\frac{1}{n-t} \pm O \left( \frac{ME(t)}{(n-t)^2} \right), \quad \frac{M \delta^+(\alpha)}{n \left( 1 - M\delta^+(\alpha) \frac{t}{n} \right)} \pm O \left( \frac{ME(t)}{(n-t)^2} \right),
	\]
	and
	\[
	\frac{M \delta^-(\alpha)}{n \left( 1 - M\delta^-(\alpha) \frac{t}{n} \right)} \pm O \left( \frac{ME(t)}{(n-t)^2} \right).
	\]
	Additionally, more than one of these events occurs if and only if $X_{t+1} = (c,r)$. By \cref{obs:A_alpha asymptotic} $\Prob' [X_{t+1} = (c,r)] = O \left( \delta(\alpha)M^2/(n-t)^2 \right) = O \left( 1/(n-t)^2 \right)$. Therefore:
	\begin{align*}
	\Prob' \left[ (c,r) \notin \mR_{y,\alpha}(t+1) \right] = \frac{1}{n-t} + \frac{M \delta^+(\alpha)}{n \left( 1 - M\delta^+(\alpha) \frac{t}{n} \right)} + \frac{M \delta^-(\alpha)}{n \left( 1 - M\delta^-(\alpha) \frac{t}{n} \right)} \pm O \left( \frac{ME(t)}{(n-t)^2} \right).
	\end{align*}
	Hence:
	\begin{align*}
	\E' & \left[ R_{y,\alpha}(t+1) - R_{y,\alpha}(t) \right]\\
	& = - R_{y,\alpha}(t) \left( \frac{1}{n-t} + \frac{M \delta^+(\alpha)}{n \left( 1 - M\delta^+(\alpha) \frac{t}{n} \right)} + \frac{M \delta^-(\alpha)}{n \left( 1 - M\delta^-(\alpha) \frac{t}{n} \right)} \right) \pm O \left( \frac{R_{y,\alpha}(t) ME(t)}{(n-t)^2} \right).
	\end{align*}
	We note that if $\tau > t$ then $R_{y,\alpha}(t) \leq 2(n-t)/M$. Thus:
	\begin{align*}
	\E' & \left[ R_{y,\alpha}(t+1) - R_{y,\alpha}(t)\right]\\
	& = - R_{y,\alpha}(t) \left( \frac{1}{n-t} + \frac{M \delta^+(\alpha)}{n \left( 1 - M\delta^+(\alpha) \frac{t}{n} \right)} + \frac{M \delta^-(\alpha)}{n \left( 1 - M\delta^-(\alpha) \frac{t}{n} \right)} \right) \pm O \left( \frac{E(t)}{n-t} \right).
	\end{align*}
	Conditioning on $\tau > t$ implies $R_{y,\alpha}(t) = r_{y,\alpha}(t) \pm E(t)$. Therefore:
	\begin{align*}
	R_{y,\alpha}(t) & \left( \frac{1}{n-t} + \frac{M \delta^+(\alpha)}{n \left( 1 - M\delta^+(\alpha) \frac{t}{n} \right)} + \frac{M \delta^-(\alpha)}{n \left( 1 - M\delta^-(\alpha) \frac{t}{n} \right)} \right)\\
	& = \left(r_{y,\alpha}(t) \pm E(t) \right) \left( \frac{1}{n-t} + \frac{M \delta^+(\alpha)}{n \left( 1 - M\delta^+(\alpha) \frac{t}{n} \right)} + \frac{M \delta^-(\alpha)}{n \left( 1 - M\delta^-(\alpha) \frac{t}{n} \right)} \right).
	\end{align*}
	Because $\delta$ has sub-uniform diagonal marginals, $M \delta^+(\alpha), M \delta^-(\alpha) \leq 1$. Thus:
	\begin{align*}
	R_{y,\alpha}(t) & \left( \frac{1}{n-t} + \frac{M \delta^+(\alpha)}{n \left( 1 - M\delta^+(\alpha) \frac{t}{n} \right)} + \frac{M \delta^-(\alpha)}{n \left( 1 - M\delta^-(\alpha) \frac{t}{n} \right)} \right)\\
	& = r_{y,\alpha}(t) \left( \frac{1}{n-t} + \frac{M \delta^+(\alpha)}{n \left( 1 - M\delta^+(\alpha) \frac{t}{n} \right)} + \frac{M \delta^-(\alpha)}{n \left( 1 - M\delta^-(\alpha) \frac{t}{n} \right)} \right) \pm \frac{3 E(t)}{n-t}.
	\end{align*}
	Finally, we observe that
	\[
	r'_{y,\alpha}(t) = - r_{y,\alpha}(t) \left( \frac{1}{n-t} + \frac{M \delta^+(\alpha)}{n \left( 1 - M\delta^+(\alpha) \frac{t}{n} \right)} + \frac{M \delta^-(\alpha)}{n \left( 1 - M\delta^-(\alpha) \frac{t}{n} \right)} \right).
	\]
	Therefore:
	\begin{align*}
	\E' \left[ R_{y,\alpha}(t+1) - R_{y,\alpha}(t) \right] = r'_{y,\alpha}(t) \pm O \left( \frac{E(t)}{n-t} \right),
	\end{align*}
	proving \ref{itm:R expected change}. The proof of \ref{itm:C expected change} follows similarly.
	
	We prove \ref{itm:D^+ expected change} with a similar argument. By definition:
	\[
	\E \left[ D^+_{x+y,\alpha}(t+1) - D^-_{x-y,\alpha}(t) \given Q(t) \right] = - \sum_{(c,r) \in \mD^+_{x+y,\alpha}(t)} \Prob \left[ (c,r) \notin \mD^+_{x+y,\alpha}(t+1) \given Q(t) \right].
	\]
	For every $(c,r) \in \mD^+_{x+y,\alpha}(t)$, the event $(c,r) \notin \mD^+_{x+y,\alpha}(t+1)$ occurs if and only if $X_{t+1}$ is in column $c$, row $r$, or minus-diagonal $r-c$. By Claims \ref{clm:row-col prob} and \ref{clm:diagonal occupation probability} if $\tau > t$ then the probability of this occurrence is
	\[
	\frac{2}{n-t} + \frac{M \delta^-(\alpha)}{n \left( 1 - M\delta^-(\alpha) \frac{t}{n} \right)} \pm O \left( \frac{ME(t)}{(n-t)^2} \right).
	\]
	Additionally, if $\tau > t$ then $D^+_{x+y,\alpha}(t) = d^+_{x+y,\alpha}(t) \pm E(t)$. Therefore:
	\begin{align*}
	\E' & \left[ D^+_{x+y,\alpha}(t+1) - D^-_{x-y,\alpha}(t) \right]\\
	& = - \left( d^+_{x+y,\alpha}(t) \pm E(t) \right) \left( \frac{2}{n-t} + \frac{M \delta^-(\alpha)}{n \left( 1 - M\delta^-(\alpha) \frac{t}{n} \right)} \pm O \left( \frac{ME(t)}{(n-t)^2} \right) \right)\\
	& = {d^+_{x+y,\alpha}}'(t) \pm O \left( \frac{E(t)}{n-t} \right),
	\end{align*}
	proving \ref{itm:D^+ expected change}.  A proof of \ref{itm:D^- expected change} is obtained by interchanging the roles of plus- and minus-diagonals.
	
	Finally, we prove \ref{itm:Z expected change}. By definition:
	\begin{align*}
	\E' & \left[ \tA_\alpha(t+1) - \tA_\alpha(t) \right] = \E' \left[ \frac{M}{n-t-1}A_\alpha(t+1) - \frac{M}{n-t}A_\alpha(t) \right]\\
	& = \E' \left[ \left(\frac{M}{n-t} + \frac{M}{(n-t)^2} \pm O \left( \frac{M}{(n-t)^3} \right)\right)A_\alpha(t+1) - \frac{M}{n-t}A_\alpha(t) \right]\\
	& = \E' \left[ \frac{M}{n-t} \left( A_\alpha(t+1) - A_\alpha(t) \right) + \frac{M}{(n-t)^2} A_\alpha(t+1) \right] \pm O \left( \frac{M(n/M)^2}{(n-t)^2} \right)\\
	& \stackrel{\text{\cref{lem:one step bound}}}{=} \frac{M}{n-t} \E' \left[ A_\alpha(t+1) - A_\alpha(t) \right] + \frac{M}{(n-t)^2} \left( A_\alpha(t) \pm O \left( \frac{n}{M} \right) \right) \pm O \left( \frac{1}{M(1-t/n)^2} \right)\\
	& \stackrel{t<\tau}{=} \frac{M}{n-t} \E' \left[ A_\alpha(t+1) - A_\alpha(t) \right] + \frac{\ta_\alpha(t)}{n-t} \pm O \left( \frac{E(t)}{n-t} \right).
	\end{align*}
	
	We now estimate the expected change to $A_\alpha(t)$. By definition:
	\[
	\E' \left[ A_\alpha(t+1) - A_\alpha(t) \right] = - \sum_{(x,y) \in \mA_{\alpha}(t)} \Prob' \left[ (x,y) \notin \mA_{\alpha}(t+1) \right].
	\]
	For every $(x,y) \in \mA_\alpha(t)$ the event $(x,y) \notin \mA_\alpha(t+1)$ occurs if and only if $X_{t+1}$ is in column $x$, row $y$, plus-diagonal $x+y$, or minus-diagonal $y-x$. By \cref{clm:row-col prob,clm:diagonal occupation probability} if $\tau > t$ then the probability of this event is
	\[
	\frac{2}{n-t} + \frac{M \delta^-(\alpha)}{n \left( 1 - M\delta^-(\alpha) \frac{t}{n} \right)} + \frac{M \delta^+(\alpha)}{n \left( 1 - M\delta^+(\alpha) \frac{t}{n} \right)} \pm O \left( \frac{ME(t)}{(n-t)^2} \right).
	\]
	Additionally, if $\tau > t$ then by \cref{obs:A_alpha asymptotic} $A_\alpha(t) = \left( 1 \pm O (ME(t) / (n-t)) \right) (n-t) \ta(t)/M$. Hence
	\begin{align*}
	\E' & \left[ A_\alpha(t+1) - A_\alpha(t) \right]\\
	& = - \frac{(n-t)\ta(t)}{M} \left( \frac{2}{n-t} + \frac{M \delta^-(\alpha)}{n \left( 1 - M\delta^-(\alpha) \frac{t}{n} \right)} + \frac{M \delta^+(\alpha)}{n \left( 1 - M\delta^+(\alpha) \frac{t}{n} \right)} \right) \pm O \left( \frac{E(t) \ta(t)}{n-t} \right).
	\end{align*}
	
	This implies that
	\begin{align*}
	\E' & \left[ \tA_\alpha(t+1) - \tA_\alpha(t) \right]\\
	& = - \ta(t) \left( \frac{2}{n-t} + \frac{M \delta^-(\alpha)}{n \left( 1 - M\delta^-(\alpha) \frac{t}{n} \right)} + \frac{M \delta^+(\alpha)}{n \left( 1 - M\delta^+(\alpha) \frac{t}{n} \right)} \right) + \frac{\ta_\alpha(t)}{n-t} \pm O \left( \frac{E(t)}{n-t} \right)\\
	& = \ta'(t) \pm O \left( \frac{E(t)}{n-t} \right),
	\end{align*}
	as desired.
\end{proof}

Next, we estimate the one-step changes of the trajectory functions.

\begin{clm}
	The following hold for every $(x,y) \in [n]^2$, $\alpha \in I_M$, and $0 \leq t < T$:
	\begin{enumerate}
		\item\label{itm:r expected change}
		$r_{y,\alpha}(t+1) - r_{y,\alpha}(t) = r'_{y,\alpha}(t) \pm \frac{E(t)}{n-t}$.
		
		\item\label{itm:c expected change}
		$c_{x,\alpha}(t+1) - c_{x,\alpha}(t) = c'_{x,\alpha}(t) \pm \frac{E(t)}{n-t}$.
		
		\item\label{itm:d^+ expected change}
		$d_{x+y,\alpha}^+(t+1) - d_{x+y,\alpha}^+(t) = {d^+_{x+y,\alpha}}'(t) \pm \frac{E(t)}{n-t}$.
		
		\item\label{itm:d^- expected change}
		$d_{x-y,\alpha}(t+1) - d_{x-y,\alpha}^-(t) = {d_{x-y,\alpha}^-}'(t) \pm \frac{E(t)}{n-t}$.
		
		\item\label{itm:z expected change}
		$\ta_\alpha(t+1) - \ta_\alpha(t) = \ta_{\alpha}'(t) \pm \frac{E(t)}{n-t}$.
		
		\item\label{itm:epsilon expected change}
		$E(t+1) - E(t) = \frac{K E(t)}{n-t} \pm \frac{E(t)}{n-t}$.
	\end{enumerate}
\end{clm}

\begin{proof}
	Each of assertions follows from Taylor's theorem. For every $x \in [0,T]$:
	\[
	E'(x) = \frac{K E(x)}{n(1-x/n)}, \quad E''(x) = \frac{K(K+1) E(x)}{n^2(1-x/n)^2} \leq \frac{ 2E(x)}{n(1-x/n)}.
	\]
	By Taylor's theorem for every $0 \leq t \leq T-1$ there exists some $\zeta \in (t,t+1)$ such that
	\[
	E(t+1) - E(t) = E'(t) + \frac{1}{2}E''(\zeta) = \frac{K E(t)}{n(1-t/n)} \pm \frac{ E(t)}{n-t},
	\]
	as desired.
	
	For the remaining assertions it suffices to show that if $f$ is one of the functions $r_{y,\alpha}$, $c_{x,\alpha}$, $d^+_{x+y,\alpha}$, $d^-_{x-y,\alpha}$, or $\ta_\alpha$ then for every $0 \leq t \leq T-1$ and every $\zeta \in (t,t+1)$ it holds that
	\begin{equation}\label{eq:second derivative bound}
	|f''(\zeta)| \leq \frac{2E(t)}{n-t}.
	\end{equation}
	This follows from direct computation. We will demonstrate this for $r_{y,\alpha}$. The other calculations are similar.
	
	For every $\zeta \in [0,T]$ it holds that
	\begin{align*}
	r_{y,\alpha}''(\zeta) = \frac{2 M L_{y,\alpha}^r\delta^-(\alpha)}{n^2} \left( 1 - M \delta^+(\alpha) \frac{\zeta}{n} \right) & + \frac{2 M L_{y,\alpha}^r\delta^+(\alpha)}{n^2} \left( 1 - M \delta^-(\alpha) \frac{\zeta}{n} \right)\\
	& + \frac{2 M^2 L_{y,\alpha}^r\delta^+(\alpha) \delta^-(\alpha)}{n^2} \left( 1 - \frac{\zeta}{n} \right).
	\end{align*}
	Since $L_{y,\alpha}^r \leq 2n/M$ and $\delta^+(\alpha),\delta^-(\alpha) \leq 1/M$, for every $0 \leq t < T$:
	\[
	|r_{y,\alpha}''(\zeta)| \leq \frac{12}{nM} \leq \frac{E(t)}{n(1-t/n)},
	\]
	and \eqref{eq:second derivative bound} holds.
\end{proof}

We are ready to show that the random variables are supermartingales.

\begin{proof}[Proof of Lemma \ref{lem:supermartingale condition}]
	Let $X$ be one of the random variables and let $* \in \{+,-\}$. Let $\sigma = 1$ if $*=+$ and $\sigma=-1$ if $*=-$. Condition on $Q(1),\ldots,Q(t)$. If $\tau \leq t$ then, by definition, $X^*(t+1) - X^*(t) = 0$. On the other hand, by the previous two claims:
	\begin{align*}
	\E & \left[ X^*(t+1) - X^*(t) \given Q(1),\ldots,Q(t), \tau > t \right]\\
	& = \sigma \left( \E \left[ X(t+1) - X(t) \given Q(t), \tau > t \right] - \left( x(t+1) - x(t) \right) \right) - \frac{1}{2}\left( E(t+1) - E(t) \right)\\
	& = \sigma \left( \left(x'(t) \pm O \left(\frac{E(t)}{n-t}\right) \right) - \left( x'(t) \pm O \left(\frac{E(t)}{n-t}\right) \right) \right) - \frac{1}{2} \left( \frac{K E(t)}{n-t} \pm \frac{E(t)}{n-t} \right)\\
	& = - \frac{K E(t)}{2(n-t)} \pm O \left(\frac{E(t)}{n-t}\right) \leq 0,
	\end{align*}
	where the last inequality holds provided the constant $K$ was chosen to be large enough.
\end{proof}

We are ready to prove Proposition \ref{prop:tau inf whp}.

\begin{proof}[Proof of Proposition \ref{prop:tau inf whp}]
	We observe that $\tau < \infty$ only if there exists some $0 \leq t \leq T$, $\alpha \in I_M$, and $(x,y) \in [n]^2$ such that for $X$ one of $R_{y,\alpha},C_{x,\alpha}D^+_{x+y,\alpha},D^-_{x-y,\alpha}$, or $\tA_\alpha$ either $X^+(t) > E(t)/2$ or $X^-(t) > E(t)/2$.
	
	Let $X$ be one of the sequences of random variables above. By Lemma \ref{lem:supermartingale condition}, both $X^+$ and $X^-$ are supermartingales. Furthermore, by Lemma \ref{lem:one step bound}, $X^+$ and $X^-$ change by at most $O(1/(1-t/n))$ in time step $t$. Therefore, by Theorem \ref{thm:azuma}, for every $0 \leq t \leq T$:
	\[
	\Prob \left[ X^+(t) > \frac{1}{2}E(t) \right], \Prob \left[ X^-(t) > \frac{1}{2}E(t) \right] \leq \exp \left( - \Omega \left( \frac{E(t)^2 (1-t/n)^2 }{T} \right) \right) = \exp \left( - \Omega \left( n^{0.75} \right) \right).
	\]
	By applying a union bound to the polynomially many random variables and times $0 \leq t \leq T$ we conclude that $\Prob \left[ \tau < \infty \right] = \exp \left( - \Omega \left( n^{0.75} \right) \right)$, as desired.
\end{proof}

We now show that w.h.p.\ $Q(T)$ approximates $\delta$. In the next claim $N$ (which we assume is larger than $\varepsilon^{-2}$) is the constant appearing in the definition of $\delta$.

\begin{clm}\label{clm:Q(T) approximates delta}
	For every $\alpha \in I_N$ it holds that
	\[
	\Prob \left[ \left||\alpha_n \cap Q(T)| - \delta(\alpha)n \right| > \varepsilon^5 n \right] \leq \exp \left( - \Omega \left( n^{0.75} \right) \right).
	\]
\end{clm}

\begin{proof}
	Observe that since (by definition) $N$ divides $M$, the partition $I_M$ is a refinement of $I_N$. For $\alpha \in I_N$, let $W_\alpha$ be the number of times $1 \leq t \leq T$ such that $Y_t \subseteq \alpha$. Then $W_\alpha$ is distributed binomially with parameters $T,\delta(\alpha)$. Therefore, by Chernoff's inequality:
	\[
	\Prob \left[ |W_\alpha - \E W_\alpha| > \varepsilon^5 n / 2 \right] = \exp \left( - \Omega (n) \right).
	\]
	Since $|\E W_\alpha - n\delta(\alpha)| = o(n)$:
	\begin{equation}\label{eq:W_alpha Chernoff}
	\Prob \left[ |W_\alpha - \delta(\alpha) n  > \varepsilon^5 n | \right] = \exp \left( - \Omega (n) \right).
	\end{equation}
	If $\tau = \infty$ then Algorithm \ref{alg:random} did not abort in which case $|\alpha_n \cap Q(T)| = W_\alpha$. Therefore, by a union bound:
	\begin{align*}
	\Prob \left[ \left||\alpha_n \cap Q(T)| - \delta(\alpha)n \right| > \varepsilon^5 n \right] \leq \Prob \left[ \tau < \infty \right] + \Prob \left[ \left|W_\alpha - \delta(\alpha)n \right| > \varepsilon^5 n \right]\\
	\stackrel{\text{\eqref{eq:W_alpha Chernoff} and Proposition \ref{prop:tau inf whp}}}{\leq} \exp \left( - \Omega \left( n^{0.75} \right) \right),
	\end{align*}
	as claimed.
\end{proof}

We conclude the section by calculating the entropy of Algorithm \ref{alg:random}.

\begin{clm}\label{clm:one step entropy}
	Let $0 \leq t < T$. Then
	\begin{align*}
	H (X_{t+1} \given X_1,X_2,\ldots,X_{t}) = & 2 \log(n-t) - D^M(\delta) + \sum_{ \alpha \in J_M } \delta^+(\alpha) \log(1 - M\delta^+(\alpha)t/n)\\
	& + \sum_{\alpha \in J_M} \delta^-(\alpha) \log(1 - M\delta^-(\alpha)t/n) \pm O \left( \frac{ME(t)}{n-t} \right).
	\end{align*}
\end{clm}

\begin{proof}
	Let $0 \leq t < T$. By the law of total probability:
	\begin{align*}
	H (X_{t+1} \given X_1,\ldots,X_{t}) = H(X_{t+1} \given X_1,\ldots,X_{t}, \tau > t) \Prob [\tau > t] + H(X_{t+1} \given X_1,\ldots,X_{t}, \tau \leq t) \Prob [ \tau \leq t].
	\end{align*}
	By Proposition \ref{prop:tau inf whp} $\Prob [\tau \leq t] = \exp \left( - \Omega \left( n^{0.75} \right) \right)$. Additionally, $X_{t+1}$ is distributed among at most $n^2$ elements. Therefore:
	\[
	H(X_{t+1} \given X_1,\ldots,X_{t}, \tau \leq t) \Prob [ \tau \leq t] \leq \log(n^2) \exp \left( - \Omega \left( n^{0.75} \right) \right) = \exp \left( - \Omega \left( n^{0.75} \right) \right).
	\]
	Thus:
	\[
	H (X_{t+1} \given X_1,\ldots,X_{t}) = H(X_{t+1} \given X_1,\ldots,X_{t}, \tau > t) \pm \exp \left( - \Omega \left( n^{0.75} \right) \right).
	\]
	By the chain rule:
	\begin{align*}
	H (X_{t+1} \given X_1,\ldots,X_{t}, \tau > t) = H(Y_{t+1} \given X_1,\ldots,X_{t}, \tau > t) + H(X_{t+1} \given X_1,\ldots,X_{t}, Y_{t+1}, \tau > t).
	\end{align*}
	Recall that $Y_{t+1}$ is independent of $X_1,\ldots,X_t$ and the event $\tau > t$. By its definition:
	\[
	H(Y_{t+1} \given X_1,\ldots,X_{t}, \tau > t) = - \sum_{ \alpha \in I_M } \delta(\alpha) \log(\delta(\alpha)).
	\]
	By definition of $X_{t+1}$:
	\[
	H(X_{t+1} \given X_1,\ldots,X_{t}, Y_{t+1}, \tau > t) = \sum_{ \alpha \in I_M } \delta(\alpha) \log(A_\alpha(t)).
	\]
	By \cref{obs:A_alpha asymptotic} if $\tau > t$ then for every $\alpha \in I_M$:
	\[
	A_\alpha(t) = |\alpha_n| \left( 1 - \frac{t}{n} \right)^2 \left( 1 - M\delta^+(\alpha) \frac{t}{n} \right) \left( 1 - M \delta^-(\alpha) \frac{t}{n} \right) \left( 1 \pm O \left( \frac{ME(t)}{n-t} \right) \right).
	\]
	Thus:
	\begin{align*}
	&H(X_{t+1} \given X_1,\ldots,X_{t}, Y_{t+1}, \tau > t) =\\
	&\sum_{ \alpha \in I_M } \delta(\alpha) \left( \log(|\alpha_n|) + 2\log(1-t/n) + \log(1-M\delta^+(\alpha)t/n) + \log(1-M\delta^-(\alpha)t/n) \right)\\
	& \qquad \pm O \left( \frac{ME(t)}{n-t} \right)\\
	& = 2\log(1-t/n) + \sum_{ \alpha \in I_M } \delta(\alpha) \log(|\alpha_n|) + \sum_{ \alpha \in J_M} \delta^+(\alpha) \log(1 - M\delta^+(\alpha)t/n)\\
	& \qquad + \sum_{ \alpha \in J_M} \delta^-(\alpha) \log(1 - M\delta^-(\alpha)t/n) \pm O \left( \frac{ME(t)}{n-t} \right).
	\end{align*}
	Recall that $|\alpha|$ is the area of $\alpha$ and that for every $\alpha \in I_M$ there holds $|\alpha_n| = n^2 |\alpha| \left( 1 \pm O(M/n) \right)$. Thus
	\[
	\sum_{ \alpha \in I_M } \delta(\alpha) \log(|\alpha_n|) = 2\log(n) + \sum_{ \alpha \in I_M } \delta(\alpha) \log(|\alpha|) \pm O \left( \frac{M}{n} \right).
	\]
	Finally, we recall that by definition
	\[
	- \sum_{ \alpha \in I_M } \delta(\alpha) \log(\delta(\alpha)) + \sum_{ \alpha \in I_M } \delta(\alpha) \log(|\alpha|) = -D^M(\delta).
	\]
	Therefore:
	\begin{align*}
	H (X_{t+1} \given X_1,X_2,\ldots,X_{t}) = & 2 \log(n-t) - D^M(\delta) + \sum_{ \alpha \in J_M } \delta^+(\alpha) \log(1 - M\delta^+(\alpha)t/n)\\
	& + \sum_{\alpha \in J_M} \delta^-(\alpha) \log(1 - M\delta^-(\alpha)t/n) \pm O \left( \frac{ME(t)}{n-t} \right),
	\end{align*}
	as desired.
\end{proof}

In the statement of the next lemma $K$ is the constant used to define $T$.

\begin{lem}\label{lem:random phase entropy}
	It holds that
	\[
	H(X_1,X_2,\ldots,X_T) = n \left( H_q^M(\delta) + 2\log n - 1 \right) \pm n^{1-1/K^3}.
	\]
\end{lem}

\begin{proof}
	By the chain rule $H(X_1,X_2,\ldots,X_T) = \sum_{t=1}^T H(X_t \given X_1,\ldots,X_{t-1})$. By \cref{clm:log integral} for $\alpha \in J_M$ and $* \in \{+,-\}$:
	\begin{align*}
	\sum_{t=0}^{T-1} & \delta^*(\alpha) \log(1 - M\delta^*(\alpha)t/n)\\
	& = -\frac{n}{M} \left( (1-M\delta^*(\alpha) \log(1-M\delta^*(\alpha)) + M\delta^*(\alpha) \right) \pm \frac{1}{M} 3(n-T) \log(1-T/n)\\
	& = - \frac{n}{M} \left( M\overline{\delta}^*(\alpha) \log \left( M\overline{\delta}^*(\alpha) \right) + M\delta^*(\alpha) \right) \pm \frac{3}{M}(n-T) \log(1-T/n)\\
	& = - n \left( \overline{\delta}^*(\alpha) \log \left( \frac{\overline{\delta}^*(\alpha)}{1/(2M)} \right) - \overline{\delta}^*(\alpha) \log(2) + \delta^*(\alpha) \right) \pm \frac{3}{M}(n-T) \log(1-T/n).
	\end{align*}
	Therefore:
	\begin{align*}
	\sum_{t=0}^{T-1} \sum_{\alpha \in J_M} \delta^*(\alpha) \log(1 - M\delta^*(\alpha)t/n) = - n D( \{\overline{\delta}^*(\alpha) \}_{\alpha \in J_M}) + n\log(2) - n \pm 6(n-T)\log(1-T/n).
	\end{align*}
		
	Applying Claim \ref{clm:one step entropy}:
	\begin{align*}
	\sum_{t=1}^T H(X_t \given X_1,\ldots,X_{t-1})
	= & \sum_{t=0}^{T-1}2\log(n-t) - T D^M(\delta) - nD(\{ \distPlus{\delta}(\alpha) \}_{\alpha \in J_M}) - nD(\{ \distMinus{\delta}(\alpha) \}_{\alpha \in J_M} )\\
	& + 2n\log(2) - 2n \pm O \left( (n-T) |\log (1-T/n)| + \sum_{t=0}^{T-1} \frac{ME(t)}{n-t} \right)\\
	\stackrel{\text{Claim \ref{clm:log integral}}}{=} & n (H_q^M(\delta) + 2\log(n) - 1) \pm O \left( (n-T) |\log (1-T/n)| + \frac{nME(T)}{n-T} \right).
	\end{align*}
	Taking account the definition of $E(t)$ and $T$ we have $nME(T)/(n-T) \leq n-T = O(n^{1-1/K^2})$. Therefore:
	\begin{align*}
	\sum_{t=1}^T H(X_t \given X_1,\ldots,X_{t-1}) = n \left( H_q(\delta) + 2\log n - 1 \right) \pm n^{1-1/K^3},
	\end{align*}
	as claimed.
\end{proof}

\subsection{Absorbers}\label{sec:absorbers}

In this section we analyze Algorithm \ref{alg:absorption}. We wish to show that it is unlikely to abort. The next lemma provides a sufficient condition.

\begin{definition}\label{def:epsilon absorbing}
	Let $\ell > 0$. A partial $n$-queens configuration $Q$ is \termdefine{$\ell$-absorbing} if for every $(c,r) \in [n]^2$, it holds that $\left|\mB_Q(c,r)\right| \geq \ell$.
\end{definition}

The following is Lemma 4.2 in \cite{luria2021lower}.

\begin{lem}\label{lem:absorbing procedure works}
	Suppose $|Q(T)| = T$ and $Q(T)$ is $10 (n-T)$-absorbing. Then Algorithm \ref{alg:absorption} does not abort.
\end{lem}

For the proof we refer the reader to \cite{luria2021lower}. We mention only that the key observation is that for every $(c,r) \in [n]^2$, every step of Algorithm \ref{alg:absorption} ``destroys'' at most $9$ absorbers in $\mB_{Q(T)}(c,r)$.

The next lemma asserts that w.h.p.\ $Q(T)$ is $\Omega(n)$-absorbing. By Proposition \ref{prop:tau inf whp}, w.h.p.\ $n - |Q(T)| = n - T = o(n)$. It then follows from Lemma \ref{lem:absorbing procedure works} that Algorithm \ref{alg:absorption} succeeds in constructing an $n$-queens configuration.

\begin{lem}\label{lem:an abundance of absorbers}
	W.h.p.\ $Q(T)$ is $\Omega(n)$-absorbing.
\end{lem}

The intuition is that $Q(T)$ contains approximately $n$ queens, each occupying a single diagonal of each type. However, the grid $[n]^2$ contains approximately $2n$ diagonals of each type. Therefore, if one chooses a diagonal uniformly at random the probability that it is unoccupied is bounded away from $0$. If we fix $(c,r)$ and choose $(x,y) \in Q(T)$ uniformly at random, we might imagine that the (four) diagonals containing $(c,y)$ and $(x,r)$ are distributed uniformly at random, which would imply that with constant probability they are unoccupied, in which case $(x,y)$ is an absorber for $(c,r)$.

In order to prove Lemma \ref{lem:an abundance of absorbers} we couple the random process $\{Q(t)\}_{t=0}^T$ with a random set that is the union of binomial random subsets of $[n]^2$. Let $\{ s_x \}_{x \in [n]^2}$ be i.i.d.\ uniform random variables in $[0,1]$. Consider the following process: Let $\tQ(0) = \emptyset$. Define $Y_1,\ldots,Y_T \in I_M$ as in Algorithm \ref{alg:random}. Suppose we have constructed $\tQ(t-1)$. Let $\alpha = Y_t$. Then, let $X_t$ be the element of $\alpha_n \cap \mA_{\tQ(t-1)}$ minimizing $s_x$. If $\alpha_n \cap \mA_{\tQ(t-1)} = \emptyset$, abort and set $X_t=X_{t+1}=\ldots=X_T=*$. Clearly, $\{Q(t)\}_{t=0}^T$ and $\{\tQ(t)\}_{t=0}^T$ have identical distributions, so we may (and do) identify them.

Define $R \subseteq [n]^2$ as follows: Recall that for $x \in [n]^2$, $\alpha(x)$ is the $\alpha \in I_M$ such that $x \in \alpha_n$. Include $x$ in $R$ if $s_x < \varepsilon \delta(\alpha(x))M^2 / n$. Let $R' \subseteq R$ be the set of elements $x \in R$ that do not share a row, column, or diagonal with any other element of $R$. Let $\tR = \{ x \in R' : s_x < \varepsilon \delta(\alpha(x))M^2 / (20n) \}$. Clearly, $\tR$ is a partial $n$-queens configuration. We will show that w.h.p.\ (over the random variables $\{s_x\}_{x \in [n]^2}$) every partial configuration containing $\tR$ and contained in $R$ is $\Omega(n)$-absorbing. Furthermore, we will show that w.h.p.\ there exists some $0 \leq T_R \leq T$ such that $\tR \subseteq Q(T_R) \subseteq R$. This implies that $Q(T_R)$ is $\Omega(n)$-absorbing. Finally, we will show that w.h.p.\ a constant fraction of the absorbers in $Q(T_R)$ survive until the end of Algorithm \ref{alg:random}, which will imply Lemma \ref{lem:an abundance of absorbers}.

Let $T_R \coloneqq \lfloor \varepsilon n / 8 \rfloor$. In order to show that $Q(T_R) \subseteq R$ we first prove that $R'$ intersects every $\alpha \in I_M$ in many places. We will use the following concentration inequality, which is a special case of \cite[Theorem 1.10]{warnke2016method}.

\begin{thm}\label{thm:bounded differences}
	Let $X_1,\ldots,X_N$ be independent (but not necessarily identically distributed) random variables taking values in a finite set $\Lambda$. Let $p \in (0,1]$ satisfy ${max \{ \Prob [X_i = \eta] : \eta \in \Lambda, i \in [N] \}} \geq 1-p$. Assume that for $K>0$ the function $f:\Lambda^N \to \R$ satisfies the Lipschitz condition $\left| f(\omega) - f(\omega') \right| \leq K$ whenever $\omega,\omega' \in \Lambda^N$ differ by a single coordinate. Then, for all $\lambda \geq 0$:
	\[
	\Prob \left[ \left|f(X_1,\ldots,X_N) - \E \left[ f(X_1,\ldots,X_N) \right]\right| \geq \lambda \right] \leq 2 \exp \left( - \frac{\lambda^2}{2K^2 N p + 2K\lambda/3} \right).
	\]
\end{thm}

We will need the following bounds on the probability that all values $s_{(x,y)}$, where $(x,y)$ ranges over a row, column, or diagonal, exceed a given threshold.

\begin{clm}\label{clm:product bounds}
	Let $(x,y) \in [n]^2$ and let $\varepsilon \geq \varepsilon_0 > 0$. The following hold:
	\begin{enumerate}
	\item\label{itm:row prod} $\prod_{a \in [n]} \left( 1 - \frac{\varepsilon_0 \delta(\alpha(a,y)) M^2}{n} \right) \geq 1 - \varepsilon_0$,

	\item\label{itm:col prod} $\prod_{a \in [n]} \left( 1 - \frac{\varepsilon_0 \delta(\alpha(x,a)) M^2}{n} \right) \geq 1 - \varepsilon_0$,

	\item\label{itm:plus diag prod} $\prod_{(a,b) \in [n]^2: a+b = x+y} \left( 1 - \frac{\varepsilon_0 \delta(\alpha(a,b)) M^2}{n} \right) \geq 1 - \varepsilon_0$,

	\item\label{itm:minus diag prod} $\prod_{(a,b) \in [n]^2: a-b = x-y} \left( 1 - \frac{\varepsilon_0 \delta(\alpha(a,b)) M^2}{n} \right) \geq 1 - \varepsilon_0$.
	\end{enumerate}
\end{clm}

\begin{proof}
	We will prove \ref{itm:row prod} and \ref{itm:plus diag prod}. \ref{itm:col prod} and \ref{itm:minus diag prod} follow similarly. We use the fact that for sufficiently small $z$, $\log(1-z) \geq -z-z^2$. This implies:
	\[
	\prod_{a \in [n]} \left( 1 - \frac{\varepsilon_0 \delta(\alpha(a,y)) M^2}{n} \right) \geq \exp \left( - \sum_{a \in [n]} \left(\frac{\varepsilon_0 \delta(\alpha(a,y)) M^2}{n} + \left(\frac{\varepsilon_0 \delta(\alpha(a,y)) M^2}{n}\right)^2 \right) \right).
	\]
	Because $\delta$ has uniform marginals, $\delta(\alpha) \leq 1/M$ for every $\alpha \in I_M$. Thus $\sum_{a \in [n]} (\varepsilon_0 \delta(\alpha(a,y)) M^2 / n)^2 \leq \varepsilon_0^2 M^2 / n = \oone$. Therefore:
	\begin{align*}
	\prod_{a \in [n]} \left( 1 - \frac{\varepsilon_0 \delta(\alpha(a,y)) M^2}{n} \right) & \geq (1-\oone) \exp \left( - \frac{\varepsilon_0 M^2}{n} \sum_{a \in [n]} \delta(\alpha(a,y)) \right)\\
	& = (1-\oone) \exp \left( - \frac{\varepsilon_0 M^2}{n} \sum_{ \alpha \in I_M } \delta(\alpha)L_{y,\alpha}^r \right)\\
	& \stackrel{\text{Claim \ref{clm:line sums} \ref{itm:row sum}}}{=} (1-\oone) \exp \left( - \frac{\varepsilon_0 M^2}{n} \times \frac{n}{2M^2} \right) \geq 1 - \varepsilon_0,
	\end{align*}
	proving \ref{itm:row prod}.
	
	We turn to \ref{itm:plus diag prod}, which is proved similarly.
	\begin{align*}
	\prod_{(a,b) \in [n]^2, a+b = x+y} \left( 1 - \frac{\varepsilon_0 \delta(\alpha(a,b)) M^2}{n} \right) & \geq (1 - \oone) \exp \left( - \frac{\varepsilon_0 M^2}{n} \sum_{ \alpha \in I_M } \delta(\alpha) L_{x+y,\alpha}^+ \right)\\
	& \stackrel{\text{Claim \ref{clm:line sums} \ref{itm:plus diagonal sum}}}{\geq} (1 - \oone) \exp \left( - \frac{\varepsilon_0 M^2}{n} \times \frac{\delta^+(\alpha) n}{2M} \right).
	\end{align*}
	$\delta$ has sub-uniform marginals, so $\delta^+(\alpha) \leq 1/M$. Hence:
	\begin{align*}
	\prod_{(a,b) \in [n]^2, a+b = x+y} \left( 1 - \frac{\varepsilon_0 \delta(\alpha(a,b)) M^2}{n} \right) \geq (1-\oone) \exp \left( - \frac{\varepsilon_0}{2} \right) \geq 1 - \varepsilon_0,
	\end{align*}
	proving \ref{itm:plus diag prod}.
\end{proof}

\begin{clm}\label{clm:R' intersection alpha_n}
	With probability $1 - \exp \left( - \Omega \left( n^{0.6} \right) \right)$ for every $\alpha \in I_M$ it holds that $\left| \alpha_n \cap R' \right| \geq \frac{7}{6} \delta(\alpha) T_R$.
\end{clm}

\begin{proof}
	Let $\alpha \in I_M$ and let $(x,y) \in \alpha_n$. By definition, $(x,y) \in R'$ if and only if $(x,y) \in R$ and $(a,b) \notin R$ for every $(a,b) \in [n]^2$ that shares a row, column, or diagonal with $(x,y)$. Because the elements of $R$ are chosen independently of each other:
	\begin{align*}
	\Prob & \left[ (x,y) \in R' \right] = \Prob \left[ (x,y) \in R \right] \left( \prod_{a \in [n], a \neq x} \Prob \left[ (a,y) \notin R \right]\right) \left( \prod_{ a \in [n], a \neq y } \Prob \left[ (x,a) \notin R \right] \right) \times\\
	& \left( \prod_{(a,b) \in [n]^2, a+b = x+y, (a,b) \neq (x,y)} \Prob \left[ (a,b) \notin R \right] \right) \left( \prod_{(a,b) \in [n]^2, a-b = x-y, (a,b) \neq (x,y)} \Prob \left[ (a,b) \notin R \right] \right).
	\end{align*}
	
	By definition, $\Prob [(x,y) \in R] = \varepsilon \delta(\alpha)M^2 / n$. Similarly, for every $(a,b) \in [n]^2$ we have $\Prob \left[ (a,b) \notin R \right] = 1 - \varepsilon \delta(\alpha{(a,b)}) M^2 / n$. Therefore, by Claim \ref{clm:product bounds}:
	\[
	\Prob \left[ (x,y) \in R' \right] \geq \frac{\varepsilon \delta(\alpha) M^2}{n} (1-\varepsilon)^4.
	\]
	Thus, since $|\alpha_n| \geq n^2/(5M^2)$ for every $\alpha$:
	\[
	\E \left[ |\alpha_n \cap R'| \right] \geq \frac{\varepsilon \delta(\alpha) M^2}{n} (1-\varepsilon)^4 | \alpha_n | \geq \frac{\varepsilon \delta(\alpha) n}{5} \geq \frac{7 \delta(\alpha) T_R}{5}.
	\]
	We observe that adding or removing an element from $R$ changes $R'$ by at most $4$ elements. Therefore, by Theorem \ref{thm:bounded differences}, with $\lambda = \E \left[ |\alpha_n \cap R'| \right] - \frac{7}{6} \delta(\alpha) T_R = \Omega \left( n / M^2 \right)$, $K = 4 = O(1)$, and $p = \max \{\varepsilon \delta(\alpha(x,y)) M^2/n : (x,y) \in [n]^2 \} = O(1/n)$:
	\begin{align*}
	\Prob \left[ |\alpha_n \cap R'| < \frac{7}{6} \delta(\alpha) T_R \right] & \leq 2 \exp \left( - \frac{\lambda^2}{ 2K^2n^2p + 2K \lambda/ 3 } \right)\\
	& = \exp \left( - \Omega \left( n/M^4 \right) \right) = \exp \left( - \Omega \left( n^{0.6} \right) \right).
	\end{align*}
	The claim follows by applying a union bound to the polynomially many elements of $I_M$.
\end{proof}

For $\alpha \in I_M$ let $W_\alpha$ be the number of $1 \leq t \leq T_R$ such that $Y_t = \alpha$.

\begin{clm}\label{clm:W_alpha}
	With probability $1 - \exp \left( - \Omega \left( n^{0.8} \right) \right)$ for every $\alpha \in I_M$, $W_\alpha = \left( 1 \pm \frac{1}{12} \right) \delta(\alpha) T_R$.
\end{clm}

\begin{proof}
	Observe that $W_\alpha$ is distributed binomially with parameters $T_R,\delta(\alpha)$. In particular $\E W_\alpha = T_R \delta(\alpha) = \Theta \left( n/M^2 \right) = \Theta \left( n^{0.8} \right)$. The claim follows by applying Chernoff's inequality and a union bound.
\end{proof}

\begin{clm}\label{clm:small tR}
	With probability $1 - \exp \left( - \Omega \left( n^{0.6} \right) \right)$ for every $\alpha \in I_M$ there are at most $\frac{1}{2} \delta(\alpha) T_R$ positions $(x,y) \in \alpha_n$ such that $s_{(x,y)} < \varepsilon \delta(\alpha)M^2/(20n)$.
\end{clm}

\begin{proof}
	Let $\alpha \in I_M$ and let $S(\alpha) = |\{ (x,y) \in \alpha_n : s_{(x,y)} < \varepsilon\delta(\alpha)M^2/(20n) \}|$. Then $S(\alpha)$ is distributed binomially with parameters $|\alpha_n|,\varepsilon\delta(\alpha)M^2/(20n)$. Recall that $M = \Theta(n^{0.1})$ and that there holds $\delta(\alpha) = O(1/M^2)$. Therefore $\E S(\alpha) = |\alpha_n|\varepsilon \delta(\alpha)M^2/(20n) = \Theta \left( n^{0.8} \right)$. For every $\alpha$, $|\alpha_n| \leq n^2/(2M^2)+O(n/M)$. Hence $\E S(\alpha) \leq \frac{2}{5}\delta(\alpha)T_R$. The claim follows from Chernoff's inequality and a union bound.
\end{proof}

\begin{clm}
	With probability $1 - \exp \left( - \Omega \left( n^{0.6} \right) \right)$ it holds that $\tR \subseteq Q(T_R) \subseteq R$.
\end{clm}

\begin{proof}
	We first prove that w.h.p.\ $Q(T_R) \subseteq R$. We will show that if $Q(T_R) \nsubseteq R$ then there exists some $\alpha \in I_M$ such that $W_\alpha > |\alpha_n \cap R'|$. Indeed, suppose that $Q(T_R) \nsubseteq R$. Then there exists a minimal $t \leq T_R$ such that $X_t \notin R$. Let $x = X_t$ and $\alpha = \alpha(x)$. By definition of $R$, $s_x \geq \varepsilon \delta(\alpha) M^2 / n$. We claim that $\alpha_n \cap R' \subseteq Q(t-1)$. Let $y \in \alpha_n \cap R'$. It holds that $s_y < \varepsilon \delta(\alpha) M^2 / n \leq s_x$. By definition of the process, $s_x$ is smaller than $s_z$ for every $z \in \alpha_n \cap \mA_{Q(t-1)}$. Therefore, $y \notin \mA_{Q(t-1)}$. By definition of $R'$, $y$ is not threatened by any element of $R$. Since $Q(t-1) \subseteq R$ this means that $y$ is not threatened by any element of $Q(t-1)$. Therefore, since $y$ is unavailable at time $t-1$, it must be that $y \in Q(t-1)$. This means that $W_\alpha \geq |\alpha_n \cap R'| + 1 > |\alpha_n \cap R'|$.
	
	We have shown that if $Q(T_R) \nsubseteq R$ then there exists some $\alpha \in I_M$ such that $W_\alpha > |\alpha_n \cap R'|$. However, Claims \ref{clm:R' intersection alpha_n} and \ref{clm:W_alpha} imply that with probability $1 - \exp \left( - \Omega \left( n^{0.6} \right) \right)$ for every $\alpha \in I_M$:
	\[
	W_\alpha \leq \frac{13}{12} \delta(\alpha) T_R < \frac{7}{6} \delta(\alpha) T_R \leq |\alpha_n \cap R'|.
	\]
	Therefore $Q(T_R) \subseteq R$ with probability $1 - \exp \left( - \Omega \left( n^{0.6} \right) \right)$.
	
	We now show that w.h.p.\ $\tR \subseteq Q(T_R)$. If $\tR \nsubseteq Q(T_R)$ then there exists some $x \in \tR \setminus Q(T_R)$. Let $\alpha = \alpha(x)$. By definition, $s_x < \varepsilon \delta(\alpha) M^2 / (20n)$ and $x$ is not threatened by any element of $R$. Therefore, if $Q(T_R) \subseteq R$ then for every $0 \leq t \leq T_R$, $x \in \mA_{Q(t)}$. Since $x \notin Q(T_R)$ this means that for every $1 \leq t \leq T_R$ $x$ did not minimize $s_x$ among all elements of $\mA_{Q(t)} \cap \alpha_n$. Therefore there exist at least $W_\alpha$ elements $z \in \alpha_n$ such that $s_z < s_x < \varepsilon \delta(\alpha) M^2 / (20n)$.
	
	We have shown that if $\tR \nsubseteq Q(T_R)$ then either $Q(T_R) \nsubseteq R$ or there exists some $\alpha \in I_M$ such that $|\{ z \in \alpha_n : s_z < \varepsilon \delta(\alpha) M^2\!/ (20n) \}| \geq W_\alpha$. However, we have shown that $Q(T_R) \subseteq R$ with probability $1 - \exp \left( - \Omega \left( n^{0.6} \right) \right)$. Furthermore, by Claims \ref{clm:W_alpha} and \ref{clm:small tR} with probability $1 - \exp \left( - \Omega \left( n^{0.6} \right) \right)$ for every $\alpha \in I_M$:
	\[
	|\{ z \in \alpha_n : s_z < \varepsilon \delta(\alpha) M^2 / (20n) \}| \leq \frac{1}{2} \delta(\alpha) T_R < \frac{11}{12} \delta(\alpha)T_R \leq W_\alpha.
	\]
	Therefore $\tR \subseteq Q(T_R)$ with probability $1 - \exp \left( - \Omega \left( n^{0.6} \right) \right)$, as desired.
\end{proof}

Next, we show that w.h.p.\ $Q(T_R)$ is $\Omega(n)$-absorbing. Recall that by Property \ref{itm:everywhere positive density} $\delta$ has density bounded away from $0$. Let $\rho > 0$ be a lower bound on the density of $\delta$.

\begin{clm}\label{clm:Q(T_R) absorbing}
	Let $C = \varepsilon\rho/1000$. With probability $1 - \exp \left( - \Omega \left( n^{0.6} \right) \right)$ it holds that $Q(T_R)$ is $C n$-absorbing.
\end{clm}

\begin{proof}
	We will show that with probability $1 - \exp \left( - \Omega \left( n^{0.6} \right) \right)$ for every $(x,y) \in [n]^2$ there are at least $Cn$ queens $(a,b) \in \tR$ such that:
	\begin{itemize}
		\item $(a,b)$ and $(x,y)$ do not share a diagonal.
		
		\item The diagonals passing through $(x,b)$ and $(a,y)$ do not contain any elements of $R$.
	\end{itemize}
	If, as happens with probability $1 - \exp \left( - \Omega \left( n^{0.6} \right) \right)$, $\tR \subseteq Q(T_R) \subseteq R$, then every such position satisfies $(a,b) \in \mB_{Q(T_R)}(x,y)$. Hence $Q(T_R)$ is $Cn$-absorbing with probability $1 - \exp \left( - \Omega \left( n^{0.6} \right) \right)$.
	
	Let $(x,y) \in [n]^2$. Let $K_{(x,y)}$ be the number of queens $(a,b)$ satisfying the conditions above. We wish to apply Theorem \ref{thm:bounded differences} to $K_{(x,y)}$. We first show that $K_{(x,y)}$ can be expressed as a function of independent random variables. Let $\Lambda = \{0,1,2\}$. For $(a,b) \in [n]^2$, let
	\[
	S_{(a,b)} =
	\begin{cases}
	0 & s_{(a,b)} < \varepsilon \delta(\alpha(a,b)) M^2 / (20n)\\
	1 & s_{(a,b)} \in [\varepsilon \delta(\alpha(a,b)) M^2 / (20n), \varepsilon \delta(\alpha(a,b)) M^2 / n]\\
	2 & s_{(a,b)} > \varepsilon \delta(\alpha(a,b)) M^2 / n.
	\end{cases}
	\]
	Note that the sets $R$ and $\tR$, and hence the value of $K_{(x,y)}$, can be recovered from the random variables $\{ S_{(a,b)} \}_{(a,b) \in [n]^2}$. Hence, we may apply Theorem \ref{thm:bounded differences} together with a union bound over the $n^2$ positions. To do so it suffices to show the following:
	\begin{enumerate}
		\item\label{itm:K expectation} $\E \left[ K_{(x,y)} \right] \geq 2Cn$.
		
		\item\label{itm:K Lipschitz} If we change $R$ or $\tR$ by either removing or adding a queen then $K_{(x,y)}$ changes by at most $6$.
		
		\item\label{itm:K p bound} For $p = M / n$ and every $(a,b) \in [n]^2$ it holds that $\Prob \left[ S_{(a,b)} = 2 \right] \geq 1 - p$.
	\end{enumerate}
	Indeed, if these conditions hold then by Theorem \ref{thm:bounded differences}:
	\begin{align*}
	\Prob \left[ K_{(x,y)} \leq C n \right] \leq \Prob \left[ K_{(x,y)} \leq \E \left[ K_{(x,y)} \right] - C n \right] & \leq 2 \exp \left( - \frac{C^2n^2}{72 n^2 M/n + 12Cn/3} \right)\\
	& = \exp \left( - \Omega(n/M) \right) = \exp \left( - \Omega \left( n^{0.9} \right) \right).
	\end{align*}
	We begin with \ref{itm:K expectation}. Let $(a,b) \in [n]^2$ such that $(a,b)$ and $(x,y)$ do not share a diagonal, row, or column. By a calculation similar to the one in the proof of Claim \ref{clm:R' intersection alpha_n},
	\[
	\Prob \left[ (a,b) \in \tR \right] \geq \Prob [s_{(a,b)} < \varepsilon \delta(\alpha(a,b)) M^2 / (20n)] (1-\varepsilon^4) = \frac{(1-\varepsilon)^4 \varepsilon \delta(\alpha(a,b)) M^2}{20n}.
	\]
	By assumption $\delta(\alpha(a,b)) \geq \rho |\alpha(a,b)| \geq \rho/(4M^2)$. Thus:
	\[
	\Prob \left[ (a,b) \in \tR \right] \geq (1-\varepsilon)^4 \varepsilon \frac{\rho}{4M^2} \times \frac{M^2}{20n} = \frac{(1-\varepsilon)^4 \varepsilon \rho}{80n}.
	\]
	Now, by Claim \ref{clm:product bounds} the probability that the four diagonals incident to $(a,y)$ and $(x,b)$ do not contain elements of $R$ is $\geq (1-\varepsilon)^4$. Therefore:
	\[
	\E \left[ K_{(x,y)} \right] \geq (1-o(1)) n^2 \frac{(1-\varepsilon)^4 \varepsilon \rho}{80n} (1 - \varepsilon)^4 \geq  2Cn,
	\]
	proving \ref{itm:K expectation}.
	
	To see that \ref{itm:K Lipschitz} holds observe that adding a queen $(a,b)$ to $R$ or $\tR$ can increase $K_{(x,y)}$ by at most $1$. At the same time, $K_{(x,y)}$ can decrease by at most $6$, as there are at most $4$ queens $(c,r) \in \tR$ such that $(a,b)$ occupied a diagonal incident to $(c,y)$ or $(x,r)$, and at most $2$ queens in $\tR$ sharing a row or column with $(a,b)$.
	
	Finally, \ref{itm:K p bound} holds because every $\alpha \in I_M$ is contained in a diagonal of width $1/M$. Therefore, for every $(a,b) \in [n]^2$, it holds that
	\[
	\Prob \left[ S_{(a,b)} = 2 \right] = 1 - \frac{\varepsilon \delta(\alpha(a,b)) M^2}{n} \geq 1 - \frac{\varepsilon M^{-1} M^2}{n} \geq 1 - \frac{M}{n}.\qedhere
	\]
\end{proof}

We now show that w.h.p.\ a constant fraction of the absorbers in $Q(T_R)$ are also absorbers in $Q(T)$ (i.e., the outcome of Algorithm \ref{alg:random}). In the next claim, $\tau$ refers to the stopping time in Definition \ref{def:tau} and the constant $C$ is the same as in the statement of Claim \ref{clm:Q(T_R) absorbing}. Define $\zeta \coloneqq \eta / (1-\eta)$, where $\eta$ is the constant from \ref{itm:diagonals not full}.

\begin{clm}\label{clm:Q(T) absorbing}
	Suppose that $Q(T_R)$ is $Cn$-absorbing and that $\tau > T_R$. Then, for $D = e^{-15 \zeta} C$, with probability $1 - \exp \left( -\Omega\left(n^{0.75}\right) \right)$, $Q(T)$ is $Dn$-absorbing.
\end{clm}

\begin{proof}
	Let $(x,y) \in [n]^2$. By assumption, $|\mB_{Q(T_R)}(x,y)| \geq Cn$. For $T_R \leq t \leq T$, let $\mC(t) = \mB_{Q(T_R)}(x,y) \cap \mB_{Q(t)}(x,y)$. We will use a martingale analysis to prove that
	\[
	\Prob \left[ |\mC(T)| < Dn \right] \leq \exp \left( - \Omega(n^{0.75}) \right).
	\]
	Since $|\mB_{Q(T)}(x,y)| \geq |\mC(T)|$ the result then follows from a union bound over the $n^2$ positions in $[n]^2$. We define the random variables $\{C_t\}_{t=T_R}^T$ as follows:
	\[
	C_t = \begin{cases}
	\left| \mC(t) \right| & \tau \geq t\\
	C_{t-1} & \tau < t.
	\end{cases}
	\]
	
	Observe that for every $T_R \leq t < T$ it holds that
	\begin{equation}\label{eq:absorber max one step change}
	|C_{t+1} - C_t| = O(1).
	\end{equation}
	This is because every queen added to a partial configuration can ``destroy'' at most $4$ absorbers for $(x,y)$. We will now show that for every $T_R \leq t < T$:
	\begin{equation}\label{eq:absorber one step change}
	\E \left[ C_{t+1} \given Q(t) \right] \geq \left( 1 - \frac{5\zeta}{ n } \right) C_t.
	\end{equation}
	Indeed, if $\tau \leq t$ then, by definition, $C_{t+1} = C_t$ and \eqref{eq:absorber one step change} holds. If $\tau > t$, then $C_{t+1} = \left| \mC(t+1) \right|$ and $C_t = \left| \mC(t) \right|$. Since $ \mC(t+1) \subseteq \mC(t)$, to prove \eqref{eq:absorber one step change} it suffices to show that for every $(a,b) \in \mC(t)$,
	\[
	\Prob \left[ (a,b) \notin \mC(t+1) \given \tau > t \right] \leq \frac{5\zeta}{n}.
	\]
	The event $(a,b) \notin \mC(t+1)$ occurs only if the queen added at time $t+1$ occupies one of the four diagonals containing $(x,b)$ or $(a,y)$. By Claim \ref{clm:diagonal occupation probability}, since $t<\tau$ for every diagonal the probability that it is occupied at time $t+1$ is $\leq (1+o(1)) \zeta/n$. By a union bound, the probability that one of the four diagonals incident to $(x,b)$ and $(a,y)$ is occupied is $\leq 5\zeta / n$. Thus $\Prob \left[ (a,b) \notin \mC(t+1) \given \tau > t \right] \leq {5\zeta}/{ n}$, as desired.
	
	Equation \eqref{eq:absorber one step change} suggests that $C_T \geq (1-5\zeta/n)^{T-T_R}C_{T_R} \geq Dn$. We will justify this heuristic with a martingale analysis. We first transform $\{C_t\}$ in order to apply Azuma's inequality (Theorem \ref{thm:azuma}). Define $C'_t = \max \{ C_t, D n \}$. It holds that
	\begin{equation}\label{eq:Y expected difference}
	\E \left[ C'_{t+1} \given Q(t) \right] \geq \left( 1 - \frac{5\zeta}{n} \right)C'_t.
	\end{equation}
	Indeed, if $C_t \leq D n$ then $C'_{t+1} = Dn = C'_t \geq \left( 1 - \frac{4 \zeta}{n} \right)C'_t$. Otherwise $C_t > D n$, implying $C'_t = C_t$. In this case
	\begin{align*}
	\E \left[ C'_{t+1} \given Q(t) \right] \geq \E \left[ C_{t+1} \given Q(t) \right] \stackrel{\eqref{eq:absorber one step change}}{\geq} \left( 1 - \frac{5\zeta}{n} \right)C_t = \left( 1 - \frac{5 \zeta}{n} \right)C'_t.
	\end{align*}
	For $T_R \leq t < T$ define $\tC_t \coloneqq C'_{t+1} / C'_t$. We note that $C'_T = C'_{T_R} \times \tC_{T_R} \times \tC_{T_R+1} \times \ldots \times \tC_{T-1}$ and hence
	\[
	\log (C'_T) = \log (C'_{T_R}) + \log (\tC_{T_R}) + \log (\tC_{T_R+1}) + \ldots + \log (\tC_{T-1}).
	\]
	It holds that
	\begin{align*}
	\E \left[ \log (\tC_{t}) \given Q(t-1) \right] & = \E \left[ \log \left( 1 + \frac{C'_{t+1} - C'_t}{C'_t} \right) \given Q(t-1) \right]\\
	& \geq \E \left[ \frac{C'_{t+1}-C'_t}{C'_t} - O \left( \left( \frac{C'_{t+1}-C'_t}{C'_t} \right)^2 \right) \given Q(t-1) \right].
	\end{align*}
	By definition, $C'_t = \Omega (n)$ and by \eqref{eq:absorber max one step change} $\left|C'_{t+1} - C'_t\right| = O(1)$. Therefore:
	\begin{equation}\label{eq:log Z submartingale}
	\E \left[ \log (\tC_{t}) \given Q(t-1) \right] \geq \E \left[ \frac{C'_{t+1}-C'_t}{C'_t} \given Q(t-1) \right] - O \left( \frac{1}{n^2} \right)
	\stackrel{\eqref{eq:Y expected difference}}{\geq} - \frac{5\zeta}{n} - O \left( \frac{1}{n^2} \right) \geq - \frac{6\zeta}{n}.
	\end{equation}
	Finally, we define, for $T_R \leq t \leq T$:
	\[
	\tZ_t = - \left( \log(C'_{T_R}) + \sum_{s=T_R}^{t-1} \left( \log(\tC_s) + \frac{6\zeta}{n} \right) \right).
	\]
	By \eqref{eq:log Z submartingale}, the sequence $\{\tZ_t\}_{t=T_R}^{T}$ is a supermartingale. Additionally, for every $T_R \leq t < T$:
	\begin{align*}
	\left| \tZ_{t+1} - \tZ_t \right| & \leq \left| \log(\tC_{t}) \right| + \frac{5\zeta}{n} = \left| \log \left( 1 + \frac{C'_{t+1}-C'_t}{C'_t} \right) \right| + O\left(\frac{1}{n}\right)\\
	& = \left|\log \left( 1 - O \left( \frac{1}{n} \right) \right)\right| + O\left(\frac{1}{n}\right) = O\left(\frac{1}{n}\right).
	\end{align*}
	Hence, by Theorem \ref{thm:azuma} (the Azuma--Hoeffding inequality) $\Prob \left[ \tZ_{T} > \tZ_{T_R} + 9\zeta  \right] = \exp \left( - \Omega(n) \right)$.	Rewriting the inequality for $e^{-\tZ_T}$ in place of $\tZ_T$ we obtain:
	\begin{equation}\label{eq:exp tZ bound}
	\Prob \left[ e^{-\tZ_{T}} <  \frac{C'_{T_R}}{e^{9\zeta}} \right] = \exp \left( - \Omega(n) \right).
	\end{equation}
	It holds that
	\begin{align*}
	\log(C'_T) = \log(C'_{T_R}) + \sum_{t=T_R}^{T-1} \log(\tC_t) = - \tZ_T - \sum_{t=T_R}^{T-1} \frac{6\zeta}{n} \geq - \tZ_T - 6\zeta.
	\end{align*}
	Thus $C'_T \geq e^{-\tZ_T}\!/e^{6\zeta}$.	Therefore:
	\[
	\Prob \left[ C'_T < \frac{C'_{T_R}}{e^{15\zeta}} \right] \leq \Prob \left[ e^{-\tZ_T} < \frac{C'_{T_R}}{e^{9\zeta}} \right] \stackrel{\eqref{eq:exp tZ bound}}{=} \exp \left( - \Omega(n) \right).
	\]
	Now, the events  $C'_T \geq C'_{T_R}e^{-15\zeta} \geq D n$ and $\tau = \infty$ imply that $|\mC_T| = C'_T \geq Dn$. Therefore,
	\[
	\Prob \left[ |\mC_T| < Dn \right] \leq \Prob [\tau<\infty] + \Prob\left[C'_T < \frac{C'_{T_R}}{e^{15\zeta}} \right] \stackrel{\text{Proposition \ref{prop:tau inf whp}}}{=} \exp \left( - \Omega \left(n^{0.75} \right) \right),
	\]
	proving the claim.
\end{proof}

\subsection{Proof of the lower bound}\label{sec:lower bound proof}

We are ready to prove the lower bound in Theorem \ref{thm:precise bounds}. By \ref{itm:delta close to gamma} we have $\dist{\delta}{\gamma} = O(\varepsilon^2)$. Therefore $B_n(\delta, \varepsilon/2) \subseteq B_n(\gamma,\varepsilon)$. Let $B$ be the set of $n$-queens configurations $q$ such that for every $\alpha \in I_N$ it holds that $|\alpha_n \cap q| = \delta(\alpha)n \pm 2\varepsilon^5 n$ (where $N$ is the constant used to define $\delta$). By Claim \ref{clm:config to queenon approximation} $B \subseteq B_n(\delta, \varepsilon/2)$.

We now show that Algorithm \ref{alg:random} followed by Algorithm \ref{alg:absorption} is likely to produce an element of $B$. Let $\mF$ be the event that
\begin{enumerate}
	\item Algorithm \ref{alg:random} does not abort,
	
	\item\label{itm:linear absorbing} $Q(T)$ is $\Omega(n)$-absorbing, and
	
	\item\label{itm:approximately delta} for every $\alpha \in I_N$, $|Q(T) \cap \alpha_n| = \delta(\alpha) n \pm \varepsilon^5 n$.
\end{enumerate}
By Proposition \ref{prop:tau inf whp} and Claims \ref{clm:Q(T) approximates delta}, \ref{clm:Q(T_R) absorbing}, and \ref{clm:Q(T) absorbing} we have $\Prob [\mF] = 1 - \exp\left( - \Omega \left(n^{0.6}\right) \right)$. Let $B'$ be the set of size-$T$ partial $n$-queens configurations satisfying \ref{itm:linear absorbing} and \ref{itm:approximately delta}. If $\mF$ holds then $(X_1,X_2,\ldots,X_T)$ is an ordered element of $B'$. Thus $H(X_1,X_2,\ldots,X_T \given \mF) \leq \log|B'| + \log(T!)$. We note that since each $X_t$ can take only $O(n^2)$ values, the crude bound $H(X_1,X_2,\ldots,X_T \given \mF^c) \leq n^2$ clearly holds. Hence, by the law of total probability:
\begin{align*}
H(X_1,X_2,\ldots,X_T \given \mF) & = \frac{H(X_1,\ldots,X_T) - H(X_1,X_2,\ldots,X_T \given \mF^c)(1-\Prob[\mF])}{\Prob[\mF]}\\
& \stackrel{\text{Lemma \ref{lem:random phase entropy}}}{=} n(H_q^M(\delta) + 2\log n - 1) \pm 2n^{1-1/K^3}.
\end{align*}
Therefore, recalling that $T = n(1-O(n^{-1/K^2}))$:
\begin{equation}\label{eq:partial configs B' lower bound}
|B'| \geq \frac{1}{T!} \left( \left( 1 - \oone \right) \frac{n^2e^{H_q^M(\delta)}}{e} \right)^n = \left( \left( 1 - \oone \right) ne^{H_q(\delta)} \right)^n,
\end{equation}
where the last equality follows from Stirling's approximation and Lemma \ref{lem:step approximation of entropy}. Let $q' \in B'$. By Lemma \ref{lem:absorbing procedure works} if Algorithm \ref{alg:absorption} is begun from $q'$ the result is an $n$-queens configuration $q$ satisfying $|q \Delta q'| \leq 3(n-T)$. We now show that $q \in B$.

\begin{clm}
	Let $q' \in B'$ and suppose that $q$ is an $n$-queens configuration satisfying $|q \Delta q'| \leq 3(n-T)$. Then $q \in B$.
\end{clm}

\begin{proof}
	We need to show that for every $\alpha \in I_N$ it holds that $|\alpha_n \cap q| = \delta(\alpha) \pm 2\varepsilon^5 n$. Let $\alpha \in I_N$. Then:
	\begin{align*}
	|\alpha_n \cap q| & = |\alpha_n \cap q'| \pm |q \Delta q'| \stackrel{q' \in B'}{=} \delta(\alpha)n \pm \left( \varepsilon^5 n + 3(n-T) \right)\\
	& = \delta(\alpha) n \pm \left( \varepsilon^5 + O \left( n^{-1/K^2} \right) \right)n
	= \delta(\alpha) n \pm 2\varepsilon^5 n,
	\end{align*}
	as desired.
\end{proof}

We now consider the number of ways a given $n$-queens configuration may be obtained as a result of running Algorithm \ref{alg:absorption}.

\begin{clm}
	Let $q$ be an $n$-queens configuration. There are at most $n^{2(n-T)}$ partial configurations $q' \in B'$ such that $q$ can be obtained by beginning Algorithm \ref{alg:absorption} from $q'$.
\end{clm}

\begin{proof}
	At each step of Algorithm \ref{alg:absorption} two queens are added to the board and one is removed. Consider the number of ways to reverse this process, beginning from $q$. At each step we must remove two queens $(a,b)$ and $(c,d)$ from the board and add either $(a,d)$ or $(c,b)$. Since there are always at most $n$ queens on the board there are $\leq \binom{n}{2}$ choices for the queens to remove. There are then at most $2$ choices which queen to add. Since there are $n-T$ steps in Algorithm \ref{alg:absorption} there are at most $\left(\binom{n}{2}2\right)^{n-T} \leq n^{2(n-T)}$ ways to reverse it.
\end{proof}

Since Algorithm \ref{alg:absorption} maps every element of $B'$ to an element of $B$, and every element of $B$ can be obtained in this manner from at most $n^{2(n-T)}$ elements of $B'$ we conclude:
\[
|B_n(\gamma,\varepsilon)| \geq |B| \stackrel{\eqref{eq:partial configs B' lower bound}}{\geq} \frac{|B'|}{n^{2(n-T)}} \geq \left( \left( 1 - \oone \right) ne^{H_q(\delta)} \right)^n.
\]
Therefore $\liminf_{n\to\infty} |B_n(\gamma,\varepsilon)|^{1/n}\! / n \geq e^{H_q(\delta)}$. By \ref{itm:delta approximates entropy} there holds $H_q(\delta) > H_q(\gamma) - \varepsilon |H_q(\gamma)|$. This proves the lower bound in Theorem \ref{thm:precise bounds}.

\section{Explicit bounds on $H_q(\gamma^*)$}\label{sec:optimizing H_q}

\begin{rmk}
	This section relies on numerical calculations. Code verifying the calculations can be obtained by downloading the source of the arXiv submission at \url{\codeURL}.
\end{rmk}

In this section we prove Claim \ref{clm:explicit bounds on H_q}. We begin with the lower bound.

\begin{clm}
	$H_q(\gamma^*) \geq -\lowerBoundConstant$.
\end{clm}

\begin{proof}
	It suffices to exhibit an explicit queenon $\gamma$ such that $H_q(\gamma) \geq -\lowerBoundConstant$. Define the $12 \times 12$ matrix
	\[\small
	A = \frac{1}{100}
	\begin{bmatrix}
	59 & 76 & 95 & 113 & 125 & 132 & 132 & 125 & 113 & 95 & 76 & 59\\
	76 & 87 & 99 & 108 & 114 & 116 & 116 & 114 & 108 & 99 & 87 & 76\\
	95 & 99 & 100 & 102 & 102 & 102 & 102 & 102 & 102 & 100 & 99 & 95\\
	113 & 108 & 102 & 94 & 92 & 91 & 91 & 92 & 94 & 102 & 108 & 113\\
	125 & 114 & 102 & 92 & 85 & 82 & 82 & 85 & 92 & 102 & 114 & 125\\
	132 & 116 & 102 & 91 & 82 & 77 & 77 & 82 & 91 & 102 & 116 & 132\\
	132 & 116 & 102 & 91 & 82 & 77 & 77 & 82 & 91 & 102 & 116 & 132\\
	125 & 114 & 102 & 92 & 85 & 82 & 82 & 85 & 92 & 102 & 114 & 125\\
	113 & 108 & 102 & 94 & 92 & 91 & 91 & 92 & 94 & 102 & 108 & 113\\
	95 & 99 & 100 & 102 & 102 & 102 & 102 & 102 & 102 & 100 & 99 & 95\\
	76 & 87 & 99 & 108 & 114 & 116 & 116 & 114 & 108 & 99 & 87 & 76\\
	59 & 76 & 95 & 113 & 125 & 132 & 132 & 125 & 113 & 95 & 76 & 59
	\end{bmatrix}.
	\]
	The sum of each row and column of $A$ is $12$. Additionally, every diagonal in $A$ has sum $\leq 12$. (These assertions can be verified with the provided code.) Hence, by \cref{clm:step queenon sufficient condition} the measure $\gamma$ whose density function has constant value $A_{i,j}$ on each square $\sigma_{i,j}^{12}$ for $i,j \in [12]$ is a $12$-step queenon. It remains to verify that $H_q(\gamma) > -\lowerBoundConstant$, for which the reader is invited to use the provided code.
\end{proof}

\begin{figure}
	\includegraphics[width=\textwidth*1/3]{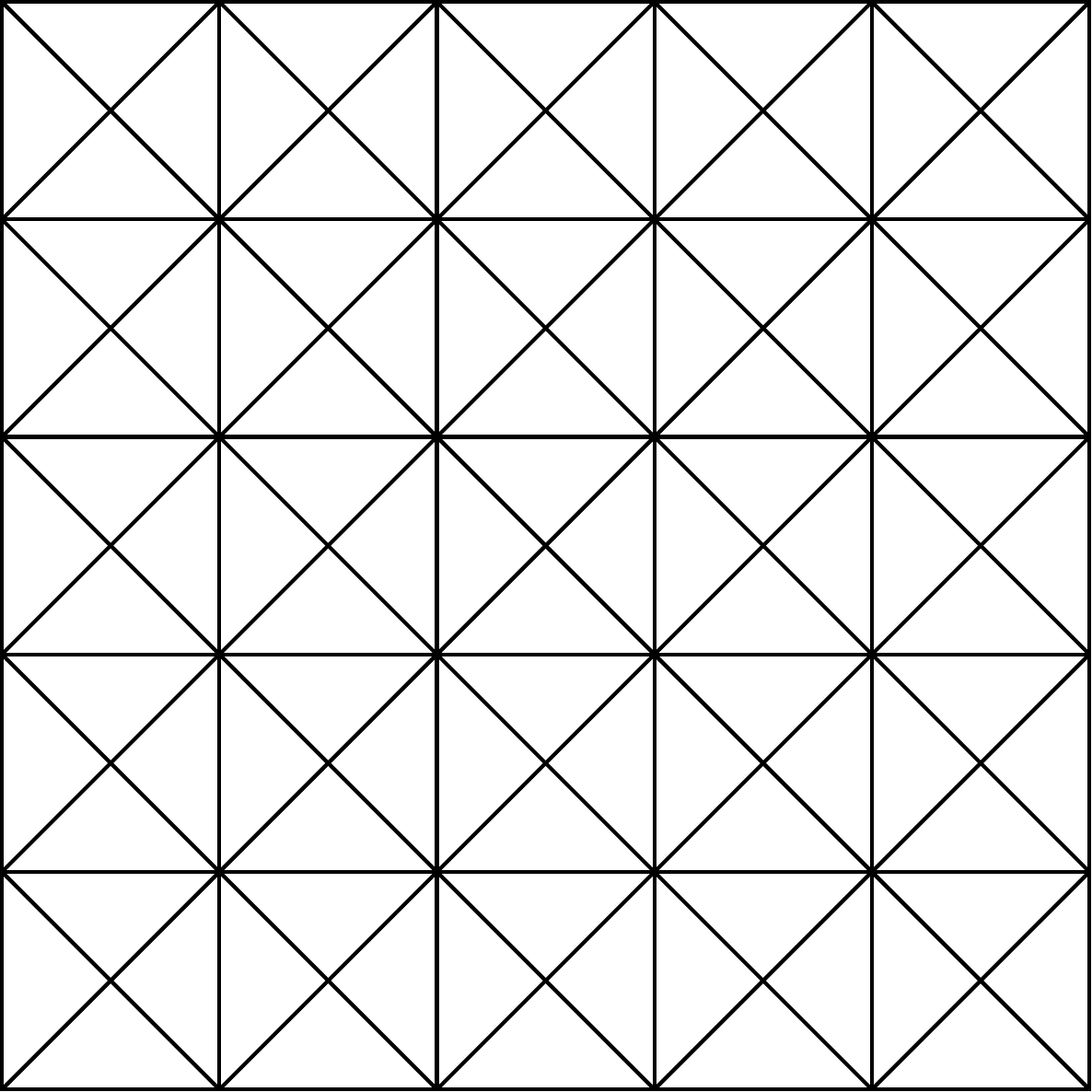}
	\caption{The division of $\interval^2$ into $K_N$, for $N=5$.}\label{fig:K_N}
	\centering
\end{figure}

We turn to the upper bound. One difficulty in bounding $H_q$ from above is that its domain is infinite-dimensional. Hence we seek a finite-dimensional approximation of $H_q$ that bounds it from above. We take the following approach: Let $N \in \N$. Let $K = K_N$ be the minimal mutual refinement of $I_N$ and $\{\sigma_{i,j}^N\}_{i,j\in[N]}$ (see Figure \ref{fig:K_N}). For $\gamma \in \Queenons$ let $\gamma_N$ be the measure on $\interval^2$ that has constant density on every $\alpha \in K$ and satisfies $\gamma_N(\alpha) = \gamma(\alpha)$. We do not claim that $\gamma_N$ is necessarily a queenon or even a permuton. However, it is the case that for every $i \in [N]$ there holds
\begin{equation*}
\sum_{j=1}^N \gamma_N(\sigma_{i,j}^N) = \sum_{j=1}^N \gamma_N(\sigma_{j,i}^N) = \frac{1}{N}
\end{equation*}
(i.e., if we partition $\interval^2$ into axis-parallel strips of width $1/N$ then $\gamma_N$ induces the uniform distribution both vertically and horizontally). Additionally, for every $\alpha \in J_N$ there holds:
\begin{equation*}
\gamma_N^+(\alpha) = \gamma^+(\alpha) \leq 1/N, \quad \gamma_N^-(\alpha) = \gamma^-(\alpha) \leq 1/N.
\end{equation*}
Thus, we may define the distributions $\distPlus{\gamma_N}$ and $\distMinus{\gamma_N}$ on $J_N$ in the natural way by setting $\distPlus{\gamma_N}(\alpha) = 1/N - \gamma_N^+(\alpha)$ and $\distMinus{\gamma_N}(\alpha) = 1/N - \gamma_N^-(\alpha)$. By concavity of the function $-x\log x$ and the definition of KL divergence, for every $\gamma \in \Queenons$:
\[
-D_{KL} (\gamma) \leq -D_{KL}(\gamma_N) = - \sum_{\alpha \in K} \gamma_N(\alpha) \log \left(\gamma_N(\alpha)\right) - 2\log(2N)
\]
and for $* \in \{+,-\}$:
\[
-D_{KL}(\overline{\gamma}^*) \leq -D( \overline{\gamma_N}^*) = - \sum_{\alpha \in J_N} \overline{\gamma_N}^*(\alpha) \log \left( \overline{\gamma_N}^*(\alpha) \right) - \log(2N).
\]

We now reformulate the problem as entropy maximization. This will allow us to bound $H_q(\gamma^*)$ using the Lagrangian dual function. Let $J_N^1$ and $J_N^2$ be two disjoint copies of $J_N$. Let $\Omega = \Omega_N \coloneqq K_N \cup J_N^1 \cup J_N^2$. Let $\mD = \mD_N \coloneqq (0,1/N)^{\Omega_N}$. Define the function $f = f_N : \mD_N \to \R$ by:
\[
f(x) = - \sum_{\alpha \in \Omega} x_\alpha \log(x_\alpha) - 4\log(2N) + 2\log(2) - 3.
\]
As observed, by concavity, $f(\gamma_N,\distPlus{\gamma_N},\distMinus{\gamma_N}) \geq H_q(\gamma)$ for every $\gamma \in \Queenons$.

To facilitate matrix notation we fix an identification of $\Omega$ with $[|\Omega|]$. For $x \in \mD$ we write $x = (\gamma_x,\gamma_{1,x},\gamma_{2,x})$ when we wish to access the three measures that comprise $x$.  Let $A$ be the $6N \times |\Omega|$ matrix and $b \in \R^{6N}$ such that $Ax = b$ if and only if $x \in \mD$ satisfies the linear equations
\begin{equation}\label{eq:A defining}
\begin{split}
&\forall i \in [N], \sum_{j=1}^N \gamma_x(\sigma_{i,j}^N) = \sum_{j=1}^N \gamma_x(\sigma_{j,i}^N) = \frac{1}{N},\\
&\forall \alpha \in J_N^1, \gamma_{1,x}(\alpha) = \frac{1}{N} - \gamma_x^+(\alpha),\\
&\forall \alpha \in J_N^2, \gamma_{2,x}(\alpha) = \frac{1}{N} - \gamma_x^-(\alpha).
\end{split}
\end{equation}
Note that for every $\gamma \in \Queenons$, $(\gamma_N,\distPlus{\gamma_N},\distMinus{\gamma_N})$ satisfies \eqref{eq:A defining}. Therefore the following concave optimization problem bounds $H_q(\gamma^*)$ from above:
\begin{equation*}
\begin{aligned}
& \underset{x\in\mD}{\text{maximize}}
& & f(x) \\
& \text{subject to:}
& & Ax = b.
\end{aligned}
\end{equation*}

We define the Lagrangian dual function $\mathcal{L} : \R^{6N} \to \R$ by:
\[
\mathcal{L}(y) = \sup_{x \in \mD} \left(f(x) + y^T(Ax-b) \right).
\]
Then, by definition, for every $y \in \R^{6N}$ there holds $\mathcal{L}(y) \geq f(\gamma^*_N, \distPlus{\gamma^*_N}, \distMinus{\gamma^*_-}) \geq H_q(\gamma^*)$.

The next claim provides an explicit form for $\mathcal{L}(y)$ for a large range of $y$. We denote the length-$4N^2$ all $1$s row vector by $\mathbbm{1}_{4N^2}$.

\begin{clm}\label{clm:explicit L_N}
	Let $N \in \N$ and let $y \in \R^{6N}$ satisfy $y^TA - \mathbbm{1}_{4N^2} \leq -\log(N) \mathbbm{1}_{4N^2}$. Define $\tilde{x} \in \mD$ by $\tilde{x}_\alpha = \exp ((y^TA)_\alpha - 1)$. Then $\mathcal{L}(y) = f(\tilde{x}) + y^T(A\tilde{x} - b)$.
\end{clm}

\begin{proof}
	We observe that if $y$ is fixed then $g(x) \coloneqq f(x) + y^T(Ax - b)$ is strictly concave on $\mD$. Therefore it suffices to show that $\nabla g(\tilde{x}) = 0$. Indeed, for every $\alpha \in \Omega$ we have $\frac{\partial g}{\partial x_\alpha}(x) = -\log(x_\alpha) - 1 + (y^TA)_\alpha$. By definition of $\tilde{x}$ this is zero when $x = \tilde{x}$.
\end{proof}

We are ready to prove the upper bound.

\begin{clm}
	$H_q(\gamma^*) < -\upperBoundConstant$.
\end{clm}

\begin{proof}
	Set $N=17$. It suffices to exhibit some $y \in \R^{6N}$ satisfying $\mathcal{L}(y) < -\upperBoundConstant$. For this we rely on Claim \ref{clm:explicit L_N} and computer calculation. The provided code contains a function that, given $y \in \R^{6N}$ satisfying the conditions of Claim \ref{clm:explicit L_N}, calculates $\mathcal{L}(y)$. The same file also contains an explicit vector $y$ satisfying these conditions and verifies that for this $y$ there holds $\mathcal{L}(y) < -\upperBoundConstant$.
\end{proof}

\section{Concluding remarks}\label{sec:conclusion}

\begin{itemize}
	\item This paper combined the entropy method and a randomized algorithm to determine the first and second order terms of $\log(\mQ(n))$. We wonder whether similar methods might succeed in obtaining more accurate estimates. More generally, for many classes of combinatorial designs (such as Steiner systems \cite{linial2013upper,keevash2018counting} and high-dimensional permutations \cite{linial2014upper,keevash2018existence}), denoting by $X(n)$ the number of order-$n$ objects, the first and second order terms of $\log(X(n))$ have been determined using similar methods. It would be very interesting to improve these estimates.
	
	\item Since the first version of this paper was published Nobel, Agrawal, and Boyd \cite{nobel2022computing} used sophisticated computational techniques to improve the bounds on the constant $\alpha$ in \cref{thm:main theorem first statement}. They showed that $\alpha \in [1.944000752, 1.944001082]$.
	
	\item The lower bounds for the number of Steiner systems and high-dimensional permutations were obtained using a random greedy algorithm to construct an approximate design. In contrast, the lower bound in our paper uses a more sophisticated algorithm. There is, of course, a very natural random greedy algorithm for the $n$-queens problem: beginning with an empty board, in each step add a queen to a position chosen uniformly at random from the available positions. As mentioned in the introduction, the asymmetry of the constraints makes this algorithm challenging to analyze. However, based on simulations, it is clear that this algorithm succeeds in placing almost $n$ queens and, furthermore, Algorithm \ref{alg:absorption} successfully completes the outcome. It is therefore worth asking if the lower bound on $\mQ(n)$ could conceivably be proved by a successful analysis of this algorithm. We believe this is not the case: Empirically, the outcomes of the random greedy algorithm do not approximate $\gamma^*$. This implies that they are contained in an atypical (and hence small) subset of the configurations.
	
	\item Let $X_n$ be the random variable equal to the number of pairs of $1$s sharing a diagonal in a uniformly random order-$n$ permutation matrix. This paper can be interpreted as studying $\Prob [X_n = 0]$. It would be interesting to understand the tails of $X_n$ more generally. For certain permutation parameters (most prominently ``pattern density'' \cite{kenyon2020permutations}) large deviations can be understood with the theory of permutons. However, $X_n$ is not continuous in the permuton topology. This suggests additional tools must be developed.
	
	\item The $n$-queens problem has many variations. Perhaps the best-known is the \textit{toroidal} or \textit{modular} problem, in which the diagonals wrap around the board. Let $\mT (n)$ be the number of toroidal $n$-queens configurations. P\'olya proved that $\mT (n) > 0$ if and only if $\gcd(n,6) = 1$ \cite{polya1921uber}. Using the entropy method, Luria showed that $\mT (n) \leq \left( (1+\oone) n/e^3 \right)^n$ \cite{luria2017new}. It is not difficult to show that the natural random greedy algorithm for constructing a toroidal $n$-queens configuration w.h.p.\ succeeds in placing $n-o(n)$ queens on the board (indeed, this is the source of the lower bound on $\mQ(n)$ in \cite{luria2021lower}). Furthermore, if the outcome of the process can w.h.p.\ be completed this would imply that $\mT (n) \geq \left( (1-\oone)n/e^3 \right)^n$. Unfortunately, the absorption method in this paper utilizes the fact that in a complete non-toroidal configuration only a fraction of the diagonals are occupied. This is not the case in the toroidal problem, indicating that a more complex absorption procedure is necessary. Indeed, subsequently to the publication of the first version of this paper, Bowtell and Keevash published a spectacular proof \cite{bowtell2021n} of this lower bound on $\mT (n)$. Their method combines randomized algebraic construction \cite{keevash2014existence} and iterative absorption \cite{glock2016existence}.
	
	\item It is also worth mentioning the semi-queens variant, in which queens attack along rows, columns, and plus-diagonals (but not minus-diagonals). It is plausible that the techniques in this paper can be applied to enumerate these configurations. Remarkably, an \textit{asymptotic} formula for the number of \textit{toroidal} semi-queens configurations was found using tools from analytic number theory \cite{eberhard2018additive}. We wonder if our bounds on $\mQ (n)$, or those of Bowtell and Keevash on $\mT (n)$, can be strengthened using this toolbox.
\end{itemize}

\section*{Acknowledgments}

I thank Zur Luria for bringing the problem to my attention and for fascinating discussions. I also wish to thank Donald Knuth and two anonymous referees for many suggestions that helped improve the manuscript. In particular, Don and one referee pointed out the necessity of \cref{lem:approximation}. The second referee found an error in \cref{sec:lower bound}, that was corrected by explicitly tracking $A_\alpha(t)$.

\appendix

\section{Notation table}\label{app:notation}
\renewcommand{\arraystretch}{1.3}

In this section we collect some of the notation used in the paper. We focus on those symbols that have a ``global'' scope, and omit some symbols that are used only within a particular section.

\subsection*{Sets and set collections}

We define the following sets and set collections. In the following, $N$ denotes a natural number.

\begin{tabular}{m{.03\textwidth}|p{.7\textwidth}}
	\hline
	$\mR$ & The collection of subsets of the plane with the form
	\[
	\{(x,y) \in \interval^2 : a_1 \leq x+y \leq b_1, a_2 \leq y-x \leq b_2\}
	\]
	for $a_1,a_2,b_1,b_2 \in [-1,1]$.\\
	\hline
	$\sigma_{i,j}^N$ & For $i,j \in [N]$ this is the square
	\[
	((i-1)/N - 1/2,i/N - 1/2) \times ((j-1)/N - 1/2,j/N - 1/2).
	\]\\
	\hline
	$I_N$ & The division of $\interval^2$ into sets of the form
	\[
	\left\{ (x,y) \in \interval^2 : - \frac{1}{N} \leq x+y + 1 - \frac{i}{N} \leq 0, - \frac{1}{N} \leq y-x + 1 - \frac{j}{N} \leq 0 \right\}
	\]
	for $i,j \in [2N]$ (see \cref{fig:lattice}).\\
	\hline
	$S_N$ & The collection of squares in $I_N$.\\
	\hline
	$T_N$ & The collection of half-squares in $I_N$.\\
	\hline
	$J_N$ & The division of $[-1,1]$ into intervals of the form ${[-1+(i-1)/N,-1+i/N]}$ for $i \in [2N]$.\\
	\hline
\end{tabular}

For $n,N \in \N$ and $\alpha \in I_N$ we define the following quantities and sets.

\begin{tabular}{m{.1\textwidth}|p{.8\textwidth}}
	\hline
	$|\alpha|$ & The area of $\alpha$ (which is $1/(2N^2)$ when $\alpha \in S_N$ and $1/(4N^2)$ when $\alpha \in T_N$).\\
	\hline
	$\alpha_n$ & The set of positions $(i,j) \in [n]^2$ such that among all $\beta \in I_N$ that satisfy $\beta \cap \sigma_{i,j}^n \neq \emptyset$, the one with the center-point that is minimal in the lexicographic order is $\alpha$ (see \cref{fig:lattice}).\\
	\hline
\end{tabular}

The following symbols are defined for all $n,N \in \N$, $x,y \in [n]$, and $\alpha \in I_N$. Whenever they are used $n$ is clear from context.

\begin{tabular}{m{.1\textwidth}|p{.8\textwidth}}
	\hline
	$L_{y,\alpha}^r$ & The number of positions in $\alpha_n$ and row $y$.\\
	\hline
	$L_{x,\alpha}^c$ & The number of positions in $\alpha_n$ and column $x$.\\
	\hline
	$L_{x+y,\alpha}^+$ & The number of positions in $\alpha_n$ and plus-diagonal $x+y$.\\
	\hline
	$L_{y-x,\alpha}^-$ & The number of positions in $\alpha_n$ and minus-diagonal $y-x$.\\
	\hline
	$\alpha^N(x,y)$ & The element $\alpha \in I_N$ such that $(x,y) \in \alpha_n$.\\
	\hline
\end{tabular}

\subsection*{Measures and queenons}

The following definitions relate to the space of queenons and, more generally, the space of Borel probability distributions on $\interval^2$.

\begin{tabular}{m{.1\textwidth}|p{.8\textwidth}}
	\hline
	$\mP$ & The set of Borel probability measures on $\interval^2$.\\
	\hline
	$\mU_\square$ & The uniform distribution on $\interval^2$.\\
	\hline
	$\mU_{[-1,1]}$ & The uniform distribution on $[-1,1]$.\\
	\hline
	$\dist{\cdot}{\cdot}$ & The metric on $\mP$ given by $\dist{\gamma_1}{\gamma_2} = \sup \{ |\gamma_1(\alpha)-\gamma_2(\alpha)| : \alpha \in \mR \}$ for $\gamma_1,\gamma_2 \in \mP$.\\
	\hline
	$\tilde{\Queenons}$ & The set of step queenons, as defined in \cref{def:queenons}.\\
	\hline
	$\Queenons$ & The set of queenons, as defined in \cref{def:queenons}\\
	\hline
	$\gamma_q$ & For an $n$-queens configuration $q \subseteq [n]^2$ this is the queenon with constant density $n$ on $\sigma_{i,j}^n$ for every $(i,j) \in q$ and density $0$ elsewhere.\\
	\hline
	$B_n(\gamma,\varepsilon)$ & For $n \in \N$, $\gamma \in \Queenons$, and $\varepsilon > 0$, this is the set of $n$-queens configurations $q$ such that $\dist{\gamma_q}{\gamma} < \varepsilon$.\\
	\hline
\end{tabular}

For a probability measure $\gamma \in \mP$, $N \in \N$, and $\alpha \in I_N$ we define the following measures and quantities (see \cref{def:interval measures}). In the next table only, $\lambda$ denotes the restriction to Borel sets of the Lebesgue measure on $[-1,1]$.

\begin{tabular}{m{.1\textwidth}|p{.8\textwidth}}
	\hline
	$\gamma^+$ & The pushforward of $\gamma$ under the map $(x,y) \mapsto x+y$.\\
	\hline
	$\gamma^-$ & The pushforward of $\gamma$ under the map $(x,y) \mapsto y-x$.\\
	\hline
	$\distPlus{\gamma}$ & $\lambda - \gamma^+$. Defined only if $\gamma$ has sub-uniform diagonal marginals.\\
	\hline
	$\distMinus{\gamma}$ & $\lambda - \gamma^-$. Defined only if $\gamma$ has sub-uniform diagonal marginals.\\
	\hline
	$\gamma^+(\alpha)$ & The quantity $\gamma^+(\beta)$ where $\beta$ is the (unique) interval in $J_N$ such that $\gamma(\alpha)$ contributes to $\gamma^+(\beta)$.\\
	\hline
	$\gamma^-(\alpha)$ & The quantity $\gamma^-(\beta)$ where $\beta$ is the (unique) interval in $J_N$ such that $\gamma(\alpha)$ contributes to $\gamma^-(\beta)$.\\
	\hline
	$\distPlus{\gamma}(\alpha)$ & The quantity $\distPlus{\gamma}(\beta)$ where $\beta$ is the (unique) interval in $J_N$ such that $\gamma(\alpha)$ contributes to $\gamma^+(\beta)$.\\
	\hline
	$\distMinus{\gamma}(\alpha)$ & The quantity $\distMinus{\gamma}(\beta)$ where $\beta$ is the (unique) interval in $J_N$ such that $\gamma(\alpha)$ contributes to $\gamma^-(\beta)$.\\
	\hline
\end{tabular}

\subsection*{Entropy functions}

Let $\mu$ be a Borel probability measure defined on either $\interval^2$ or $[-1,1]$. Let $\gamma \in \Queenons$ be a queenon. Let $(p_1,\ldots,p_n)$ be a finite probability distribution. Let $X$ and $Y$ be random variables defined on the same probability space. Let $N \in \N$. We define the following entropy and divergence functions.

\begin{tabular}{m{.15\textwidth}|p{.8\textwidth}}
	\hline
	$H(X)$ & The entropy of $X$.\\
	\hline
	$H(X|Y)$ & The conditional entropy of $X$ given $Y$.\\
	\hline
	$D_{KL}(\mu)$ & The Kullback--Leibler (KL) divergence of $\mu$ with respect to the uniform distribution.\\
	\hline
	$H_q(\gamma)$ & The Q-entropy of $\gamma$ (see \cref{def:Q entropy}).\\
	\hline
	$D(\{p_i\}_{i=1,\ldots,n})$ & The KL divergence of $\{p_i\}_{i=1,\ldots,n}$ with respect to the uniform distribution, i.e., $\sum_{i=1}^np_i\log(np_i)$.\\
	\hline
	$D^N(\gamma)$ & $\sum_{ \alpha \in I_N} \gamma(\alpha) \log(\gamma(\alpha) / |\alpha|)$.\\
	\hline
	$H_q^N(\gamma)$ & A discrete approximation of $H_q(\gamma)$ (see \cref{def:discrete Q entropy}).\\
	\hline
\end{tabular}

\bibliography{variational_queens}

\providecommand{\bysame}{\leavevmode\hbox to3em{\hrulefill}\thinspace}
\providecommand{\MR}{\relax\ifhmode\unskip\space\fi MR }
\providecommand{\MRhref}[2]{%
  \href{http://www.ams.org/mathscinet-getitem?mr=#1}{#2}
}
\providecommand{\href}[2]{#2}
\begin{thebibliography}{10}

\bibitem{bell2009survey}
Jordan Bell and Brett Stevens, \emph{A survey of known results and research
  areas for n-queens}, Discrete Mathematics \textbf{309} (2009), no.~1, 1--31.

\bibitem{beneliezer_et_al:LIPIcs.ITCS.2021.42}
Omri Ben-Eliezer, Eldar Fischer, Amit Levi, and Yuichi Yoshida, \emph{{Ordered
  Graph Limits and Their Applications}}, 12th Innovations in Theoretical
  Computer Science Conference (ITCS 2021) (Dagstuhl, Germany) (James~R. Lee,
  ed.), Leibniz International Proceedings in Informatics (LIPIcs), vol. 185,
  Schloss Dagstuhl--Leibniz-Zentrum f{\"u}r Informatik, 2021, pp.~42:1--42:20.

\bibitem{bogachev2007measure}
Vladimir~I Bogachev, \emph{Measure theory}, vol.~1, Springer Science \&
  Business Media, 2007.

\bibitem{bowtell2021n}
Candida Bowtell and Peter Keevash, \emph{The $ n $-queens problem}, arXiv
  preprint arXiv:2109.08083 (2021).

\bibitem{cooper2020quasirandom}
Jacob~W Cooper, Daniel Kr{\'a}l', Ander Lamaison, and Samuel Mohr,
  \emph{Quasirandom {Latin} squares}, Random Structures \& Algorithms
  \textbf{61} (2022), no.~2, 298--308.

\bibitem{eberhard2018additive}
Sean Eberhard, Freddie Manners, and Rudi Mrazovi{\'c}, \emph{Additive triples
  of bijections, or the toroidal semiqueens problem}, Journal of the European
  Mathematical Society \textbf{21} (2018), no.~2, 441--463.

\bibitem{garbe2020limits}
Frederik Garbe, Robert Hancock, Jan Hladk{\`y}, and Maryam Sharifzadeh,
  \emph{Limits of {Latin} squares}, arXiv preprint arXiv:2010.07854 (2020).

\bibitem{glebov2015finitely}
Roman Glebov, Andrzej Grzesik, Tereza Klimo{\v{s}}ov{\'a}, and Daniel Kr\'al',
  \emph{Finitely forcible graphons and permutons}, Journal of Combinatorial
  Theory, Series B \textbf{110} (2015), 112--135.

\bibitem{glock2016existence}
Stefan Glock, Daniela K{\"u}hn, Allan Lo, and Deryk Osthus, \emph{The existence
  of designs via iterative absorption: hypergraph {$F$}-designs for arbitrary
  {$F$}}, Memoirs of the American Mathematical Society (to appear).

\bibitem{hoppen2013limits}
Carlos Hoppen, Yoshiharu Kohayakawa, Carlos~Gustavo Moreira, Bal{\'a}zs
  R{\'a}th, and Rudini~Menezes Sampaio, \emph{Limits of permutation sequences},
  Journal of Combinatorial Theory, Series B \textbf{103} (2013), no.~1,
  93--113.

\bibitem{keevash2014existence}
Peter Keevash, \emph{The existence of designs}, arXiv preprint arXiv:1401.3665
  (2014).

\bibitem{keevash2018counting}
\bysame, \emph{Counting designs}, Journal of the European Mathematical Society
  \textbf{20} (2018), no.~4, 903--927.

\bibitem{keevash2018existence}
\bysame, \emph{The existence of designs {II}}, arXiv preprint arXiv:1802.05900
  (2018).

\bibitem{kenyon2020permutations}
Richard Kenyon, Daniel Kr{\'a}l', Charles Radin, and Peter Winkler,
  \emph{Permutations with fixed pattern densities}, Random Structures \&
  Algorithms \textbf{56} (2020), no.~1, 220--250.

\bibitem{kral2013quasirandom}
Daniel Kr\'al' and Oleg Pikhurko, \emph{Quasirandom permutations are
  characterized by 4-point densities}, Geometric and Functional Analysis
  \textbf{23} (2013), no.~2, 570--579.

\bibitem{linial2013upper}
Nathan Linial and Zur Luria, \emph{An upper bound on the number of {Steiner}
  triple systems}, Random Structures \& Algorithms \textbf{43} (2013), no.~4,
  399--406.

\bibitem{linial2014upper}
\bysame, \emph{An upper bound on the number of high-dimensional permutations},
  Combinatorica \textbf{34} (2014), no.~4, 471--486.

\bibitem{lovasz2012large}
L{\'a}szl{\'o} Lov{\'a}sz, \emph{Large networks and graph limits}, vol.~60,
  American Mathematical Soc., 2012.

\bibitem{luria2017new}
Zur Luria, \emph{New bounds on the number of n-queens configurations}, arXiv
  preprint arXiv:1705.05225 (2017).

\bibitem{luria2021lower}
Zur Luria and Michael Simkin, \emph{A lower bound for the n-queens problem},
  Proceedings of the 2022 Annual ACM-SIAM Symposium on Discrete Algorithms
  (SODA), SIAM, 2022, pp.~2185--2197.

\bibitem{mcdiarmid1998concentration}
Colin McDiarmid, \emph{Concentration}, 1998, pp.~195--248. \MR{1678578}

\bibitem{nobel2022computing}
Parth Nobel, Akshay Agrawal, and Stephen Boyd, \emph{Computing tighter bounds
  on the n-queens constant via {N}ewton’s method}, Optimization Letters
  (2022), 1--12.

\bibitem{parthasarathy2005probability}
Kalyanapuram~Rangachari Parthasarathy, \emph{Probability measures on metric
  spaces}, vol. 352, American Mathematical Soc., 2005.

\bibitem{polya1921uber}
George P{\'o}lya, \emph{Uber die “doppelt-periodischen” l{\"o}sungen des
  n-damen-problems}, W. Ahrens, Mathematische Unterhaltungen und Spiele
  \textbf{1} (1921), 364--374.

\bibitem{radhakrishnan1997entropy}
Jaikumar Radhakrishnan, \emph{An entropy proof of {Bregman's} theorem}, Journal
  of combinatorial theory, Series A \textbf{77} (1997), no.~1, 161--164.

\bibitem{rivin1994n}
Igor Rivin, Ilan Vardi, and Paul Zimmermann, \emph{The n-queens problem}, The
  American Mathematical Monthly \textbf{101} (1994), no.~7, 629--639.

\bibitem{rodl1985packing}
Vojt{\v{e}}ch R{\"o}dl, \emph{On a packing and covering problem}, European
  Journal of Combinatorics \textbf{6} (1985), no.~1, 69--78.

\bibitem{rudin1966real}
Walter Rudin, \emph{Real and complex analysis}, McGraw-Hill Book Co., New
  York-Toronto, Ont.-London, 1966. \MR{0210528}

\bibitem{oeis}
Neil J.~A. Sloane and The OEIS~Foundation Inc., \emph{The on-line encyclopedia
  of integer sequences}, 2020.

\bibitem{spencer1995asymptotic}
Joel Spencer, \emph{Asymptotic packing via a branching process}, Random
  Structures \& Algorithms \textbf{7} (1995), no.~2, 167--172.

\bibitem{warnke2016method}
Lutz Warnke, \emph{On the method of typical bounded differences},
  Combinatorics, Probability and Computing \textbf{25} (2016), no.~2, 269--299.

\bibitem{wormald1995differential}
Nicholas~C Wormald, \emph{Differential equations for random processes and
  random graphs}, The annals of applied probability \textbf{5} (1995), no.~4,
  1217--1235.

\bibitem{wormald1999differential}
\bysame, \emph{The differential equation method for random graph processes and
  greedy algorithms}, Lectures on approximation and randomized algorithms
  \textbf{73} (1999), 155.

\bibitem{zhang2009counting}
Cheng Zhang and Jianpeng Ma, \emph{Counting solutions for the {$N$}-queens and
  {Latin}-square problems by {Monte} {Carlo} simulations}, Physical Review E
  \textbf{79} (2009), no.~1, 016703.

\end{thebibliography}
\bibliographystyle{amsplain}

\end{document}